%% file: clsqfma.tex
\documentclass[a4paper,12pt]{amsart}

\usepackage{amsmath,amssymb,amsfonts,mathrsfs,amsthm,stmaryrd,mathtools,enumerate, enumitem}
\usepackage[all]{xy}
\usepackage{url}
\usepackage[colorlinks=true,linkcolor=blue,urlcolor=magenta,citecolor=blue]{hyperref}
\usepackage[nobysame,alphabetic,initials,msc-links]{amsrefs}
\usepackage{cleveref}

\DefineSimpleKey{bib}{arxiv}
\newcommand\myurl[1]{\url{#1}}

\def\MR#1{\quad \href{http://www.ams.org/mathscinet-getitem?mr=#1}{MR#1}}

\DefineSimpleKey{bib}{how}
\renewcommand{\eprint}[1]{#1}
\BibSpec{misc}{%
  +{}{\PrintAuthors}  {author}
  +{,}{ \textit}      {title}
  +{,}{ }             {how}
  +{}{ \parenthesize} {date}
  +{,} { available at \eprint}        {eprint}
  +{,}{ available at \url}{url}
  +{,}{ }             {note}
  +{.}{}              {transition}
}

\newtheorem{thrm}{Theorem}[section]
\newtheorem{prop}[thrm]{Proposition}
\newtheorem{coro}[thrm]{Corollary}
\newtheorem{lemm}[thrm]{Lemma}

\theoremstyle{definition}
\newtheorem{defn}[thrm]{Definition}
\newtheorem{exam}[thrm]{Example}

\newtheorem{rema}[thrm]{Remark}

\newtheorem*{defn*}{Definition}
\newtheorem*{thrm*}{Theorem}

\renewcommand{\bar}[1]{\overline{#1}}
\renewcommand{\tilde}[1]{\widetilde{#1}}
\renewcommand{\epsilon}{\varepsilon}

\newcommand{\aE}{\acute{E}}

\newcommand{\aF}{\acute{F}}

\newcommand{\ah}{\acute{\h}}
\newcommand{\aK}{\acute{K}}

\newcommand{\Z}{\mathbb{Z}}
\newcommand{\Q}{\mathbb{Q}}
\newcommand{\R}{\mathbb{R}}
\newcommand{\C}{\mathbb{C}}

\renewcommand{\k}{\mathfrak{k}}
\newcommand{\frb}{\mathfrak{b}}
\newcommand{\g}{\mathfrak{g}}

\newcommand{\h}{\mathfrak{h}}
\newcommand{\frl}{\mathfrak{l}}

\newcommand{\m}{\mathfrak{m}}
\newcommand{\n}{\mathfrak{n}}
\renewcommand{\P}{\mathbb{P}}

\newcommand{\frp}{\mathfrak{p}}

\renewcommand{\t}{\mathfrak{t}}
\newcommand{\fru}{\mathfrak{u}}

\newcommand{\SU}{\mathrm{SU}}
\newcommand{\fsu}{\mathfrak{su}}
\newcommand{\SL}{\mathrm{SL}}
\newcommand{\fsl}{\mathfrak{sl}}
\newcommand{\lsbg}{L_S\backslash G}

\newcommand{\eps}{\varepsilon}
\renewcommand{\phi}{\varphi}
\newcommand{\lbd}{\lambda}
\newcommand{\Lbd}{\Lambda}
\newcommand{\std}{\mathrm{std}}
\newcommand{\zero}{\mathrm{zero}}
\newcommand{\catO}[1]{\cl{O}_{q,#1}}

\newcommand{\intO}[1]{\cl{O}_{q,#1}^{\mathrm{int}}}
\newcommand{\pintO}[1]{\cl{O}_{q,#1}^{\frp_S,\mathrm{int}}}

\newcommand{\unipintO}[1]{\mathrm{C}^*\cl{O}_{q,#1}^{\frp_S,\mathrm{int}}}

\newcommand{\fssorb}{H\backslash G}
\newcommand{\ssorb}{L_S\backslash G}

\newcommand{\ssorba}{\SL_I\backslash \SL_n}

\newcommand{\sflag}{K_S\backslash K}

\newcommand{\sflaga}{\SU(I) \backslash \SU(n)}

\newcommand{\map}[3]{#1\colon#2 \longrightarrow #3}

\newcommand{\abs}[1]{\lvert #1 \rvert}
\newcommand{\nor}[1]{\lVert #1 \rVert}

\newcommand{\ip}[1]{\left\langle #1 \right\rangle}

\newcommand{\cl}[1]{\mathcal{#1}}

\newcommand{\bs}[1]{\boldsymbol{#1}}

\newcommand{\qbin}[3]{\begin{bmatrix} #1 \\ #2 \end{bmatrix}_{#3}}
\newcommand{\tensor}{\otimes}
\newcommand{\fin}{\mathrm{f}}
\newcommand{\quot}{\mathrm{quot}}

\renewcommand{\star}{$\ast$}
\newcommand{\cstar}{C*}
\newcommand{\tend}{\textendash}
\renewcommand{\mod}[1]{#1\text{-}\mathrm{Mod}}
\newcommand{\roots}{R}
\newcommand{\simples}{\Delta}

\newcommand{\dqpind}[1]{\mathrm{ind}_{\frp_S,q}^{\g,\chi} #1}
\newcommand{\dqind}[1]{\mathrm{ind}_{\frb,q}^{\g,\chi} #1}
\newcommand{\bdqind}[1]{\mathrm{ind}_{\frb,q}^{\g,\chi} #1}
\newcommand{\qind}[1]{\mathrm{ind}_{\frb,q}^{\g} #1}
\newcommand{\qpind}[1]{\mathrm{ind}_{\frp_S,q}^{\g} #1}
\newcommand{\ch}[1]{\mathrm{ch}\,#1}
\newcommand{\deff}{\overset{\text{def}}{\Longleftrightarrow}}

\DeclareMathOperator{\Ch}{\mathrm{Ch}}

\DeclareMathOperator{\support}{\mathrm{supp}}

\DeclareMathOperator{\id}{\mathrm{id}}
\DeclareMathOperator{\wt}{\mathrm{wt}}

\DeclareMathOperator{\aspan}{\mathrm{span}}

\DeclareMathOperator{\rep}{\mathrm{Rep}}

\DeclareMathOperator{\irr}{\mathrm{Irr}}

\DeclareMathOperator{\End}{\mathrm{End}}

\title[ss module categories with fusion rules of
the cpt flag mfd type]{Semisimple module categories with fusion rules of
the compact flag manifold type}
\author{Mao Hoshino}
\address{iTHEMS, RIKEN\\
Hirosawa 2-1, Wako, Saitama \mbox{351-0198}, Japan}
\email{mao.hoshino@riken.jp}
\subjclass[2020]{Primary~17B37, Secondary~17B10, 46L65, 46L67}
\keywords{quantum group, quantization, representation theory, operator algebra}
\thanks{This work was supported by JSPS KAKENHI Grant Number JP23KJ0695 and WINGS-FoPM Program at the University of Tokyo.}

\begin{document}

\begin{abstract}
We classify semisimple left module categories over the representation category of a type A quantum group whose fusion rules arise from Levi subalgebras. The classification is connected to equivariant Poisson structures on compact flag manifolds in the operator-algebraic setting, and on semisimple coadjoint orbits in the algebraic setting. We also provide an explicit construction based on the BGG categories of deformed quantum enveloping algebras, whose unitarizability corresponds to being of quotient type.
Finally, we present a brief discussion of the non-quantum group case.
\end{abstract}

\maketitle
\tableofcontents

\section{Introduction}
The purpose of the present paper is to contribute a new result on
the Poisson geometric aspect of the Drinfeld-Jimbo deformation.

In the formal setting, the quantum coordinate ring $\cl{O}_h(G)$,
which is the Hopf dual of the quantum enveloping algebra $U_h(\g)$,
gives a deformation quantization of a semisimple algebraic group $G$
with respect to the standard Poisson-Lie structure.
By equipping these algebras with their
natural \star-structures,
one also obtains a deformation quantization of a compact real form $K$
of $G$. This observation leads to the theory of equivariant
deformation quantizations of homogeneous spaces over $K$, including
compact flag manifolds and symmetric spaces. For a semisimple coadjoint
orbit, which is a complexification of a compact flag manifold,
J. Donin showed that every equivariant Poisson structure has a corresponding deformation quantization equipped with an action of the quantum group \cite[Proposition 5.2]{MR1817512}. Moreover he classified such quantizations by Poisson structures admitting higher degree terms with respect to $h$ \cite[Proposition 5.3]{MR1817512}.
On the other hand, for symmetric spaces, a recent work \cite{MR4585468} due to K. De Commer, S. Neshveyev, L. Tuset and M. Yamashita give a certain classification in the framework of
quasi-coactions of the multiplier quasi-bialgebra $\cl{U}(G)$.
They also give a classification of ribbon braids,
which enables them to compare representations of the type B braid groups arising from the cyclotomic KZ equation and Letzter-Kolb coideals.

Even outside the formal setting, one finds a more indirect but still significant relationship between Poisson geometry and quantum groups through representation theory. As discussed in \cite{MR1116413},
irreducible representations of $C_q(K)$ are parametrized by the symplectic
leaves of $K$ with respect to the standard Poisson structures. Moreover,
similar results hold for quantizations of certain homogeneous
spaces of $K$ such as partial flag manifolds \cite{MR1697598} and
quotients by the Poisson subgroups \cite{MR2914062}. In another case,
recent works \cite{decommermoore1, MR4822588, moore3}
due to K. De Commer and S. Moore
reveal such a relationship for the quantization of the space
$H(N)$ of $N\times N$ hermitian matrices
with respect to the STS bracket, which is realized as the reflection equation algebra with respect to the Yang-Baxter operator on $\C^n$
arising from $U_q(\mathfrak{gl}_n)$.
These correspondences are not established via a direct geometric construction, but rather through indirect algebraic arguments, which nonetheless reveal the parallel.

The result presented in this paper may be regarded as part of this
series of indirect but remarkable relationships between Poisson geometry
and quantum groups in the non-root-of-unity case.
We focus on \emph{actions of $\sflag$-type} (Definition \ref{defn:action of flag manifold type}), where $K_S$ is a closed subgroup associated to a set $S$ of simple roots. They are defined as
semisimple left $\rep_q^{\fin} K$-module \cstar-categories
with the fusion rule same with that of a left $\rep^{\fin} K$-module
category $\rep^{\fin} K_S$.
We also consider the algebraic setting, in which the base field
$k$ is of characteristic $0$ and the module categories are called \emph{semisimple actions of $\ssorb$-type} (Definition \ref{defn:action of ssorb type}).

The main results of the present paper are the classifications of these actions. For a field $k$ of characteristic $0$, we define $X_{\ssorb}(k)$ and
$X_{\ssorb}^{\circ}(k)$ as follows:
\begin{align*}
 X_{\ssorb}(k) &:= \left\{(\phi_{\alpha})_{\alpha} \in k^{\roots\setminus \roots_S}\,\middle|\,
\begin{aligned}
\phi_{-\alpha} &= -\phi_{\alpha}, \\
 \phi_{\alpha} &= \phi_{\alpha + \gamma} \text{ if } \gamma \in \roots_S,\\
\phi_{\alpha}\phi_{\beta} + 1 &= \phi_{\alpha + \beta}(\phi_{\alpha} + \phi_{\beta}).
\end{aligned} \right\},\\
 X_{\ssorb}^{\circ}(k) &:= \{\phi \in X_{\ssorb}(k)\mid \phi_{\alpha} - 1 \not\in (\phi_{\alpha} + 1)q_{\alpha}^{2\Z}\},
\end{align*}
where $\roots$ is the root system associated with $\h \subset \g$
and $\roots_S$ is the closed subsystem associated to $S$.
We also consider $X_{\sflag}$ and $X_{\sflag}^{\quot}$ for the classification in the \cstar-algebraic setting:
\begin{align*}
 X_{\sflag} &:= X_{\ssorb}(\R),\\
 X_{\sflag}^{\quot} &:= \{\phi \in X_{\sflag}\mid -1 \le \phi_{\alpha} \le 1\}.
\end{align*}

Using the deformed quantum enveloping algebra introduced in
\cite{MR4837934}, we obtain a left $\rep_q^{\fin} G$-module category
$\pintO{\phi}$ for any $\phi \in X_{\ssorb}(k)$. By Theorem \ref{thrm:characterization of semisimplicity}, Proposition \ref{prop:semisimple action arising from the category O} and Lemma \ref{lemm:identification of toric and moduli}, this gives a semisimple action
of $\ssorb$-type when $\phi \in X_{\ssorb}^{\circ}(k)$.
In the type $\mathrm{A}$ case, we show that these actions exhaust all:

\vspace{3mm}
\noindent
\textbf{Theorem \ref{thrm:algebraic classification}.}
Let $\cl{M}$ be a semisimple action of $\ssorba$-type.
Then there is a unique $\phi \in X_{\ssorba}^{\circ}(k)$
such that $\cl{M} \cong \cl{O}_{q,\phi}^{\frp_I,\mathrm{int}}$.
\vspace{3mm}

\noindent
\textbf{Corollary \ref{coro:operator algebraic classification}.}
Let $\cl{M}$ be an action of $\sflaga$-type.
Then there is a unique $\phi \in X_{\sflaga}^{\quot}$
such that $\cl{M} \cong \cl{O}_{q,\phi}^{\frp_I,\mathrm{int}}$.
\vspace{3mm}

At least for the latter statement, one possible interpretation
can be found in the theory of quantization. By Tannaka-Krein duality
\cite{MR3121622, MR3426224}, an action of $\sflag$-type produces
a \cstar-algebra equipped with an ergodic action of $K_q$.
Then it is natural to
regard such an action as a noncommutative analogue of $\sflag$
since $\rep^{\fin} K_S$ corresponds to $C(\sflag)$ under this duality,
On the other hand, it is known that $X_{\sflag}$ classifies
the equivariant Poisson structures on $\sflag$ (c.f. \cite{MR1817512}).
Moreover, as discussed in Proposition \ref{prop:characterization of quotient Poisson structure}, $X_{\sflag}^{\quot}$
classifies the equivariant Poisson structures
admitting a $0$-dimensional symplectic leaf. Hence
the main theorem says that noncommutative compact flag manifolds
of $SU(n)$ are
classified by the suitable Poisson structures on $\sflaga$. 
This situation is somewhat similar to the situation in the theory
of deformation quantization (c.f. \cite{MR2062626}).
The same interpretation also would be applicable in the algebraic setting
(c.f. the duality theorem \cite[Theorem 4.6]{MR3847209}),
after removing the assumption of semisimplicity which excludes
Poisson structures that should originally be taken into account.
In that case the classification would be modified and contain
$\pintO{\phi}$ for $\phi \in X_{\ssorb}(k)\setminus X_{\ssorb}^{\circ}(k)$.
However we do not pursue this direction in the present paper
since our original motivation is in the \cstar-algebaic setting,
in which the semisimplicity is completely natural.

It also should be noted that the situation in the \cstar-algebraic setting
is similar to but differs from the theory of
deformation quantization at the point that the nontrivial restriction is imposed on the equivariant Poisson structures. As in
Woronowicz's no-go theorem \cite[Theorem 4.1]{MR1096123} on the quantization of $\SL(2,\R)$,
existence of such a restriction can be naturally interpreted
as a no-go theorem for equivariant Poisson structures.
In this respect, the same phenomenon can be observed for quantum groups beyond type $\mathrm{A}$ for
the family of left $\rep_q^{\fin} G$-module categories
arising from deformed quantum enveloping algebras.

\vspace{3mm}
\noindent
\textbf{Theorem \ref{prop:characterization of unitarizability}.}
 For $\phi \in X_{\ssorb}(\C)$, $\pintO{\phi}$
is unitarizable iff $\phi \in X_{\sflag}^{\mathrm{quot}}$.
\vspace{3mm}

Besides the quantization perspective, Theorem \ref{thrm:algebraic classification} and Corollary \ref{coro:operator algebraic classification}
also belongs to the context of classification
of tensor categories and related structures.

If we focus on the statement itself, there are several related results, including reconstruction theorems of tensor categories \cite{MR1237835, MR2132671, MR3266525}, the classfication of fiber functors on $\rep_q^{\fin} K$
with the classical dimension functions \cite[Corollary 4.4]{MR3556413},
and the classification of quantum spheres \cite[Theorem 1, Theorem 2]{MR0919322}.
One of the most strongly related works is
\cite{MR3420332}, in which
ergodic actions of the quantum $\SU(2)$
are classified by graphes equipped with numerical data, called fair and balanced cost. In particular, they give
a classification \cite[Example 3.12]{MR3420332} of quantum spheres,
which is the rank $1$ case of Corollary \ref{coro:operator algebraic classification}.

On the other hand, if we focus on the strategy used in the proof of Theorem
\ref{thrm:algebraic classification} and Corollary \ref{coro:operator algebraic classification}, the paragroup theory relates to our theorem. The paragroup theory, introduced by A. Ocneanu \cite{MR996454}, has played an essential role in the classification of
subfactors and the discovery of new quantum symmetries including
Haagerup symmetry \cite{MR1317352, MR1686551}. In the theory, tensor categorical structures are encoded into graphs together with certain numerical data, representing
fusion rules and associators respectively. This reformulation makes it possible to treat some abstract conditions in a more combinatorial manner.
Our strategy to prove the main theorems
is also based on the same idea. Actually we consider
numerical data called \emph{scalar systems of $\ssorba$-type} (Definition \ref{defn:scalar system})
and classify them with the parameter space $X_{\ssorba}^{\circ}(k)$.
This data naturally arises
from every semisimple action of $\ssorba$-type
by focusing on generating morphisms in $\rep_q^{\fin} \SL_n$ described
in \cite{MR3263166}. In light of this background,
it is also possible to find a concrete connection with Ocneanu's cell system \cite{MR1907188} (c.f. \cite{MR2545609}), which is based on Kuperberg's spider for $\mathrm{A}_2$ \cite{MR1403861}.

\subsection*{Outline of the paper}
In Section \ref{sect:Poisson geometric thing} we give a brief review
on $K^{\std}$-equivariant Poisson structures on compact flag manifolds
and prove the characterization of Poisson structures of ``quotient type''
in terms of the parameter space $X_{\sflag}$.

After this section, there is no discussion on Poisson structure.
In Section \ref{sect:catO}, we investigate the category $\cl{O}$ of
deformed quantum enveloping algebras introduced in
\cite{MR4837934}. Some properties including simplicity and projectivity
of twisted Verma modules are discussed therein.

In Section \ref{sect:actions}, we introduce the main subject of this paper,
semisimple actions of $\ssorb$-type and actions of $\sflag$-type.
We also introduce some operations applicable to general semisimple
actions of $\ssorb$-type and investigate its properties
concerned with the actions arising from
the deformed quantum enveloping algebras.
At the end of this section, we discuss on the unitarizability.

In Section \ref{sect:classification}, we show the classification
results on semisimple actions
of $\ssorba$-type and actions of $\sflaga$-type. This part
is most technical in this paper, which relies on
the paragroup-theoretical argument.

In Section \ref{sect:non-quantum}, we present a discussion
in the non-quantum group case.
In particular we show that actions of $\sflaga$-type are
equivalent to $\sflaga$, which implies that $\sflaga$ admits
no nontrivial equivariant quantization in the operator algebraic
setting.

\section{Preliminaries}
\subsection{Notations and convensions}
Throughout this paper, the base field $k$ is of characteristic $0$
and not assumed to be algebraically closed. In the operator algebraic
setting, we consider the field $\C$ of complex numbers.

For $q$-integers, we use the following symbols:
\begin{gather*}
  q_{\alpha} = q^{d_{\alpha}},\quad
 [n]_q = \frac{q^n - q^{-n}}{q - q^{-1}},\quad
 [n]_q! = [1]_q[2]_q\cdots [n]_q,\\
\qbin{n}{k}{q} =
\begin{cases}
 {\displaystyle \frac{[n]_q[n - 1]_q\cdots [n - (k - 1)]_q}{[k]_q!}} & (k \ge 0), \\
 0 & (k < 0).
\end{cases}
\end{gather*}
Additionally we also use the following notation for $\chi = [x:y] \in \P^1(k)$
and $n, m \in \Z$:
\begin{align*}
 \frac{[n ; \chi]_q}{[m ; \chi]_q} := \frac{xq^n - yq^{-n}}{xq^m - yq^{-m}}.
\end{align*}
Note that we have
\begin{align*}
 \frac{[n;q^{2l}]_q}{[m;q^{2l}]_q} = \frac{[n + l]_q}{[m + l]_q},\quad
 \frac{[n;\infty]_q}{[m;\infty]_q} = q^{n - m}
\end{align*}
where $x \in k$ in general is identified with $[x:1] \in \P^1(k)$
and $\infty$ denotes $[1:0]$.
A quantum commutator is defined for elements in suitable algebras which
admit weight space decompositions as stated in Subsection \ref{subsec:quantum group}:
\begin{align*}
 [x,y]_q = xy - q^{-(\wt{x},\wt{y})}yx.
\end{align*}
We use the following notations on a multi-index $\Lambda = (\lambda_i)_i \in \Z_{\ge 0}^n$.
\begin{itemize}
 \item $\abs{\Lambda} = \sum_{i} \lambda_i$.
 \item $\support \Lambda = \{i\mid \lambda_i \neq 0\}$.
 \item $\Lambda \subset (k,l) \deff \support \Lambda \subset \{k + 1, k + 2,\cdots, l - 1\}$. For an interval $I$, like $[k,l]$,
$\Lambda \subset I$ is defined in a similar way.
 \item $\Lambda < k \deff \Lambda \subset (0,k)$. Similarly
$\Lambda \le k, \Lambda > k, \Lambda \ge k$ are defined.
 \item $\Lambda\cdot\alpha = \sum_i \lambda_i\alpha_i$ for
a sequence $(\alpha_i)_i$ of vectors.
 \item $x^{\Lambda} = x_1^{\lambda_1}x_2^{\lambda_2}\cdots x_n^{\lambda_n}$ for a sequence $(x_i)_i$ in a (possibly non-commutative) ring.
\end{itemize}

\subsection{Lie algebras and Lie groups}
In this paper $\g$ and $\h$ denote a split semisimple Lie algebra and
its split Cartan subalgebra respectively.
The associated set of roots is denoted by $\roots$.
We fix an invariant symmetric bilinear form $B(\tend,\tend)$ on
$\g$ and consider the induced bilinear form $(\tend,\tend)$ on $\h^*$.
We normalize the original bilinear form $B$
so that $(\alpha,\alpha) = 2$ for all short roots $\alpha$.

Then this induces an inner product on $\h_{\R}^* := \R\tensor_{\Q} \Q\roots$, which makes $R \subset \h_{\R}^*$ into a root system.
The reflection with respect to $\alpha \in \roots$ is
denoted by $s_{\alpha}$. The associated Weyl group is denoted by $W$.

We fix a positive system $\roots^+$, which induces
a triangular decomposition $\g = \n^-\oplus \h\oplus \n^+$
and defines a set $\simples = \{\epsilon_1,\epsilon_2,\dots,\epsilon_r\}$
of simple roots. Note that the number $r$ of simple roots is
the rank of $\g$. We also set $N = \abs{\roots^+}$.
The set of reflections with respect to simple roots generates $W$,
and defines the length function $\map{\ell}{W}{\Z_{\ge 0}}$.
The unique longest element is denoted by $w_0$, whose length is $N$.

We set $d_{\alpha}, \alpha^{\vee}, a_{ij}$ as follows:
\[
d_{\alpha} = \frac{(\alpha,\alpha)}{2},\quad
\alpha^{\vee} = d_{\alpha}^{-1}\alpha,\quad
a_{ij} = (\eps_i^{\vee},\eps_j).
\]
The fundamental weights, which are dual to $(\epsilon_i^{\vee})_i$
with respect to $(\tend,\tend)$, are denoted by $\varpi_i$.
The root lattice $Q$ (resp. the weight lattice $P$) is
the $\Z$-linear span of $\simples$ (resp. $(\varpi_i)_i$).
We also use the positive cone $Q^+$ and $P^+$:
\begin{align*}
 Q^+ &= \Z_{\ge 0}\epsilon_1 + \Z_{\ge 0}\epsilon_2 + \cdots + \Z_{\ge 0}\epsilon_r,\\
 P^+ &= \Z_{\ge 0}\varpi_1 + \Z_{\ge 0}\varpi_2 + \cdots + \Z_{\ge 0}\varpi_r.
\end{align*}
We usually replace $\epsilon_i$ by the symbol $i$ when $\epsilon_i$
appears as a subscript.
For instance we use $s_i, d_i, H_i, K_i$ instead of
using $s_{\eps_i}, d_{\epsilon_i}, H_{\eps_i}, K_{\eps_i}$.

For each subset $S \subset \simples$, we define a closed subsystem
$\roots_S = \roots \cap \mathrm{span}_{\Z} S$. This gives
the Levi subalgebra $\frl_S$ and the parabolic subalgebra $\frp_S$.

At the end of this subsection, we give a brief review
on the representation theory. Let $G$ be the connected universal
algebraic group associated to $\g$ and $L_S$ be the subgroup
corresponding to $\frl_S$.

In this paper, the category of finite dimensional representations
of $L_S$ is denoted by $\rep^{\fin} L_S$.
Note that $\rep^{\fin} L_S$ is equivalent to $\rep^{\fin} \frl_S$ as
$k$-linear tensor category.
We identify $\irr{\rep^{\fin} L_S}$ with $P^{+_S}$ defined as follows:
\begin{align*}
 P^{+_S} := \sum_{\eps_i \in S} \Z_{\ge 0}\varpi_i + \sum_{\eps_i \notin S} \Z\varpi_i
\end{align*}

\subsection{The Drinfeld-Jimbo deformations}
\label{subsec:quantum group}
Basically we refer the convension in \cite{MR4162277} and \cite{MR1492989}.
A textbook \cite{MR1359532} is also helpful for basic facts on
quantum groups.

Let $L$ be the smallest positive integer such that
$(\lambda,\mu) \in L^{-1}\Z$ for any $\lambda,\mu \in P$.
We fix a homomorphism $\map{q}{(2L)^{-1}\Z}{k^{\times}},\,r \longmapsto q^r$ and assume that this is injective, i.e., $q$ is not a root of unity.

The \emph{Drinfeld-Jimbo deformation} of $\g$ is a Hopf algebra
$U_q(\g)$ generated by $E_i, F_i, K_{\lambda}$ for $1 \le i \le r$ 
and $\lambda \in P$, with relations
\begin{align*}
 &K_0 = 1, &&K_{\lambda}E_iK_{\lambda}^{-1} = q^{(\lambda,\epsilon_i)}E_i, 
&& [E_i, F_j] = \delta_{ij}\frac{K_i - K_i^{-1}}{q_i - q_i^{-1}},\\
 &K_{\lambda}K_{\mu} = K_{\lambda + \mu}, 
&&K_{\lambda}F_iK_{\lambda}^{-1} = q^{-(\lambda,\epsilon_i)}F_i,
\end{align*}
and the quantum Serre relations:
\begin{align*}
&\sum_{k = 0}^{1 - a_{ij}}(-1)^k\qbin{1 - a_{ij}}{k}{q_i}E_i^{1 - a_{ij} - k}E_jE_i^k = 0,\\
&\sum_{k = 0}^{1 - a_{ij}}(-1)^k\qbin{1 - a_{ij}}{k}{q_i}F_i^{1 - a_{ij} - k}F_jF_i^k = 0.
\end{align*}
The coproduct $\Delta$, the antipode $S$ and the counit $\eps$ 
are given as follows on the generators:
\begin{align*}
 &\Delta(K_{\lambda}) = K_{\lambda}\tensor K_{\lambda}, &&S(K_{\lambda}) = K_{\lambda}^{-1}, &&\eps(K_{\lambda}) = 1,\\
 &\Delta(E_i) = E_i\tensor 1 + K_i\tensor E_i, &&S(E_i) = -K_i^{-1}E_i, &&\eps(E_i) = 0,\\
 &\Delta(F_i) = F_i\tensor K_i^{-1} + 1\tensor F_i, &&S(F_i) = -F_iK_i, &&\eps(F_i) = 0.
\end{align*}

Next we introduce some subalgebras of $U_q(\g)$.
The most fundamental ones are $U_q(\n^+), U_q(\n^-)$ and
$U_q(\h)$, which are generated by
$\{E_i\}_i, \{F_i\}_i, \{K_{\lambda}\}_{\lbd \in P}$ respectively.
These allow us to decompose $U_q(\g)$ into the tensor products
$U_q(\n^{\pm})\tensor U_q(\h)\tensor U_q(\n^{\mp})$ via
the multiplication maps. We also use $U_q(\frb^{\pm})$
for the subalgebras generated by $U_q(\h)$ and $U_q(\n^{\pm})$
respectively.
Note that $U_q(\frb^{\pm})$ are Hopf subalgebras of $U_q(\g)$.

Let $\h_q^*$ be the set of $k^{\times}$-valued characters on $P$.
The weight lattice $P$ is embedded into $\h_q^*$
by $\lbd \longmapsto q^{(\lbd,\tend)}$, which is injective
by our assumption on $q$. More generally, we substitute
$q^{(\xi,\tend)}$ for $\xi \in \h_q^*$. In this notation
the canonical structure of $\h_q^*$ is presented additively,
i.e. we have $\xi(\lbd)\eta(\lbd) = (\xi + \eta)(\lbd) = q^{(\xi + \eta,\lbd)}$.

For a $U_q(\h)$-module $M$ and $v \in M$, we say that
$v$ is a weight vector with weight $\xi \in \h_q^*$ when
$K_{\mu}v = q^{(\xi, \mu)}v$ for all $\mu \in P$. In this case
$\xi$ is denoted by $\wt v$. The submodule of elements of
weight $\xi$ is denoted by $M_{\xi}$.
To consider the weight of
an element of $U_q(\g)$, we regard $U_q(\g)$ as a
$U_q(\h)$-module by the left adjoint action
$x\triangleright y = x_{(1)}yS(x_{(2)})$.

Next we describe the braid group action on $U_q(\g)$ and
the quantum PBW bases.
At first we have an algebra automorphism
$\cl{T}_i$ on $U_q(\g)$ for each $\eps_i \in \simples$,
which satisfies
\begin{align*}
\cl{T}_i(K_{\lambda}) = K_{s_i(\lambda)}, \quad
\cl{T}_i(E_i) = -F_iK_i, \quad
\cl{T}_i(F_i) = -K_i^{-1}E_i
\end{align*}
and other formulae in \cite[8.14]{MR1359532}
which determine $\cl{T}_i$ uniquely.

Then the family $(\cl{T}_i)_i$ satisfies the Coxeter relations and
defines an action of the braid group on $U_q(\g)$. Especially we have
$\cl{T}_w$ for each $w \in W$, which is given by
$\cl{T}_w = \cl{T}_{i_1}\cl{T}_{i_2}\cdots\cl{T}_{i_{\ell(w)}}$ where
$w = s_{i_1}s_{i_2}\cdots s_{i_{\ell(w)}}$ is a reduced expression.

This action produces PBW bases of $U_q(\g)$.
Let $w_0$ be the longest element in $W$ and fix its reduced expression
$w_0 = s_{\bs{i}} = s_{i_1}s_{i_2}\cdots s_{i_N}$, where
$\bs{i} = (i_1,i_2,\dots,i_N)$.
Then each $\alpha \in R^+$ has a unique positive integer $k \le N$
with $\alpha = \alpha^{\bs{i}}_{k} := s_{i_1}s_{i_2}\cdots s_{i_{k - 1}}(\eps_{i_k})$.
Finally we set $E_{\bs{i},\alpha}$ and 
$F_{\bs{i},\alpha}$, the quantum root vectors, as follows:
\begin{align*}
 E_{\bs{i},\alpha} &= E_{\bs{i},k} 
:= \cl{T}_{s_{i_1}s_{i_2}\cdots s_{i_{k - 1}}}(E_{i_k}) 
= \cl{T}_{s_{i_1}}\cl{T}_{s_{i_2}}\cdots\cl{T}_{s_{i_{k - 1}}}(E_{i_k}), \\
 F_{\bs{i},\alpha} &= F_{\bs{i},k} 
:= \cl{T}_{s_{i_1}s_{i_2}\cdots s_{i_{k - 1}}}(F_{i_k}) 
= \cl{T}_{s_{i_1}}\cl{T}_{s_{i_2}}\cdots\cl{T}_{s_{i_{k - 1}}}(F_{i_k}).
\end{align*}

Though these elements depend on $\bs{i}$, we still have an
analogue of the Poincar\'{e}-Birkhoff-Witt theorem in $U_q(\g)$ 
i.e. $\{F_{\bs{i}}^{\Lambda^-}K_{\mu}E_{\bs{i}}^{\Lambda^+}\}_{\Lambda^{\pm},\mu}$ forms a basis of $U_q(\g)$. Each element of
this family is called a \emph{quantum PBW vector}.

For $S \subset \simples$, the subalgebra generated by $U_q(\h), \{E_{\eps}\}_{\eps \in S}$ and $\{F_{\eps}\}_{\eps \in S}$ is denoted by $U_q(\frl_S)$.
In this paper, a \emph{finite dimensional representation} of $U_q(\g)$
is a finite dimensional $U_q(\frl_S)$-module admitting a weight space decomposition with weights in $P$. The category of finite dimensional representations of $U_q(\frl_S)$ is denoted by $\rep_q^{\fin} L_S$.
Then we identify $\irr{\rep_q^{\fin} L_S}$ with $P^{+_S}$,
hence with $\irr{\rep^{\fin} L_S}$,
by looking at highest weights with respect to $U_q(\n^+)$.
The irreducible representation corresponding to $\lbd \in P^+$
is denoted by $L^S(\lbd)$. If $S$ is the whole set $\simples$,
we also write $L(\lbd)$.
Note that this identification preserves the fusion rules, i.e.,
$\Z_+(\rep^{\fin} L_S) \cong \Z_+(\rep^{\fin}_q L_S)$ as $\Z_+$-rings
(\cite[Definition 3.1.1]{MR3242743}). This $\Z_+$-ring is
denoted by $\Z_+(L_S)$.
Note that $\Z_+(L_S)$ has a natural structure of $\Z_+$-module over
$\Z_+(G)$, which is compatible with the identifications.

\subsection{Compact real forms}
\label{subsec:compact real forms}

We also consider the operator algebraic setting,
in which quantum groups should be considered as quantizations of
compact Lie groups.

Assume $k = \C$. The compact real form of $(\g,\h)$ is
denoted by $(\k,\t)$, i.e., $\k$ is a compact Lie subalgebra
of the real Lie algebra $\g_{\R}$ satisfying $\g_{\R} = \k\oplus i\k$,
and $\t := \k \cap \h$ is a Cartan subalgebra of $\k$ satisfying
$\h_{\R} = \t \oplus i\t$.
Then we have a conjugate linear involutive
anti-automorphism $X \longmapsto X^*$ on $\g$ defined as $X^* = -X$
for $X \in \k$. The compactness of $\k$ implies that $(X,Y) := B(X^*,Y)$
is an hermitian inner product on $\g$.

We also fix a Chevalley system which is compatible with $(\k,\t)$, i.e.,
a family $\{(E_{\alpha},F_{\alpha}, H_{\alpha})\}_{\alpha \in R^+}$
of $\fsl_2$-triplets such that $E_{\alpha}^* = F_{\alpha}$ and $H_{\alpha}^* = H_{\alpha}$. Note that $\nor{E_{\alpha}} = \nor{F_{\alpha}} = d_{\alpha}^{-1/2}$, where $\nor{\tend}$ is the norm induced from the inner product above. Using this system, $\k$ and $\t$ are presented as follows:
\begin{align*}
 \t = \bigoplus_{i = 1}^r i\R H_i,\quad
 \k = \t \oplus\bigoplus_{\alpha \in R^+} \R(E_{\alpha} - F_{\alpha})\oplus\bigoplus_{\alpha \in R^+} i\R(E_{\alpha} + F_{\alpha}).
\end{align*}

Since $G$ is a (complex) Lie group
in this setting, there is connected closed subgroup $K$
correponding to $\k$. We also consider the closed subgroup $K_S$
of $K$ corresponding to $\k_S := \frl_S \cap \k$.
This is also a connected compact Lie
group. Moreover the complexifications of $K_S$ is isomorphic to
$L_S$.

The category of finite dimensional unitary representations
of $K_S$ is denoted by $\rep^{\fin} K_S$.
Note the canonical equivalence $\rep^{\fin} K_S \cong \rep^{\fin} L_S$.
In particular we can identify $\irr{\rep^{\fin} K_S}$
with $P^{+_S}$.

We also consider the compact real form
of $U_q(\g)$. Assume $q^r > 0$ for all $r \in (2L)^{-1}\Z$ and $q < 1$.
We define a Hopf \star-algebra $U_q(\k)$,
whose underlying Hopf algebra is $U_q(\g)$ and whose \star-structure is given as follows:
\begin{align*}
 E_i^* = K_iF_i,\quad
 F_i^* = E_iK_i^{-1},\quad
 K_{\lbd}^* = K_{\lbd}.
\end{align*}
Then $U_q(\frl_S)$ is closed under the involution
and defines a Hopf \star-algebra $U_q(\k_S)$.
A finite dimensional unitary representation of $U_q(\k_S)$
is a finite dimensional representation of $U_q(\g)$
on a Hilbert space $H$ such that
$\ip{\xi,X\eta} = \ip{X^*\xi,\eta}$ holds for all
$\xi,\eta \in H$ and $X \in U_q(\k_S)$.
Then we have the \cstar-tensor category
$\rep_q^{\fin} K_S$, whose irreducible objects
are parametrized by $P^{+_S}$.

For consistency, the $\Z_+$-rings $\Z_+(L_S)$ is denoted by $\Z_+(K_S)$
in this setting.

\section{Equivariant Poisson structures on compact flag manifolds}
\label{sect:Poisson geometric thing}
In this section, we recall the classification of
equivariant Poisson structures on semisimple coadjoint orbits
and compact flag manifolds. Additionally we also provide a classification
of equivariant Poisson structures with $0$-dimensional symplectic leaves, which is important to interpret Corollary \ref{coro:operator algebraic classification}.

\subsection{A brief review on the classification theorem for $L_S \backslash G$}

At first we recall a Poisson geometric aspect of $G$.
Let $r$ be the following elements of $\bigwedge^2 \g$, which is called
the \emph{starndard r-matrix}:
\begin{align*}
 r := \sum_{\alpha \in R^+} d_{\alpha}E_{\alpha}\wedge F_{\alpha}
\end{align*}
For $v \in \bigwedge^{\bullet} \g$, the corresponding left (resp. right) invariant
polyvector field is denoted by $v_L \in \Gamma(G, \bigwedge^{\bullet} TG)$ (resp. $v_R \in \Gamma(G,\bigwedge^{\bullet} TG)$). Under this notation,
the standard Poisson structure on $G$ can be presented as follows:
\begin{align*}
 \pi_G := r_R - r_L.
\end{align*}
It is known that this makes $G$ into a Poisson algebraic group,
which is denoted by $G^{\std}$ in this paper.

Next we look at $G$-actions on Poisson varieties.
A \emph{Poisson $G^{\std}$-variety} is a pair of a Poisson variety
$(X,\pi_X)$ and a right $G$-action on $X$
such that the action map $X\times G^{\std} \longrightarrow X$
is a morphism of Poisson varieties.

Consider a subset $S \subset \simples$. Then we have a right $G$-variety
$L_S\backslash G$. Note that the space of right invariant
polyvector fields is identified with $(\bigwedge^{\bullet} \m_S)^{\frl_S}$,
where $\m_S := \sum_{\alpha \in R\setminus R_S}\g_{\alpha}$. For
$v \in (\bigwedge^{\bullet} \m_S)^{\frl_S}$, the corresponding
right invariant polyvector field is denoted by $v_R$.

Let $r_L$ be the bivector field on $L_S\backslash G$
induced from $r$ by the right $G$-action. We also introduce
$X_{L_S\backslash G}(k)$ defined as the subset of $k^{R\setminus R_S}$
consisting of elements satisfying the following conditions:
\begin{enumerate}
 \item $\phi_{-\alpha} = -\phi_{\alpha}$ for all $\alpha \in R\setminus R_S$. \label{rel1}
 \item $\phi_{\alpha}\phi_{\beta} + 1 = \phi_{\alpha + \beta}(\phi_{\alpha} + \phi_{\beta})$ when $\alpha,\beta,\alpha + \beta \in R\setminus R_S$.
\label{rel2}
 \item $\phi_{\alpha} = \phi_{\beta}$ when $\alpha,\beta \in R\setminus R_S$ and $\alpha - \beta \in \aspan_{\Z} S$. \label{rel3}
\end{enumerate}
Note that $\phi \in X_{L_S\backslash G}(k)$ defines $v(\phi) \in (\bigwedge^2\m_S)^{\frl_S}$ as follows:
\begin{align*}
 v(\phi) = \sum_{\alpha \in R^+\setminus R_S^+} d_{\alpha}\phi_{\alpha}E_{\alpha}\wedge F_{\alpha},
\end{align*}

The following fact is pointed out in \cite{MR1817512}.
\begin{prop} \label{prop:classification of equiv Poisson var}
Let $\pi$ be a bivector field on $L_S\backslash G$.
Then $(L_S\backslash G, \pi)$ is a Poisson $G^{\std}$-variety
if and only if $\pi = v(\phi)_R - r_L$ for some
$\phi \in X_{L_S\backslash G}(k)$.
\end{prop}

\subsection{The classification theorem for $K_S\backslash K$}

Next we focus on the compact real form $K$. Note that
we have the following expression of the standard r-matrix:
\begin{align*}
 r = \frac{1}{2i}\sum_{\alpha \in R^+} d_{\alpha}(E_{\alpha} - F_{\alpha})\wedge(iE_{\alpha} + iF_{\alpha}).
\end{align*}
Hence we can regard $ir$ as an element of $\bigwedge^2 \k$. Then
this defines the standard Poisson structure $\pi_K := (ir)_L - (ir)_R$
on $K$, which makes $K$ into a Poisson-Lie group denoted by $K^{\std}$.
As same with $G^{\std}$, we have the notion of \emph{Poisson $K^{\std}$-manifold}.

Fix a subset $S \subset \simples$ and set $X_{K_S\backslash K} := X_{\lsbg}(\R)$.
Then we can see that $iv(\phi) \in (\bigwedge^2 (\k\cap \m_S))^{\k_S}$,
which defines a right $K$-invariant bivector $(iv(\phi))_R$ on $K_S\backslash K$. We can see the following proposition in the completely same way with
Proposition \ref{prop:classification of equiv Poisson var}:

\begin{prop} \label{prop:classificatoin of equiv Poisson var}
Let $\pi$ be a bivector field on $K_S\backslash K$.
Then $(K_S\backslash K, \pi)$ is a Poisson $K^{\std}$-manifold
if and only if $\pi = (iv(\phi))_R - (ir)_L$ for some
$\phi \in X_{K_S\backslash K}$.
\end{prop}

Let us recall the notion of a symplectic leaf of a Poisson manifold.
For a Poisson manifold $(M,\pi_M)$, a symplectic leaf is
a connected Poisson submanifold on which $\pi_M$ is non-degenerate
at each point. It is known that every Poisson manifold has a decomposition
into its symplectic leaves. We also remark here that
$\{m\} \subset M$ is a symplectic leaf if and only if $\pi_M(m) = 0$.

The following characterization is important when we consider
a classification of ``noncommutative flag manifolds''
in the \cstar-algebraic setting:

\begin{prop} \label{prop:characterization of quotient Poisson structure}
For $\phi \in X_{K_S\backslash K}$, the following are equivalent:
\begin{enumerate}
 \item There eixsts a Poisson $K^{\std}$-map $(K,\pi_K) \longrightarrow (K_S\backslash K, \pi_{\phi})$.
 \item There exists a $0$-dimensional symplectic leaf of $(K_S\backslash K, \pi_{\phi})$.
 \item For all $\alpha \in \roots\setminus\roots_S$, $\abs{\phi_{\alpha}} \le 1$.
\end{enumerate}
\end{prop}

We say that $\pi_{\phi}$ is \emph{of quotient type} if $\phi$
satisfies the conditions above.
The set of $\phi \in X_{K_S\backslash K}$
satisfying the conditions above is denoted by
$X_{K_S\backslash K}^{\quot}$:
\begin{align*}
 X_{K_S\backslash K}^{\quot} := \{\phi \in X_{K_S\backslash K}\mid -1 \le \phi_{\alpha} \le 1\}.
\end{align*}

\begin{proof}[Proof of Proposition \ref{prop:characterization of quotient Poisson structure} (i) $\Longleftrightarrow$ (ii) $\Longrightarrow$ (iii)]
For $[x_0] = K_Sx_0 \in K_S\backslash K$, we define
$\map{\ell_{[x_0]}}{K}{K_S\backslash K}$ by $x\longmapsto [x_0x]$.

It is not difficult to see the equivalence of (i) and (ii).
Actually $\ell_{[x_0]}$ is a Poisson $K^{\std}$-map
if and only if $\{[x_0]\}$ is a symplectic leaf.

To see (ii) $\Longrightarrow$ (iii), assume $(K_S\backslash K, \pi_{\phi})$ has a $0$-dimensional
symplectic leaf $\{[x_0]\}$. Then $\ell_{[x_0]}$
is a Poisson map. Hence we have
\begin{align*}
iv(\phi) - \mathrm{pr}(ir)
= \pi_{\phi}([e]) = d\ell_{[x_0]}\pi_K(x_0^{-1})
= \mathrm{pr}(\mathrm{Ad}_{x_0}(ir) - ir),
\end{align*}
where $\map{\mathrm{pr}}{\k}{\k\cap \m_S}$ is the canonical projection.
For convenience, we regard $\bigwedge^2 (\k\cap \m_S)$
as an $\R$-subspace of $\bigwedge^2 \m_S$. Then we have
\begin{align*}
 \sum_{\alpha \in \roots^+\setminus\roots_S^+} d_{\alpha}\phi_{\alpha}E_{\alpha}\wedge F_{\alpha} = \mathrm{pr}(\mathrm{Ad}_{x_0}(r)).
\end{align*}
Hence, for any $\alpha \in \roots^+\setminus \roots_S^+$, we have
\begin{align*}
 \phi_{\alpha}
&= \ip{d_{\alpha}E_{\alpha}\wedge F_{\alpha}, \mathrm{Ad}_{x_0}(r)} \\
&= \sum_{\beta \in \roots^+}d_{\alpha}d_{\beta}(\ip{E_{\alpha}, \mathrm{Ad}_{x_0}E_{\beta}}\ip{F_{\alpha}, \mathrm{Ad}_{x_0}F_{\beta}} - \ip{F_{\alpha}, \mathrm{Ad}_{x_0}E_{\beta}}\ip{E_{\alpha}, \mathrm{Ad}_{x_0}F_{\beta}}).
\end{align*}
Since we have
\begin{align*}
 \ip{X,\mathrm{Ad}_k(Y)} = \bar{\ip{X^*,\mathrm{Ad}_k (Y^*)}}
\end{align*}
for $X,Y \in \g$ and $k \in K$, we can estimate the first summation
and the second summation as follows using Bessel's inequality:
\begin{align*}
 \sum_{\beta \in \roots^+} d_{\alpha}d_{\beta}\ip{E_{\alpha}, \mathrm{Ad}_{x_0}E_{\beta}}\ip{F_{\alpha}, \mathrm{Ad}_{x_0}F_{\beta}}
= \sum_{\beta \in \roots^+} d_{\alpha}d_{\beta}\abs{\ip{E_{\alpha},\mathrm{Ad}_{x_0}E_{\beta}}}^2 \le d_{\alpha}\nor{E_{\alpha}}^2 = 1,\\
 \sum_{\beta \in \roots^+} d_{\alpha}d_{\beta}\ip{F_{\alpha}, \mathrm{Ad}_{x_0}E_{\beta}}\ip{E_{\alpha}, \mathrm{Ad}_{x_0}F_{\beta}}
= \sum_{\beta \in \roots^+} d_{\alpha}d_{\beta}\abs{\ip{E_{\alpha},\mathrm{Ad}_{x_0}F_{\beta}}}^2 \le d_{\alpha}\nor{E_{\alpha}}^2 = 1.
\end{align*}
We also see that the LHSs are non-negative
since so are the middle terms.
Hence we see $-1 \le \phi_{\alpha} \le 1$.
\end{proof}

To see the converse direction, we need some observations on $X_{K_S\backslash K}^{\quot}$. We use the following elementary lemma without proof.

\begin{lemm} \label{lemm:elementary inequality}
Let $x,y,z$ be real numbers satisfying $xy + 1 = z(x + y)$.
If $\abs{x},\abs{y},\abs{z} \le 1$, either of $\abs{x} = 1$ or $\abs{y} = 1$ holds.
\end{lemm}

With an abuse of notation, we use $\simples$ for the Dynkin diagram associated to $(\roots,\roots^+)$
since its vertices are simple roots.
Then we say that a subset $\Gamma$ of $\simples$ is connected
when the associated full subgraph of the Dynkin diagram is connected.

\begin{lemm} \label{lemm:reduction to symmetric spaces}
Let $\phi$ be an element in $X_{K_S\backslash K}^{\quot}$ such that
$\phi_{\alpha}\neq -1$ for $\alpha \in \roots^+\setminus\roots_S^+$.
Then each connected component $\Gamma$ of
$\{\eps \in \Delta\setminus S\mid \phi_{\eps}\neq 1\} \cup S$
contains at
most one $\eps \in \Delta\setminus S$. Moreover the coefficient
of $\eps$ in the highest root $\beta_{\Gamma}$ of $\roots_{\Gamma}$
is $1$.
\end{lemm}
\begin{proof}
Take a connected component $\Gamma$ and $\beta_{\Gamma}$ be
the highest root in the root system generated by $\Gamma$.
Then we can find a sequence $\{\delta_j\}_{j = 1}^k$ in $\Gamma$
such that $\beta_l := \delta_1 + \delta_2 + \cdots + \delta_l \in \roots_{\Gamma}^+$ and $\beta_k = \beta_{\Gamma}$
(\cite[Chapter VI, Section 1, Proposition 19]{MR1890629}).

If $\Gamma$ is contained in $S$, there is nothing to prove.

Assume $\Gamma \not\subset S$ and
take $m \ge 1$ so that $\delta_l \in S$ for $1 \le l \le m - 1$
and $\delta_m \not \in S$.
Assume that there is another $m' > m$
such that $\delta_l \in S$ for $m < l < m'$ and $\delta_{m'} \not\in S$.
Then we have
\begin{align*}
 \phi_{\beta_{m'}}(\phi_{\delta_{m'}} + \phi_{\beta_{m' - 1}})
= \phi_{\delta_{m'}}\phi_{\beta_{m' - 1}} + 1.
\end{align*}
Since $\phi_{\beta_{m' - 1}} = \phi_{\beta_m} = \phi_{\delta_m}$ holds,
Lemma \ref{lemm:elementary inequality} implies $\abs{\phi_{\beta_{m'}}} > 1$, which
contradicts to our assumption. Hence there is at most one $\delta_l$
which is not in $S$. Moreover this argument also shows that
its multiplicity in $\beta_{\Gamma}$ is $1$.
\end{proof}

We recall some facts on irreducible hermitian symmetric pairs.
See \cite[Subsection 1.6]{MR4585468} for brief description.

Let $\eps$ be a simple root whose multiplicity in the
highest root is $1$. This defines an involutive automorphism $\nu$
on $\g$ by $\id$ on $\frl_{\simples\setminus\{\eps\}}$
and $-\id$ on $\m_{\simples\setminus\{\eps\}}$. Then this involution
restricts to an involution on $\k$, whose fixed point part is $\k_{\simples\setminus\{\eps\}}$. This implies that
$\k_{\simples\setminus\{\eps\}} \subset \k$ is an hermitian symmetric pair.
If $\g$ is simple, \cite[Proposition 3.8]{MR4585468} implies that, for
any $\phi \in [-1,1]$, there exists an element $x_0 \in K$ such that
$\mathrm{pr}(\phi r) = \mathrm{pr}(\mathrm{Ad}_{x_0}(r))$.

\begin{proof}[Proof of Proposition \ref{prop:characterization of quotient Poisson structure} (ii) $\Longrightarrow$ (i)]
Take $\phi \in X_{K_S\backslash K}^{\quot}$ and assume
$\phi_{\alpha} \neq -1$ for $\alpha \in \roots^+\setminus\roots_S^+$. By Lemma \ref{lemm:reduction to symmetric spaces} and the discussion above,
we have $x_0 \in K$ such that
$v(\phi) = - \mathrm{pr}(\mathrm{Ad}_{x_0}(r))$. This means
that $\{x_0\}$ is a $0$-dimensional symplectic leaf of
$(K_S\backslash K, \pi_{\phi})$.

To prove the statement in general,
take $\phi \in X_{K_S\backslash K}^{\quot}$ and
consider the following subset:
\begin{align*}
 P
:= \{\alpha \in \roots\setminus\roots_S\mid \phi_{\alpha} > 0\}
\cup\{\alpha \in \roots^+\setminus\roots_S^+\mid \phi_{\alpha} = 0\}
\cup \roots_S^+.
\end{align*}
This is a parabolic subset in the sense of
\cite[Chapter VI, Section 1, Definition 4]{MR1890629}. Hence we can take a positive system
$\roots_0^+$ contained in $P$. Moreover, we can see the
following property of $\roots_S^+ \subset \roots_0^+$:
\begin{itemize}
 \item For $\alpha,\beta \in \roots_0^+$, $\alpha,\beta \in \roots_S^+$
if and only if $\alpha + \beta \in \roots_S^+$.
\end{itemize}
This implies that $S$ is contained in the set of simple roots of $\roots_0^+$.

Now take $w$ such that $w(\roots^+) = \roots_0^+$ and
set $S' := w^{-1}(S) \subset \simples$.
Consider the left multiplication $\map{\ell_{w^{-1}}}{K_S\backslash K}{K_{S'}\backslash K}$ defined by $K_Sx \longmapsto K_{S'}w^{-1}x$. Then
this is $K$-equivariant. Hence we have
\begin{align*}
 d\ell_{w^{-1}}(\pi_{\phi}) = (i\mathrm{Ad}_{w^{-1}}(v(\phi)))_R - (ir)_L.
\end{align*}
Moreover we have
\begin{align*}
 \mathrm{Ad}_{w^{-1}}(v(\phi)) = v(w^{-1}_*\phi),\quad
w^{-1}_*\phi = (\phi_{w(\alpha)})_{\alpha \in \roots\setminus\roots_{S'}}.
\end{align*}
Since $\phi_{w(\alpha)} \ge 0$ for all $\alpha \in \roots^+\setminus\roots_{S'}^+$, the discussion at the beginning implies that
$(K_{S'}\backslash K, \pi_{w^{-1}_*\phi})$
has a $0$-dimensional symplectic leaf. Then we see that
$(K_S\backslash K, \pi_{\phi})$ also has a $0$-dimensional
symplectic leaf since $\ell_{w^{-1}}$ preserves the Poisson structures.
\end{proof}

\subsection{The toric variety associated to a root system}

In this subsection, we recall the toric variety $X_{\roots}$
associated to the root system $\roots$ and also recall
how $X_{\lsbg}$ is embedded into $X_{\roots}$. For convenience
in later sections, we also give an embedding of $X_{\roots}$
into a product of projective lines.

All constructions can be carried out at the level of algebraic varieties,
but we restrict ourselves to description in terms of $k$-valued points,
which is sufficient for the present paper. See \cite[Subsection 5.2]{MR4837934}
for the embedding $X_{\lsbg} \subset X_{\roots}$ as algebraic varieties.

To define the set of $k$-valued points of $X_{\roots}$,
we would like to begin with the monoid algebra over $k$.
For a monoid $M$, the monoid algebra with coefficients in a commutative
ring $k$ is denoted by $k[M]$. It has a canonical $k$-basis
$\{e_{m}\}_{m \in M}$, for which we have $e_me_{m'} = e_{mm'}$.
Note that there is a canonical correspondence between
the set of monoid homomorphisms from $M$ to the multiplicative monoid $k$
and the set of $k$-algebra homomorphisms from $k[M]$ to $k$.
These sets are denoted by the same symbol $\Ch_k M$.

For an arbitrary positive system $\roots_0^+$, the corresponding
positive cone is denoted by $Q_0^+$. Then the set of $k$-valued
points of $X_{\roots}$ is defined as follows
\begin{align}
 X_{\roots}(k) := \left(\bigcup_{\roots_0^+} \Ch_k 2Q_0^+\right)\biggm/\sim, \label{eq:def of toric var}
\end{align}
where $\chi_1 \in \Ch_k 2Q_1^+$ and $\chi_2 \in \Ch_k 2Q_2^+$
are equivalent when there is $\chi \in \Ch_k (2Q_1^+ + 2Q_2^+)$
such that $\chi|_{2Q_i^+} = \chi_i$ for $i = 1,2$.

Note that $X_{\roots}(k)$ has an action of $W$, called the shifted
action on $X_{\roots}(k)$. For $\chi \in \Ch_k 2Q_0^+$ and $w \in W$,
$w\cdot\chi \in \Ch_k w(2Q_0^+)$ is defined as
\begin{align*}
 (w\cdot\chi)_{2\beta}
= q^{(w\rho - \rho,2\beta)}\chi_{w^{-1}(2\beta)}.
\end{align*}

Next we consider the $k$-valued points of the projective line, defined
as follows:

\begin{align*}
 \P^1(k) := (k^2\setminus\{(0,0)\})/\sim,
\end{align*}
where $(x_1,x_2)\sim (y_1,y_2)$ when $x_i = \lbd y_i$ for some $\lbd \in k$. The equivalence class containing $(x,y)$ is denoted by $[x:y]$
as usual. Let $\roots_0^+$ be a positive system.
Then $\chi \in \Ch_k 2Q_0^+$ defines the following element of
$\P^1(k)^{\roots}$, which is denoted by $\chi = \{\chi_{2\alpha}\}_{\alpha \in \roots}$ again:
\begin{align*}
 \chi_{2\alpha} =
\begin{cases}
 [\chi_{2\alpha}:1] & (\alpha \in \roots_0^+),\\
 [1:\chi_{-2\alpha}] & (\alpha \not\in \roots_0^+).
\end{cases}
\end{align*}
It is not difficult to check that this assignment is compatible
with the equivalence relation in (\ref{eq:def of toric var}).
Hence we have a map from $X_{\roots}(k)$ to $\P^1(k)^{\roots}$,
which is injective and has the following image.
\begin{align*}
 \{([x_{\alpha}:y_{\alpha}])_{\alpha \in \roots} \in \P^1(k)^{\roots}\mid[x_{-\alpha}:y_{-\alpha}] = [y_{\alpha}:x_{\alpha}],\,x_{\alpha}x_{\beta}y_{\alpha + \beta} = y_{\alpha}y_{\beta}x_{\alpha + \beta}\}.
\end{align*}

On the other hand, for any $\phi \in X_{\lsbg}(k)$,
we have an element $\chi_{\phi} \in \P^1(k)^{\roots}$ defined as follows:
\begin{align*}
 \chi_{\phi,2\alpha} =
\begin{cases}
 [\phi_{\alpha} + 1:\phi_{\alpha} - 1] & (\alpha \in \roots\setminus\roots_S),\\
 [1:1] & (\alpha \in \roots_S).
\end{cases}
\end{align*}
Then the conditions (\ref{rel1}), (\ref{rel2}), (\ref{rel3})
implies that $\chi_{\phi}$ is contained in the image
of $X_{\roots}(k)$, which allows us
to consider $\chi_{\phi}$ as an element of $X_{\roots}(k)$. Moreover
we can see that $\phi \longmapsto \chi_{\phi}$ gives an embedding
of $X_{\lsbg}(k)$ into $X_{\roots}(k)$.

For later use, we consider the following subset:
\begin{align*}
  X_{\ssorb}^{\circ}(k) &:= \{\phi \in X_{\ssorb}(k)\mid (\phi_{\alpha} + 1) \not\in (\phi_{\alpha} - 1)q_{\alpha}^{2\Z}\}.
\end{align*}

 \begin{lemm} \label{lemm:identification of toric and moduli}
The image of $X_{\ssorb}^{\circ}(k)$ under the embedding $X_{\ssorb}(k) \longrightarrow X_{\roots}(k)$ is the following:
\begin{align*}
 \{\chi \in X_{\roots}(k)\mid \chi_{2\alpha}\not\in q_{\alpha}^{2\Z}\text{ for }\alpha \notin \roots_S,\,\chi_{2\alpha} = 1\text{ for }\alpha \in \roots_S\}
\end{align*}
\end{lemm}
\section{The parabolic category $\cl{O}$ for deformed QEA} \label{sect:catO}
\subsection{Deformed quantum enveloping algebras} \label{ssect:dqea}
We recall deformed quantum enveloping algebras (deformed QEAs)
introduced in \cite[Definition 3.6]{MR4837934}, which enable
us to consider a certain limit of a Verma module twisted by a character
(Proposition \ref{lemm:comparison with the usual induction}).
To make the description consistent with literature,
\textbf{we give a definition slightly different from \cite{MR4837934}}.
We also avoid introducing an integral form of deformed QEAs for simplicity.

Let $U_{q,e}^+(\g)$ be a $k[2Q^+]$-subalgebra of $U_q^{k[P]}(\g) := k[P]\tensor U_q(\g)$ generated by
\begin{align*}
 \aE_i := E_i,\quad
 \aF_i := F_iK_i,\quad
 \aK_{\lbd} := e_{-\lbd}K_{\lbd}.
\end{align*}
with $1 \le i \le n$ and $\lbd \in 2P$.
This algebra is universal with respect to the following relations:
\begin{gather*}
 \aK_{\lbd}\aE_{\lbd} = q^{(\lbd,\eps_i)}\aE_{\lbd}\aK_{\lbd},\quad
 \aK_{\lbd}\aF_{\lbd} = q^{-(\lbd,\eps_i)}\aF_{\lbd}\aK_{\lbd},\quad
 [\aE_i,\aF_j]_q = \delta_{ij}\frac{e_{2\eps_i}\aK_i^2 - 1}{q_i - q_i^{-1}},\\
 \sum_{k = 0}^{1 - a_{ij}}(-1)^k\aE_i^{(k)}\aE_j\aE_i^{(1 - a_{ij} - k)} = 0,\quad
 \sum_{k = 0}^{1 - a_{ij}}(-1)^k\aF_i^{(k)}\aF_j\aF_i^{(1 - a_{ij} - k)} = 0.
\end{gather*}

Next take $w \in W$ and consider a $k$-algebra automorphism
$\map{t_w}{U_q^{k[P]}(\g)}{U_q^{k[P]}(\g)}$
defined by $t_w(e_{\lbd}\tensor x) = e_{w(\lbd)}\tensor \cl{T}_w(x)$.
Then $U_{q,e}^w(\g)$ is defined as $t_w^{-1}(U_{q,e}^+(\g))$.
Note that this is a $k[w^{-1}(2Q_+)]$-subalgebra of $U_q^{k[P]}(\g)$.

To give a generating set of $U_{q,e}^w(\g)$, we look at
quantum root vectors in $U_q(\g)$.
Consider a reduced expression $s_{\bs{i}}$ of $w_0$ which
begins with a reduced expression of $w^{-1}w_0$. Then we have
 that
$\alpha^{\bs{i}}_k \in w^{-1}(\roots^+)\cap \roots^+$
for $1 \le k \le l$ and
$\alpha^{\bs{i}}_k \in \roots^+\setminus w^{-1}(\roots^+)$
for $l < k \le N$, where $l$ is the length of $w^{-1}w_0$.
Take another reduced expression $s_{\bs{j}}$ of $w_0$
defined as follows:
\begin{align*}
 \eps_{j_k} :=
\begin{cases}
 -w_0(\eps_{i_{k + l}}) & (1 \le k \le N - l), \\
 \eps_{i_{k - (N - l)}} & (N - l < k \le N).
\end{cases}
\end{align*}
Then we have
\begin{align*}
 t_w(E_{\bs{i},k}) &=
\begin{cases}
 E_{\bs{j},k + (N - l)} & (1 \le k \le l), \\
 -F_{\bs{j},k - l}K_{\alpha^{\bs{j}}_{k - l}} & (l < k \le N),
\end{cases}\\
 t_w(F_{\bs{i},k}K_{\alpha^{\bs{i}}_k}) &=
\begin{cases}
 F_{\bs{j},k + (N - l)}K_{\alpha^{\bs{j}}_{k + (N - l)}} & (1 \le k \le l), \\
 -q^2_{\alpha^{j}_{k - l}}K_{\alpha^{\bs{j}}_{k - l}}^{-2}E_{\bs{j},k - l} & (l < k \le N),
\end{cases}\\
 t_w(K_{\lbd}) &= K_{w(\lbd)}.
\end{align*}
This implies the following elements of $U_{q,e}^w(\g)$ form
a generating set:
\begin{align*}
 \aE_{\bs{i},k} := E_{\bs{i},k},\quad
 \aK_{\lbd} := e_{-\lbd}K_{\lbd},\quad
 \aF_{\bs{i},k} :=
\begin{cases}
 F_{\bs{i},k}K_{\alpha^{\bs{i}}_k} & (1 \le k \le l),\\
 e_{-2\alpha^{\bs{i}}_k}F_{\bs{i},k}K_{\alpha^{\bs{i}}_k} & ( l < k \le N).
\end{cases}
\end{align*}
Moreover the same argument with \cite[Proposition 3.13]{MR4837934}
shows the PBW theorem for $U_{q,e}^w(\g)$, i.e.,
$\{\aF_{\bs{i}}^{\Lambda^-}K_{\lbd}\aE_{\bs{i}}^{\Lambda^+}\}_{\Lambda^{\pm},\lbd}$ is a basis of $U_{q,e}^w(\g)$.
Actually we only need simpler argument since we do not consider
the integral form of $U_{q,e}^w(\g)$.

The following lemma is crucial to construct
a left $\rep_q^{\fin} G$-module category.

\begin{lemm}[c.f. {\cite[Proposition 3.9]{MR4837934}}]
\label{lemm:left coaction}
For any $w \in W$, $U_{q,e}^w(\g)$ is a left coideal $k[w^{-1}(2Q^+)]$-subalgebra of $U_q^{k[P]}(\g)$.
\end{lemm}
Before the proof, we introduce the following completion of $U_q(\g)\tensor U_q(\g)$:
\begin{align*}
\cl{U}_q(\g\times\g) := \prod_{\lbd,\mu \in P^+} \mathrm{End}_k(L_{\lbd}\tensor L_{\mu}).
\end{align*}
Then we can embed $U_q(\g)\tensor U_q(\g)$ into $\cl{U}_q(\g\times \g)$
by its actions on $L_{\lbd}\tensor L_{\mu}$. Moreover the following
sum is well-defined in this algebra:
\begin{align*}
 \exp_{q_{\alpha}}((q_{\alpha} - q_{\alpha}^{-1})&K_{\alpha}^{-1}E_{\bs{j},\alpha}\tensor F_{\bs{j},\alpha}K_{\alpha})\\
&:= \sum_{n = 0}^{\infty} \frac{q_{\alpha}^{n(n - 1)/2}}{[n]_{q_{\alpha}}!}(q_{\alpha} - q_{\alpha}^{-1})^n(K_{\alpha}^{-1}E_{\bs{j},\alpha}\tensor F_{\bs{j},\alpha}K_{\alpha})^n.
\end{align*}
\begin{proof}[Proof of Lemma \ref{lemm:left coaction}]
The case of $w = 1$ can be confirmed directly, using the generating
set.

Take $w \in W$ arbitrary and consider the reduced expressions
$s_{\bs{i}}$ and $s_{\bs{j}}$ as above. Let $A_w$ be the following
element of $\cl{U}_q(\g\times\g)$.
\begin{align*}
 A_w := \prod_{k = 1}^{N - l}\exp_{q_{j_k}}((q_{j_k} - q_{j_k}^{-1})K_{\alpha^{\bs{j}}_k}^{-1}E_{\bs{j},k}\tensor F_{\bs{j},k}K_{\alpha^{\bs{j}}_k}).
\end{align*}
Then the following formula holds (\cite[Proposition 3.81]{MR4162277}):
\begin{align*}
 \Delta(\cl{T}_w(x)) = A_w(\cl{T}_w\tensor \cl{T}_w)\Delta(x)A_w^{-1}.
\end{align*}
Equivalently the following formula also holds:
\begin{align*}
 \Delta(\cl{T}_w^{-1}(x)) = B_w^{-1}(\cl{T}_w^{-1}\tensor \cl{T}_w^{-1})\Delta(x)B_w
\end{align*}
where
\begin{align*}
 B_w = \prod_{k = l + 1}^N \exp_{q_{i_k}}((q_{i_k} - q_{i_k}^{-1})F_{\bs{i},k}\tensor E_{\bs{i},k}).
\end{align*}
Now the statement follows.
\end{proof}
The deformed quantum enveloping algebra is now defined as an evaluation of $U_{q,e}^w(\g)$
 by a character on $w^{-1}(2Q^+)$.
\begin{defn}
Let $\roots_0^+$ be a positive system and $Q_0^+$ be a submonoid
generated by $\roots_0^+$. For a character $\map{\chi}{2Q_0^+}{k}$,
we define $U_{q,\chi}(\g)$ as $k\tensor_{k[2Q_0^+]} U_{q,e}^{w}(\g)$,
where $w \in W$ is the unique element satisfying $w^{-1}(\roots^+) = \roots_0^+$.
\end{defn}
\begin{rema}
We have a natural generating set
$\aE_{\bs{i},k},\aK_{\lbd},\aF_{\bs{i},k}$ of $U_{q,\chi}(\g)$
induced from those in $U_{q,e}^{w}(\g)$
and can see that the PBW theorem holds for $U_{q,\chi}(\g)$,
i.e., $\{\aF_{\bs{i}}^{\Lambda^-}\aK_{\lbd}\aE_{\bs{i}}^{\Lambda^+}\}_{\Lambda^{\pm},\lbd}$ is a basis of $U_{q,\chi}(\g)$.
\end{rema}
\begin{rema} \label{rema:independence of the domain}
Let $\roots_0^+$ and $\roots_1^+$ be positive systems of $\roots^+$
and $\chi$ be a $k$-valued character on $2Q_0^+ + 2Q_1^+$.
Then we have a natural identification of $U_{q,\chi|_{2Q_0^+}}(\g)$
and $U_{q,\chi|_{2Q_1^+}}(\g)$ as left $U_q(\g)$-comodule algebras,
induced from $k[2Q_0^+ + 2Q_1^+]U_{q,e}^{w_0}(\g) = k[2Q_0^+ + 2Q_1^+]U_{q,e}^{w_1}(\g)$ as a subalgebra of $U_q^{k[P]}(\g)$.
See \cite[Proposition 5.9]{MR4837934} for detail.
 This identification allows us to interpret $\chi$ as an
element of $X_{\roots}(k)$.
\end{rema}

We consider the following subalgebras of $U_{q,\chi}(\g)$:
\begin{itemize}
 \item $U_{q,\chi}(\ah)$, generated by $(\aK_{\mu})_{\mu \in 2P}$.
 \item $U_{q,\chi}(\frb)$, generated by $(\aK_{\mu})_{\mu \in 2P}, (\aE_{\bs{i},\alpha})_{\alpha \in \roots^+}$.
 \item $U_{q,\chi}(\n^+)$ (resp. $U_{q,\chi}(\n^-)$), generated by $(\aE_{\bs{i},\alpha})_{\alpha \in \roots^+}$ (resp. $(\aF_{\bs{i},\alpha})_{\alpha \in \roots^+}$).
\end{itemize}
By the PBW theorem, these give decompositions as follows:
\begin{align*}
 U_{q,\chi}(\g) &\cong U_{q,\chi}(\n^-)\tensor U_{q,\chi}(\frb),\\
 U_{q,\chi}(\frb) &\cong U_{q,\chi}(\ah)\tensor U_{q,\chi}(\n^+).
\end{align*}
Note that these subalgebras and decompositions are preserved under
the identification in Remark \ref{rema:independence of the domain}.

Finally we give a comparison of a deformed quantum enveloping
algebra and the usual quantum enveloping algebra.

\begin{lemm} \label{lemm:comparison with the usual qea}
Let $\map{\chi}{P}{k^{\times}}$ be a character.
Then $U_{q,\chi|_{2Q_0^+}}(\g)$ has a canonical embedding into
$U_q(\g)$ as a left $U_q(\g)$-comodule algebra:
\begin{align*}
 \aE_{\bs{i},k}\longmapsto E_{\bs{i},k}, \quad
 \aK_{\lbd}\longmapsto \chi_{-\lbd}K_{\lbd},\quad
 \aF_{\bs{i},k}\longmapsto 
\begin{cases}
 F_{\bs{i},k}K_{\alpha^{\bs{i}}_k} & (1 \le k \le l),\\
 \chi_{-2\alpha^{\bs{i}}_k}F_{\bs{i},k}K_{\alpha^{\bs{i}}_k} & ( l < k \le N).
\end{cases}
\end{align*}
\end{lemm}
\begin{rema}
Note that this map is not surjective since we restrict the indices
of the Cartan part to $2P$, not $P$. This yields some
differences between the module theory of $U_{q,\chi}(\g)$
and the module theory of $U_q(\g)$, as the weight space decomposition
with respect to $U_{q,\chi}(\g)$ can be different from that of $U_q(\g)$.
At least in the present paper, this difference is convenient. It
makes the theory of Verma modules simple
and suitable to our objective, constructing semisimple
actions of $\fssorb$-type.
See \cite[Subsection 3.13]{MR4162277} for the description
on this point, especially the linkage class in the usual setting.
\end{rema}

\subsection{The category $\catO{\chi}$}

In this subsection we would like to investigate the category
$\catO{\chi}$. Note that $U_{q,\chi}(\h)$ and $U_{q,\chi}(\n^+)$
allow us to consider the notion of weight, weight space and
highest weight vector for $U_{q,\chi}(\g)$-modules.

In the following, the Cartan part of $U_{q,\chi}(\h)$,
which is independent of $\chi$, is denoted by $U_q(\ah)$.
By definition it is isomorphic to $k[2P]$.
Hence the set of weights with respect to the action of $U_q(\ah)$
is $\Ch_k 2P$, which is denoted by $\ah_q^*$ in the following.

\begin{defn}
Let $U_{q,\chi}(\g)$ be a defomed quantum enveloping algebra.
The category $\catO{\chi}$ is the full subcategory of $\mod{U_{q,\chi}(\g)}$ whose objects are all of $U_{q,\chi}(\g)$-module $M$
satisfying the following conditions:
\begin{enumerate}
 \item $M$ is finitely generated as a $U_{q,\chi}(\g)$-module.
 \item The action of $U_q(\ah)$ on $M$ is semisimple i.e. it admits
a weight space decomposition.
 \item For any $m \in M$, $U_{q,\chi}(\n^+)m$ is finite dimensional.
\end{enumerate}
\end{defn}

As same with the usual category $\cl{O}_q$, the category
$\catO{\chi}$ is abelian. Also note that it is canonically
a left $\rep_q^{\fin} G$-module category by
the left $U_q(\g)$-comodule algebra structure on $U_{q,\chi}(\g)$.

\begin{defn}
For any $\chi \in X_{\roots}(k)$, the \emph{$\chi$-shifted induction functor}
$\dqind$ is defined as
$\map{U_{q,\chi}(\g)\tensor_{U_{q,\chi}(\frb)}\tend}{\mod{U_q(\ah)}}{\mod{U_{q,\chi}(\g)}}$.
\end{defn}

\begin{exam}
Let $\lambda$ be a character on $U_q(\ah)$ and
$k_{\lbd}$ be the corresponding $1$-dimensional representation.
Then $M_{\chi}(\lbd) := \mathrm{ind}_{\frb, q}^{\g,\chi}k_{\lbd}$
is an object of $\catO{\chi}$, which is called a \emph{$\chi$-shifted Verma module} with highest weight $\lbd$.
\end{exam}

For a character $\map{\chi}{P}{k^{\times}}$, we have the following
comparison with the usual induction functor
$\qind{} := U_q(\g)\tensor_{U_q(\frb)}\tend$.
This enables us to extend the known results on the category $\cl{O}_q$
for $U_q(\g)$ to the category $\catO{\chi}$.

\begin{lemm} \label{lemm:comparison with the usual induction}
Let $\chi$ be a character on $P$ and $V$ be a $U_q(\h)$-module.
Under the isomorphism
in Lemma \ref{lemm:comparison with the usual qea}, we have the following
natural isomorphism as $U_{q,\chi}(\g)$-modules:
\begin{align*}
 \dqind V &\cong \qind(V\tensor \C_{\chi}), \\
1\tensor v &\longmapsto 1\tensor (v\tensor 1).
\end{align*}
\end{lemm}
\begin{proof}
The statement follows from the universal property.
\end{proof}

We analyze the category $\catO{\chi}$ by the standard
argument. At first we show the Harish-Chandra theorem on
the center $ZU_{q,\chi}(\g)$.

\begin{prop}
Let $\map{P}{U_{q,\chi}(\g)}{U_{q}(\ah)}$ be the projection
along with the triangular decomposition $U_{q,\chi}(\g)= U_{q,\chi}(\n^-)U_{q}(\ah)U_{q,\chi}(\n^+)$. Then this is an injective algebra homomorphism on $ZU_{q,\chi}(\g)$ with the following image:
\begin{align*}
 \mathrm{span}_k\bigg\{\sum_{\tilde{w} \in W} q^{(\rho,\mu - \tilde{w}\mu)}\chi_{w^{-1}\mu - \tilde{w}\mu}\aK_{-\tilde{w}\mu}\bigg\}_{\mu \in 2P^+},
\end{align*}
where $w$ is the element of $W$
satisfying $w(\roots_0^+) = \roots^+$.
\end{prop}
\begin{proof}
Since the $\chi$-shifted Verma modules distinguish elements of $U_{q,\chi}(\g)$, the homomorphism is injective.

To determine the image, we assume $\chi \in \Ch_k 2Q^+$
at first. In this case we have $w = 1$. Recall the adjoint action
of $U_q(\g)$ on $U_q(\g)$, given by $x \triangleright y = x_{(1)}yS(x_{(2)})$. Then it is not difficult to see that this induces an action of
$U_q(\g)$ on $U_{q,e}^+(\g)$, which is also denoted by $\tend\triangleright \tend$. Now fix $\mu \in 2P^+$. By the discussion in the proof
of \cite[Theorem 8.6]{MR1198203}, there is $x_{\mu} \in U_q(\g)$ such that $x_{\mu}\triangleright K_{-\mu}$ is central and whose
image under $P$ is
\begin{align*}
\sum_{\tilde{w} \in W} q^{(\rho,\mu - \tilde{w}\mu)}K_{-\tilde{w}\mu}.
\end{align*}
Then $x_{\mu}\triangleright \aK_{-\mu} \in U_{q,e}^+(\g)$ is also
central and its image under $P$ is
\begin{align*}
 \sum_{\tilde{w} \in W} q^{(\rho,\mu - \tilde{w}\mu)}e_{\mu - \tilde{w}\mu}\aK_{-\tilde{w}\mu}.
\end{align*}
By evaluating $e$ by $\chi$, we can see that the image of the homomorphism
 contains
\begin{align*}
 \sum_{\tilde{w} \in W} q^{(\rho,\mu - \tilde{w}\mu)}\chi_{\mu - \tilde{w}\mu}\aK_{-\tilde{w}\mu}
\end{align*}
for all $\mu \in 2P^+$. To see that these elements span the
image, we consider the subalgebra
$U_{q,\chi}(\frl_i)$ generated by $\aE_i, \aF_i$ and $U_{q,\chi}(\h)$.
Then the quantum PBW bases defines a projection $\map{P_i}{U_{q,\chi}(\g)}{U_{q,\chi}(\frl_i)}$, which does not depend on
the choice of quantum root vectors.
Then we can see that
$P_i(ZU_{q,\chi}(\g)) \subset ZU_{q,\chi}(\frl_i)$ and
$P = P\circ P_i$ on $U_{q,\chi}(\g)$. Now direct computation shows
that the image of $ZU_{q,\chi}(\frl_i)$ under $P$ is
generated by $\{\aK_{2\varpi_j}\}_{j \neq i}$ and $\aK_{-2\varpi_i} + q_i^2\chi_{2\eps_i}\aK_{-s_i(2\varpi_i)}$. Hence we have
\begin{align*}
 P(ZU_{q,\chi}(\frl_i)) = \mathrm{span}_k\{\aK_{-\mu} + q^{(\rho,\mu - s_i(\mu))}\chi_{\mu - s_i(\mu)}\aK_{-s_i(\mu)} \mid \mu \in 2P, (\mu, \eps_i^{\vee}) \ge 0\}.
\end{align*}
Since $P(ZU_{q,\chi}(\g)) \subset \bigcap_i P(ZU_{q,\chi}(\frl_i))$,
we obtain the statement.

For general $\chi \in \Ch_k 2Q_0^+$, take $w \in W$ so that $w(\roots_0^+) = \roots^+$. Then we have an isomorphism
$\map{t_w}{U_{q,\chi}(\g)}{U_{q,w\chi}(\g)}$
induced from $\map{t_w}{U_{q,e}^w(\g)}{U_{q,e}^+(\g)}$.
Hence we also have an isomorphism
$ZU_{q,\chi}(\g)\cong ZU_{q,w\chi}(\g)$.
To reduce the statement for $\chi$ to the statement for $w\chi$,
which is proven by the discussion above, it suffices to show that
the following diagram is commutative:
\begin{align*}
 \xymatrix{
ZU_q(\g) \ar[r]^{\cl{T}_w} \ar[d]_-P & ZU_q(\g) \ar[d]^-P \\
U_q(\h) \ar[r]^{\cl{T}_w} & U_q(\h).
}
\end{align*}
This follows from that $\cl{T}_w$ is implemented on each finite dimensional representation and that finite dimensional representations
distinguish elements of $U_q(\g)$
\end{proof}

The value of $\lbd \in \ah_q^*$ at $K_{\mu}$is denoted by
$\chi_{\lbd}(K_{\mu}) = q^{(\lbd, \mu)}$.
The product in $\ah_q^*$ is written as an addition i.e. $\chi_{\lbd}\chi_{\lbd'} = \chi_{\lbd + \lbd'}$ and $q^{(\lbd,\mu)}q^{(\lbd',\mu)} = q^{(\lbd + \lbd',\mu)}$. We also use $q^{-(\lbd,\mu)} := (q^{(\lbd,\mu)})^{-1} = q^{(\lbd,-\mu)} = q^{(-\lbd,\mu)}$.

Note that there is an embedding of $P$
into $\ah_q^*$,
using the inner product on $\h_{\R}^*$ and $r \in (2L)^{-1}\longmapsto q^r \in k^{\times}$.
Then a partial order on $\ah_q^*$ is defined by $\lbd \le \lbd'$
if and only if $\lbd' - \lbd \in Q^+$.

To describe the linkage class in our setting, we introduce some
notations. For $\chi \in X_{\roots}(k)$, we define
$\roots_{\chi} := \{\alpha \in \roots\mid \chi_{2\alpha} \neq 0,\infty\}$.
Then we have $\roots_{\chi} = \roots\cap \aspan_{\Z} \{\alpha \in \roots\mid \chi_{2\alpha} \neq 0,\infty\}$, which means that $\roots_{\chi}$ is
a root system with the Weyl group $W_{\chi} \subset W$ generated by
$\{s_{\alpha}\mid \alpha \in \roots, \chi_{2\alpha}\neq 0,\infty\}$.
The root lattice associated to $\roots_{\chi}$ is denoted by $Q_{\chi}$.

\begin{defn}
The \emph{$\chi$-shifted action} of $W_{\chi}$ on $\ah_q^*$ is
defined by
\begin{align*}
 q^{(w\cdot_{\chi}\lbd,\mu)}:= \chi_{w^{-1}\mu - \mu}q^{(\rho,w^{-1}\mu - \mu)}q^{(\lbd,w^{-1}\mu)},
\end{align*}
where $\chi$ is extended to a character on $2Q_{\chi}$.
When $\chi$ is the trivial character, we simply say
the \emph{shifted action}.
\end{defn}

Note $q^{(\lbd - s_{\alpha}\cdot_{\chi}\lbd,\mu)} = (\chi_{2\alpha}q^{(\rho + \lbd,2\alpha)})^{(\alpha^{\vee}/2,\mu)}$ for $\alpha \in \roots_{\chi}$.

When $\chi$ is a character on $2P$, we have the following comparison
with the usual shifted action of $W$.

\begin{lemm}
Let $\chi$ be a character on $2P$. In this case
$\roots_{\chi} = \roots$ and $W_{\chi} = W$.
Moreover the $\chi$-shifted action
$W_{\chi}\curvearrowright \ah_q^*$ is isomorphic to
the usual shifted action $W \curvearrowright \ah_q^*$ via
$\lbd \longmapsto \lbd + \chi$.
\end{lemm}

The objective of this section is to determine $\chi \in X_{\roots}(k)$
such that the integral part of the category $\catO{\chi}$
is semisimple. As expected, this involves the shifted version of
dominancy and antidominancy.

\begin{defn} \label{defn:shifted dominancy}
 We say that $\lbd \in \ah_q^*$ is \emph{$\chi$-dominant} (resp. \emph{$\chi$-antidominant}) when
$\lbd$ is maximal (resp. minimal) in $W_{\chi}\cdot_{\chi}\lbd$.
\end{defn}

We have a characterization analogous to
\cite[Proposition 3.5]{MR2428237}
(c.f. \cite[Proposition 5.7, Proposition 5.8]{MR4162277}
for the quantum group version).

\begin{lemm} \label{lemm:characterization of dominancy}
For $\lbd \in \ah_q^*$, the following conditions are equivalent:
\begin{enumerate}
 \item The element $\lbd$ is maximal (resp. minimal) in $W_{\chi}\cdot_{\chi}\lbd$.
 \item $q^{(\lbd + \rho,2\alpha)}\chi_{2\alpha} \not\in q_{\alpha}^{2\Z_{<0}}$ (resp. $q_{\alpha}^{2\Z_{>0}}$) for all $\alpha \in \roots_{\chi}^+ := \roots_{\chi}\cap\roots^+$.
\end{enumerate}
\end{lemm}

To prove this lemma, we need the following variation of
\cite[Satz 1.3]{MR552943}.

\begin{lemm} \label{lemm:integral roots}
Let $\roots \subset E$ be a root system and $P$ be the weight lattice. Let $\lbd$ be a $k^{\times}$-valued character on $2P$.
We define $R_{[\lbd]}$ and $W_{[\lbd]}$
as follows:
\begin{align*}
 R_{\lbd} := \{\alpha \in \roots\mid \lbd_{2\alpha} \in q_{\alpha}^{2\Z}\},\quad
 W_{\lbd} := \{w \in W\mid w\lbd - \lbd \in Q\}.
\end{align*}
Then $R_{\lbd}$ is a root system, whose Weyl group is
 $W_{\lbd}$.
\end{lemm}
\begin{proof}
By consider the image of $\lbd$, we may assume that $k$ is finitely
generated over $\Q$ as a field. Then we can embed $k$ into $\C$.
Hence it suffices to show the statement when $k = \C$.

 Set $E_{\C}:= E\tensor_{\R}\C$ and
take $h \in \C$ so that $q = \exp(i\pi h)$. Note that $h \not \in \Q$
since $q$ is not a root of unity.

Let $\Gamma$ be the set of $\C$-valued characters on $2P$. Then we can identify $E_{\C}/h^{-1}Q^{\vee}$ with $\Gamma$ via $[\mu] \longmapsto \exp(i\pi h(\mu,\tend))$, where $Q^{\vee}$ is the coroot lattice.
In this picture, it is convenient to consider a basis $(e_i)_{i \in I}$ of $\C$ over $\Q$ such that $e_0 = 1$ and $e_1 = h^{-1}$.
Let $E_{\Q}$ be the $\Q$-linear span of $R$. Then we have the following presentation of $\mu$:
\begin{align*}
 \mu = \sum_{i \in I} e_i\mu_i,\quad \mu_i \in E_{\Q}.
\end{align*}
Then $R_{[\mu]}$ and $W_{[\mu]}$ are presented as follows:
\begin{align*}
 R_{[\mu]} &= \{\alpha \in R\mid (\mu_0,\alpha^{\vee}) \in \Z,\,(\mu_1,\alpha) \in \Z,\,(\mu_i,\alpha^{\vee}) = 0\},\\
 W_{[\mu]} &= \{w \in W\mid w\mu_0 - \mu_0 \in Q,\,w\mu_1 - \mu_1 \in Q^{\vee},w\mu_i - \mu_i = 0\}.
\end{align*}
Consider
\begin{align*}
 R' &= \{\alpha \in R\mid (\mu_0,\alpha^{\vee}) \in \Z,\,(\mu_i,\alpha^{\vee}) = 0\},\\
 W' &= \{w \in W\mid w\mu_0 - \mu_0 \in Q,w\mu_i - \mu_i = 0\}.
\end{align*}
Then the proof of \cite[Satz 1.3]{MR552943} implies $R'$ is a root system with
the Weyl group $W'$. By considering the dual root system of $R'$
and the orthogonal decomposition $\mu_1 = \mu_1' + \mu_1''$
according to $E = \R R'\oplus (\R R')^{\perp}$,
another application of the discussion in
\cite[Satz 1.3]{MR552943} proves the statement.
\end{proof}

\begin{proof}[Proof of Lemma \ref{lemm:characterization of dominancy}]
It is not difficult to see (i) $\Longrightarrow$ (ii).
To see the converse, we replace $k$ by its algebraic closure and
extend $\chi|_{2Q_{\chi}}$ to a character $\chi'$ on $P$.
Then $\lbd \longmapsto \lbd + \chi'$ gives an isomorphism
from the $\chi$-shifted action
$W_{\chi}\curvearrowright \ah_q^*$ to the restriction of
the shifted action
$W \curvearrowright \ah_q^*$.

Let $P_{\chi}$ be the weight lattice of $\roots_{\chi}$ and
$\rho_{\chi}$ be the half sum of $\roots_{\chi}$. Since
$\roots_{\chi}$ is a closed subsystem generated by simple
roots of a positive system, we have the canonical map $\map{\pi}{P}{P_{\chi}}$ and $\map{i}{P_{\chi}}{P}$ with $\pi\circ i = \id$.
Then $\lbd \longmapsto \lbd' = (\lbd + \rho)\circ i|_{2P_{\chi}} - \rho_{\chi}$ preserves the shifted action of $W_{\chi}$.

Now the assumption implies that $\lbd'$ satisfies
$q^{(\lbd' + \rho_{\chi},2\alpha)}\chi_{2\alpha} \not\in q_{\alpha}^{2\Z_{<0}}$ for all $\alpha \in \roots_{\chi}^+$.
Hence the discussion in
\cite[Proposition 3.5]{MR2428237}, after replacing
\cite[Theorem 3.4]{MR2428237} by
Lemma \ref{lemm:integral roots}, implies that $\lbd'$ is maximal in
$W_{\chi}\cdot\lbd'$. This proves (ii) $\Longrightarrow$ (i).
\end{proof}

Now we give the sufficient conditions for projectivity and simplicity.
We omit the proof since the usual argument can be applied.
See \cite[Proposition 3.8]{MR2428237} for example.

\begin{prop} \label{prop:dominancy implies projectivity}
For $\lbd \in \ah_q^*$, the following statements hold:
\begin{itemize}
 \item If $\lbd$ is $\chi$-dominant, $M_{\chi}(\lbd)$ is projective.
 \item If $\lbd$ is $\chi$-antidominant, $M_{\chi}(\lbd)$ is simple.
\end{itemize}
\end{prop}
\begin{rema}
The converse direction is also likely to be true, but we do not pursue the argument here since it plays no role in the present paper.
\end{rema}

Next we proceed to the characterization of the simplicity.

\begin{defn}
We say that $\lbd \in \ah_q^*$ is $\chi$-strongly linked to $\lbd' \in \ah_q^*$, denoted by $\lbd \uparrow_{\chi} \lbd'$, if there is a sequence $\alpha_1,\alpha_2,\cdots,\alpha_k$
in $\roots_{\chi}$ with the following condition:
\begin{align*}
 \lbd = s_{\alpha_k}s_{\alpha_{k - 1}}\cdots s_{\alpha_1}\cdot_{\chi}\lbd' < s_{\alpha_{k - 1}}\cdots s_{\alpha_1}\cdot_{\chi}\lbd'
< \cdots < s_{\alpha_1}\cdot_{\chi}\lbd' < \lbd'.
\end{align*}
\end{defn}

The following is a variant of Verma's theorem in our setting.

\begin{prop} \label{prop:Verma theorem}
For $\lbd,\lbd' \in \ah_q^*$ such that $\lbd$ is
$\chi$-strongly linked to $\lbd'$,
there is an embedding $M_{\chi}(\lbd) \longrightarrow M_{\chi}(\lbd')$.
\end{prop}


The injectivity of the map is due to the following lemma.
Again we omit the proof since the discussion in
\cite[Proposition 3.134]{MR4162277} can be applied.

\begin{lemm}
There exists no zero-divisor in $U_{q,\chi}(\n^-)$.
\end{lemm}

\begin{proof}[Proof of Proposition \ref{prop:Verma theorem}]
We may assume that $k$ is algebraically closed.
In the case that $k$ is not algebraically closed,
we consider the base change by its algegraic closure.

At first we show the statement for $\chi \in \Ch_k 2Q$.
By our assumption on $k$, we can extend $\chi$ to a character on $P$,
which is also denoted by $\chi$. Then Lemma \ref{lemm:comparison with the usual qea} and Lemma \ref{lemm:comparison with the usual induction} reduces the existence of an embedding $M_{\chi}(\lambda) \subset M_{\chi}(\lbd')$ to the existence of
an embedding $M(\lbd + \chi) \subset M(\lbd' + \chi)$
after extending $\lbd$ and $\lbd'$ to characters on $P$
so that $\lbd + \chi$ is strongly linked to $\lbd'$.
The latter is a conclusion of Verma's theorem
(\cite[Theorem 5.14]{MR4162277}).

Now fix $\alpha \in \roots_0^+, c \in k^{\times}$ and $n \in \Z$ so that $n\alpha \in Q^+\setminus\{0\}$.
Take $\lbd' \in \ah_q^*$ satisfying $cq^{(\rho + \lbd',2\alpha)} = q_{\alpha}^{2n}$. Then we have $s_{\alpha}\cdot_{\chi} \lbd' = \lbd' - n\alpha < \lbd'$ for $\chi$ in the following algebraic subset:
\begin{align*}
A_{\alpha,c} := \{\map{\chi}{2Q_0^+}{k}\mid \chi_{2\alpha} = c.\}
\end{align*}
We can also see that existence of a heighest weight vector in
$M_{\chi}(\lbd')_{\lbd' - n\alpha}$ is an algebraic condition on
$\chi$
since it is equivalent to non-injectivity of the following map,
where $U_{q,1}(\n^-)$ is identified with $U_{q,\chi}(\n^-)\cong M_{\chi}(\lbd')$ through the PBW basis $\{\aF^{\Lambda}\}_{\Lambda}$:
\begin{align*}
 U_{q,1}(\n^-)_{- n\alpha} \cong M_{\chi}(\lbd')_{\lbd' - n\alpha} &\longrightarrow
\bigoplus_{\eps \in \simples} M_{\chi}(\lbd')_{\lbd' - n\alpha + \eps}\cong \bigoplus_{\eps \in \simples} U_{q,1}(\n^-)_{-n\alpha + \eps},\\
x &\longmapsto (E_{\eps}x)_{\eps \in \simples}.
\end{align*}
Hence the discussion in the case that $\chi_{2\alpha} \neq 0$ for all
$\alpha \in \roots_0^+$ implies the existence of $M_{\chi}(s_{\alpha}\cdot \lbd') \subset M_{\chi}(\lbd')$
for all $\chi \in A_{\alpha,c}$. This concludes the statement
since we consider all possible choices of $(\alpha,c,n)$.
\end{proof}

To show the BGG theorem, we consider the following pairing.

\begin{defn}
A pairing
$\map{\cl{S}}{U_{q,\chi}(\n^+)\times U_{q,\chi}(\n^-)}{U_q(\ah)}$ is defined by $S(y,x) := P(yx)$. Its restriction on $U_{q,\chi}(\n^+)_{\nu}\times U_{q,\chi}(\n^-)_{-\nu}$
is denoted by $\cl{S}_{\nu}$
\end{defn}

\begin{lemm} \label{lemm:null space of Shapovalov pairing}
Take $\lbd \in \ah_q^*$. For any $x \in U_{q,\chi}(\n^-)_{-\nu}$,
$\chi_{\lbd}(\cl{S}_{\nu}(\cdot,x)) = 0$ if and only if
$x\tensor 1 \in M_{\chi}(\lbd)$ is contained in
a proper submodule of $M_{\chi}(\lbd)$.
\end{lemm}
\begin{proof}
Note that the assumption on $x$ implies
$\chi_{\lbd}(\cl{S}(y,x)) = 0$ for all $y \in U_{q,\chi}(\n^+)$,
which is equivalent to
$(U_{q,\chi}(\n^+)x\tensor 1)_{\lbd} = 0$.
Now the statement follows since
$(U_{q,\chi}(\g)(x\tensor 1))_{\lbd} = (U_{q,\chi}(\n^+)(x\tensor 1))_{\lbd} = 0$.
\end{proof}
\begin{prop} \label{lemm:Shapovalov determinant}
Fix a basis of $U_{q,\chi}(\n^{\pm})_{\pm \nu}$ and
consider the matrix presentation of $\cl{S}_{\nu}$ and
its determinant $\det \cl{S}_{\nu}$. Then this is
a product of an invertible element of $U_{q}(\ah)$
and the following element:
\begin{align*}
 \prod_{\beta \in \roots^+ \cap \roots_0^+}\prod_{m = 1}^{\infty}&
(q_{\beta}^{2(\rho,\beta^{\vee})}\chi_{2\beta}\aK_{2\beta} - q_{\beta}^{2m})^{P(\nu - m\beta)} \\
&\times
 \prod_{\beta \in \roots^+ \setminus \roots_0^+}\prod_{m = 1}^{\infty}
(q_{\beta}^{2(\rho,\beta^{\vee})}\aK_{2\beta} - q_{\beta}^{2m}\chi_{-2\beta})^{P(\nu - m\beta)}.
\end{align*}
\end{prop}
\begin{proof}
Take $w \in W$ so that $w^{-1}(\roots^+) = \roots_0^+$. Note that the paring is well-defined for
$U_{q,e}^w(\g)$

Consider the PBW basis of $U_{q,e}^{w}(\n^{\pm})$. Then Lemma
\ref{lemm:comparison with the usual qea} and
the corresponding statement for $U_q(\g)$ (\cite[Theorem 5.22]{MR4162277}) implies that
$\det \cl{S}_{\nu} \in U_{q,e}^w(\h)$
is divided by the factor above. Moreover
we can see that
the remaining factor is a scalar multiple of $\aK_{\mu}$ for some
$\mu \in 2P$ since
$U_{q,e}^w(\h)^{\times} = \cup_{\mu \in 2P}k^{\times}\aK_{\mu}$.
Since $\det \cl{S}_{\nu} \in U_{q}(\ah)$ for $\chi \in \Ch_k 2Q_0^+$
is the evaluation of $\det \cl{S}_{\mu} \in U_{q,e}^w(\h)$ by $\chi$,
it suffices to show that $\det \cl{S}_{\nu} \neq 0$.

Assume $\det \cl{S}_{\nu} = 0$. Then $M_{\chi}(\lbd)$
is not simple for all $\lbd \in \ah_q^*$ by Lemma \ref{lemm:null space of Shapovalov pairing}.
This contradicts to the simplicity of $M_{\chi}(\lbd)$
for $\chi$-antidominant $\lbd$.
\end{proof}

Now the usual argument yields the BGG theorem from this proposition.
We consider the following slight variation of \cite[Lemma 5.1]{MR552943}:

In the following lemma, $A$ is a PID and $p \in A$ is
a prime element. The residue field $A/pA$ is denoted by $\bar{A}$
and $M\tensor_A \bar{A}$ is denoted by $\bar{M}$.

We also define $\map{v_p}{A}{\Z_{\ge 0}\cup\{\infty\}}$ by $v_p(a) = n$ for $a \in p^nA\setminus p^{n + 1}A$
and $v_p(0) = \infty$.

\begin{lemm}
Let $M, N$ be a free $A$-module of rank $r$ and
$\map{(\cdot,\cdot)}{M\times N}{A}$ be a bilinear map.
Set $D$ be the determinant of a representation matrix of
$(\cdot,\cdot)$ with respect to bases of $M,N$.

Then $M(k) := \{m \in M\mid (m, N) \subset p^kA\}$
and $N(k) := \{n \in N\mid (M,n) \subset p^kA\}$
define descending filtrations
$\{M(n)\}_{n = 0}^{\infty}$ and $\{N(n)\}_{n = 0}^{\infty}$,
which satisfy the following properties:

\begin{enumerate}
 \item $M(0) = M, N(0) = N$.
 \item $v_p(D) = \sum_{k =1}^{\infty} \dim_{\bar{A}} \bar{M(k)} = \sum_{k = 1}^{\infty} \dim_{\bar{A}}\bar{N(k)}$.
 \item The $A$-valued bilinear map $(m,n)_k := p^{-k}(m,n)$ is well-defined
and induces an non-degenerate $\bar{A}$-valued bilinear map
on $\bar{M(k)}/\bar{M(k + 1)}\times \bar{N(k)}/\bar{N(k + 1)}$.
\end{enumerate}
\end{lemm}

We show BGG theorem in our setting.

\begin{prop} \label{prop:BGG theorem}
For $\lbd,\lbd' \in \h_q^*$, $[M_{\chi}(\lbd')\colon L_{\chi}(\lbd)] \neq 0$ implies $\lbd \uparrow_{\chi} \lbd'$.
\end{prop}
\begin{proof}
By the usual argument, we obtain a descending
filtration $\{M(k)\}_{k = 0}^{\infty}$ of $M_{\chi}(\lbd')$
with the following properties:
\begin{enumerate}
 \item $M(k) = M_{\chi}(\lbd')$.
 \item The submodule $M(1)$ is the maximal proper submodule of $M_{\chi}(\lbd')$.
 \item Set $A = k[T]$ and $p = T - 1$.
For any $\nu \in Q^+$,
\begin{align*}
&\sum_{k = 1}^{\infty} \dim_k M(k)_{\lbd' - \nu} \\
&= \sum_{\beta \in \roots^+\cap \roots_0^+}\sum_{m = 1}^{\infty} \cl{P}(\nu - m\beta)v_{T-1}(q_{\beta}^{2(\rho,\beta^{\vee})}\chi_{2\beta}q^{(\lbd',2\beta)}T^{(\rho,2\beta)} - q_{\beta}^{2m})\\
&+\sum_{\beta \in \roots^+\setminus \roots_0^+}\sum_{m = 1}^{\infty} \cl{P}(\nu - m\beta)v_{T-1}(q_{\beta}^{2(\rho,\beta^{\vee})}q^{(\lbd',2\beta)}T^{(\rho,2\beta)} - q_{\beta}^{2m}\chi_{-2\beta}).
\end{align*}
\end{enumerate}
Since the term with $\beta \in \roots^+$ such that $\chi_{2\beta} = 0$
is $0$, we can conclude the following formula:
\begin{align*}
 \sum_{k = 1}^{\infty} \mathrm{ch}\,M(k) = \sum_{\substack{\beta \in \roots^+\cap \roots_{\chi},\\ s_{\beta}\cdot_{\chi}\lbd' < \lbd'}} \mathrm{ch}\,M_{\chi}(s_{\beta}\cdot_{\chi}\lbd').
\end{align*}
The remaing part of our proof is same with that of BGG theorem in the usual setting.
\end{proof}

\subsection{The parabolic category $\pintO{\chi}$}

In this subsection, we fix $S \subset \simples$.
Let $\roots_0^+$ be a positive system containing $\roots_S^+$.
Then we can find a reduced expression $s_{\bs{i}}$ of $w_0$
which begins with a reduced expression $w^{-1}w_0$. Moreover
we assume that $\{\alpha^{\bs{i}}_k\}_{k = 1}^{N_S} = \roots_S^+$,
where $N_S := \abs{\roots_S^+}$.

Let $\chi$ be a $k$-valued character on $2Q_0^+$
\textbf{satisfying $\chi_{2\alpha} = 1$ for all $\alpha \in \roots_S^+$}.
Consider the following subalgebra of $U_{q,\chi}(\g)$:
\begin{itemize}
 \item $U_{q,\chi}(\frl_S)$, generated by $(\aK_{\mu})_{\mu \in 2P}, (\aE_{\bs{i},\alpha})_{\alpha \in \roots_S^+}, (\aF_{\bs{i},\alpha})_{\alpha \in \roots_S^+}$. 
 \item $U_{q,\chi}(\frp_S)$, generated by $(\aK_{\mu})_{\mu \in 2P}, (\aE_{\bs{i},\alpha})_{\alpha \in \roots^+}, (\aF_{\bs{i},\alpha})_{\alpha \in \roots_S^+}$. 
 \item $U_{q,\chi}(\fru^+_S)$, generated by $(\aE_{\bs{i},\alpha})_{\alpha \in \roots^+\setminus\roots_S^+}$. 
 \item $U_{q,\chi}(\fru^-_S)$, generated by $(\aF_{\bs{i},\alpha})_{\alpha \in \roots^+\setminus\roots_S^+}$. 
\end{itemize}
By the PBW theorem, we have decompositions as follows:
\begin{align*}
 U_{q,\chi}(\g) &\cong U_{q,\chi}(\fru_S^-)\tensor U_{q,\chi}(\frp_S),\\
 U_{q,\chi}(\frp_S) &\cong U_{q,\chi}(\frl_S)\tensor U_{q,\chi}(\fru_S^+).
\end{align*}

\begin{lemm}
There is a canonical embedding of $U_{q,\chi}(\frp_S)$ into $U_{q}(\frp_S)$, which induces an embedding $U_{q,\chi}(\frl_S)$ into $U_q(\frl_S)$.
More concretely, this embedding is given as follows:
\begin{align*}
 \aE_{\bs{i},\alpha} \longmapsto E_{\bs{i},\alpha},\quad
 \aK_{\lbd} \longmapsto K_{\lbd},\quad
 \aF_{\bs{i},\alpha} \longmapsto F_{\bs{i},\alpha}K_{\bs{i},\alpha}.
\end{align*}
\end{lemm}
\begin{proof}
 Direct calculation.
\end{proof}


Using the identification above, we regard $U_q(\frl_S)$-modules
as $U_{q,\chi}(\frl_S)$-modules. We also regard $U_{q,\chi}(\frl_S)$-modules as $U_{q,\chi}(\frp_S)$-modules via $U_{q,\chi}(\frp_S)\longrightarrow U_{q,\chi}(\frl_S)$.

\begin{defn}
We define the twisted parabolic induction functor $\dqpind$
as $\map{U_{q,\chi}(\g)\tensor_{U_{q,\chi}(\frp_S)}\tend}{\mod{U_q(\frl_S)}}{\mod{U_{q,\chi}(\g)}}$.
\end{defn}

For a character $\map{\chi}{P}{k^{\times}}$, we have the following
comparison with $\qpind{} := U_q(\g)\tensor_{U_q(\frp_S)}\tend$.

\begin{lemm}
Let $\chi$ be a character on $P$ such that $\chi_{\eps} = 1$
for all $\eps \in S$. Then, under the isomorphism
in Lemma \ref{lemm:comparison with the usual qea}, we have the following
isomorphism of functors:
\begin{align*}
 \dqpind V &\cong \qpind(V\tensor \C_{\chi}), \\
1\tensor v &\longmapsto 1\tensor (v\tensor 1).
\end{align*}
\end{lemm}

\begin{defn}
We define $\pintO{\chi}$ as the full subcategory of $\mod{U_{q,\chi}(\g)}$
whose objects are all of $U_{q,\chi}(\g)$-module
$M$ satisfying the following conditions:
\begin{enumerate}
 \item $M$ is finitely generated as a $U_{q,\chi}(\g)$-module.
 \item As a $U_{q,\chi}(\frl_S)$-module, $M$ decomposes into
a direct sum of objects in $\rep_q^{\fin} L_S$.
 \item For any $m \in M$, $U_{q,\chi}(\fru_S^+)m$ is finite dimensional.
\end{enumerate}
\end{defn}
\begin{exam}
Let $L^S(\lbd)$ be the simple object of $\rep_q^{\fin} L_S$
with the highest weight $\lbd$. Then
$M_{S,\chi}(\lbd) := \dqpind L^S(\lbd)$
is an object of $\pintO{\chi}$.
This module has a unique irreducible quotient $L_{S,\chi}(\lbd)$.
Moreover any object in $\pintO{\chi}$
has a filtration whose quotients are isomorphic to $L_{S,\chi}(\lbd)$
for some $\lbd$.
\end{exam}

It is not difficult to see that $\pintO{\chi}$ is a full subcategory
of $\catO{\chi}$. Moreover it is closed under the tensor product
with finite dimensional representations of $U_q(\g)$.

\begin{lemm}
The category $\pintO{\chi}$ has a canonical structure of left
$\rep_q^{\fin} G$-module category, which is induced from
left $U_q(\g)$-comodule algebra structure of $U_{q,\chi}(\g)$.
\end{lemm}

In this paper, we focus on $\chi$ for which
$\pintO{\chi}$ is semisimple.

\begin{lemm} \label{lemm:simultaneous simplicity for pinto}
The following are equivalent:
\begin{enumerate}
 \item For any $\alpha \in \roots_0^+\setminus\roots_S^+$,
$\chi_{2\alpha} \not\in q_{\alpha}^{2\Z}$.
 \item For any $\lbd \in P^{+_S}$,
$M_{S,\chi}(\lbd)$ is simple.
\end{enumerate}
\end{lemm}
\begin{proof}
At first we show (i) $\Longrightarrow$ (ii).
By the same argument with
\cite[Proposition in Subsection 9.6]{MR2428237}, we obtain
\begin{align*}
 \ch{M_{S,\chi}(\lbd)} = \sum_{w \in W_S}(-1)^{\ell(w)}\ch{M_{\chi}(w\cdot_{\chi} \lbd)}.
\end{align*}
If $M_{S,\chi}(\lbd)$ has a composition factor $L_{S,\chi}(\mu)$
with $\mu \neq \lbd$, it appears as a composition factor
of $M_{\chi}(w\cdot_{\chi}\lbd)$ for some $w \in W_S$.
If $w \neq 1$, $w\cdot_{\chi}\lbd \not\in P^{+_S}$. This implies
$\mu < w\cdot_{\chi}\lbd$. The same conclusion also holds when $w = 1$
since $\mu \neq \lbd$ by assumption.

Proposition \ref{prop:BGG theorem} implies that $\mu$ is $\chi$-strongly linked
to $w\cdot_{\chi}\lbd$. Hence we can take a sequence
$\{\alpha_i\}_{i = 1}^n \in \roots_0^+\cap\roots_{\chi}$ such that
\begin{align*}
 \mu = s_{\alpha_n}s_{\alpha_{n - 1}}\cdots s_{\alpha_1}w\cdot_{\chi}\lbd < s_{\alpha_{n - 1}}\cdots s_{\alpha_1}w\cdot_{\chi}\lbd < \cdots < s_{\alpha_1}w\cdot_{\chi}\lbd < w\cdot_{\chi}\lbd.
\end{align*}
Since $W_S\cdot_{\chi}\lbd$ does not contain $\mu$, there is
the first $\alpha_i$ such that $\alpha_i \not\in \roots_S^+$.
Then $w' = s_{\alpha_{i - 1}}\cdots s_{\alpha_1}w$
is an element of $W_S$
and $s_{\alpha_i}\cdot_{\chi}(w'\cdot_{\chi}\lbd) < w'\cdot_{\chi} \lbd$. This implies $\chi_{2w'^{-1}\alpha_i} \in q_{w'^{-1}\alpha_i}^{2\Z}$,
where $w'^{-1}\alpha_i \in \roots_0^{+}\setminus\roots^+_S$.
This contradicts to our assumption.

Next we show (ii) $\Longrightarrow$ (i). Assume
the contrary of (i) and set $\roots_{\chi}^{\mathrm{int}} := \{\alpha \in \roots_{\chi}\mid \chi_{2\alpha} \in q_{\alpha}^{2\Z}\}$. This is
a closed subsystem of $\roots$. We consider the positive system
$\roots_{\chi}^{\mathrm{int},+} = \roots^+\cap \roots_{\chi}^{\mathrm{int}}$. This makes the associated Weyl group $W_{\chi}^{\mathrm{int}}$
into a Coxeter group.
Also note that $S$ is contained in the set of simple roots of $\roots_{\chi}^{\mathrm{int}}$.

Our assumption implies this root system is not empty.
Now we consider $\lbd \in P^{+_S}$ with
the following properties:

\begin{enumerate}
 \item $\abs{W_{\chi}\cdot_{\chi}\lbd} = \abs{W_{\chi}}$.
 \item The element $\lbd$ is maximal in $W_{\chi}\cdot_{\chi}\lbd$.
 \item For any $\beta \in \roots_{\chi}^{\mathrm{int},+}$,
$n + (\lbd + \rho,\beta^{\vee}) \in \Z_{> 0}$ where $n$ is determined
by $\chi_{2\beta} = q_{\beta}^{2n}$.
\end{enumerate}

By the assumption (iii), $M_{\chi}(s_{\beta}\cdot_{\chi}\lbd)$
is contained in $M_{\chi}(\lbd)$ for any $\beta \in \roots_{\chi}^{\mathrm{int},+}$. Then the simplicity of
$M_{S,\chi}(\lbd)$ implies that $M_{\chi}(s_{\beta}\cdot_{\chi}\lbd)$ is contained in $M_{\chi}(w\cdot_{\chi}\lbd)$ for some $w \in W_S\setminus\{1\}$, which implies $s_{\beta}\cdot_{\chi}\lbd \uparrow_{\chi} w\cdot_{\chi}\lbd$.
Hence we can take a sequence $\{\alpha_i\}_{i = 1}^k \subset \roots_{\chi}^+$ with the following property:
\begin{align*}
 s_{\beta}\cdot_{\chi}\lbd = s_{\alpha_k}s_{\alpha_{k - 1}}\cdots s_{\alpha_1}w\cdot_{\chi}\lbd < s_{\alpha_{k - 1}}\cdots s_{\alpha_1}w\cdot_{\chi}\lbd < s_{\alpha_1}w\cdot_{\chi}\lbd < w\cdot_{\chi}\lbd.
\end{align*}
This implies $s_{\beta} = s_{\alpha_k}s_{\alpha_{k - 1}}\cdots s_{\alpha_1}w$ and
\begin{align*}
 q^{(\lbd + \rho, 2w^{-1}s_{\alpha_1}\cdots s_{\alpha_{i - 1}}(\alpha_i))}\chi_{2w^{-1}s_{\alpha_1}\cdots s_{\alpha_{i - 1}}(\alpha_i)} \in q_{\alpha}^{2\Z_{>0}}
\end{align*}
for all $1 \le i < k$. Then our assumption (ii) implies that $w^{-1}s_{\alpha_1}\cdots s_{\alpha_{i - 1}}(\alpha_i) \in \roots_{\chi}^+$.
Moreover we also have $\chi_{2\alpha_i} \in q_{\alpha_i}^{2\Z}$
by induction on $i$. Hence $\alpha_i \in \roots_{\chi}^{\mathrm{int},+}$ by the definition of $\roots_{\chi}^{\mathrm{int}}$. In particular we see
$s_{\alpha_i}s_{\alpha_{i - 1}}\cdots s_{\alpha_1}w > s_{\alpha_{i - 1}}\cdots s_{\alpha_1}w$ in the Bruhat order on $W_{\chi}^{\mathrm{int}}$ for all $1 \le i \le k$, which also concludes that $1 < w < s_{\beta}$ in $W_{\chi}^{\mathrm{int}}$. In particular the length of $s_{\beta}$ in this Coxeter group is greater than $1$.

Now we consider the reduced expression of $s_{\beta}$ in $W_{\chi}^{\mathrm{int}}$.
Since $s_{\beta}$ is not an element of $W_S$, the expression
 contains a simple reflection $s_{\beta'}$ with $\beta' \not\in \roots_{\chi}^{\mathrm{int},+}\setminus \roots_S$. Then the discussion above implies the length of $s_{\beta'}$ is greater than $1$, which contradicts to the simplicity of $s_{\beta'}$. This concludes
(ii) $\Longrightarrow$ (i).
\end{proof}

\begin{lemm} \label{lemm:projectivity in pintO}
If $\lbd \in P^{+_S}$ satisfies $q^{(\lbd + \rho,2\alpha)}\chi_{2\alpha} \not\in q_{\alpha}^{2\Z_{<0}}$ for all $\alpha \in \roots_{\chi}^+$, $M_{S,\chi}(\lbd)$ is projective in $\pintO{\chi}$.
\end{lemm}
\begin{proof}
Let $M \longrightarrow N$ be an epimorphism and take
a non-zero homomorphism $\map{\phi}{M_{S,\chi}(\lbd)}{N}$.
Let $v$ be a highest weight vector in $L_{\lbd}$ with weight $\lbd$.
Then $\phi(1\tensor v)$ is also a weight vector in $N$ with weight
$\lbd$. Then we can take a lift $m$ of this element which is again
a weight vector of weight $\lbd$. Then we can find a highest weight
vector $w$ in $U_{q,\chi}(\n^+)m$, whose weight $\mu$
is $W_{\chi}$-linked to $\lbd$ and satisfies $\lbd \le \mu$.
Then Lemma \ref{lemm:characterization of dominancy}
implies $\lbd = \mu$, which means
$w$ is a scalar multiple of $m$. Since $M$ is an object of $\pintO{\chi}$, this implies that $w$ generates a $U_{q,\chi}(\frl_S)$-module
isomorphic to $L_{\lbd}$. Then we obtain a lift of $\phi$ by using the universal property of $M_{S,\chi}(\lbd)$.
\end{proof}

\begin{thrm} \label{thrm:characterization of semisimplicity}
The category $\pintO{\chi}$ is semisimple if and only if
$\chi_{2\alpha} \not\in q_{\alpha}^{2\Z}$ for all $\alpha \in \roots_0^+\setminus\roots_S^+$. In this case, the twisted parabolic induction
functor $\dqind$ gives an equivalence $\rep_q^{\fin} L_S \cong \pintO{\chi}$
as $k$-linear semisimple categories.
\end{thrm}
\begin{proof}
Assume that $\pintO{\chi}$ is semisimple. Since each $M_{S,\chi}(\lbd)$
is indecomposable, this assumption implies the simplicity of
$M_{S,\chi}(\lbd)$. Hence Lemma \ref{lemm:simultaneous simplicity for pinto} implies $\chi_{2\alpha} \not\in q_{\alpha}^{2\Z}$ for all $\alpha \in \roots_0^+\setminus\roots_S^+$.

Next we assume $\chi_{2\alpha} \not\in q_{\alpha}^{2\Z}$ for all $\alpha \in \roots_0^+\setminus\roots_S^+$. Then Lemma \ref{lemm:simultaneous simplicity for pinto} implies
each $M_{S,\chi}(\lbd)$ is simple. Hence it suffices to show
that there is no non-trivial extension of $M_{S,\chi}(\lbd)$ by
$M_{S,\chi}(\lbd')$. This follows from Lemma \ref{lemm:projectivity in pintO}.
\end{proof}

\subsection{Highest weight vectors in tensor products}

For later use, we investigate highest weight vectors
in tensor products of finite dimensional representations
and $\chi$-shifted Verma modules.

We say that $\Lambda \subset \ah_q^*$ is \emph{$\chi$-strongly regular} when
either of $q^{(\Lambda + \rho,2\beta)}\chi_{2\beta} \cap q_{\beta}^{2\Z_{\ge 0}} = \emptyset $ or $q^{(\Lambda + \rho,2\beta)}\chi_{2\beta} \cap q_{\beta}^{2\Z_{\le 0}} = \emptyset$ holds for all $\beta \in \roots_0^+$.

For a $U_{q,\chi}(\g)$-module $M$, we define $M^{\n^+}$
as follows:
\begin{align*}
 M^{\n^+} := \{m \in M \mid \aE_{\bs{i},\alpha}m = 0 \text{ for all }\alpha \in \roots^+\}.
\end{align*}

\begin{prop} \label{prop:strong regularity}
Let $\lbd \in \ah_q^*$ be a weight and $V$ be a finite dimensional representation of $U_q(\g)$.
If $\lbd + \wt V$  is $\chi$-strongly regular,
the canonical map $(V\tensor M_{\chi}(\lbd))^{\n^+} \longrightarrow V\tensor M_{\chi}(\lbd)_{\lbd} \cong V$
is an isomorphism.
\end{prop}
\begin{proof}
By the usual discussion, we have a filtration $(M_k)_{k = 0}^{\dim V}$
of $V\tensor M_{\chi}(\lbd)$ such that $M_0 = 0, M_{\dim V} = V\tensor M_{\chi}(\lbd)$ and $M_{i + 1}/M_i \cong M_{\chi}(\mu_i + \lbd)$ for some
$\mu_i \in \wt{V}$. Then our assumption implies that
each $M_i$ has a complement submodule in $M_{i + 1}$, in particular
we have an isomorphism $V\tensor M_{\chi}(\lbd) \cong \dqind (V\tensor k_{\lbd})$, which implies $\dim (V\tensor M_{\chi}(\lbd))^{\n^+} = \dim V$.

Now take a highest weight vector in $(V\tensor M_{\chi}(\lbd))_{\mu + \lbd}$, where $\mu \in \wt{V}$, presented as follows:
\begin{align*}
\sum_{\Lambda} v_{\Lambda}\tensor \aF^{\Lambda}\tensor 1.
\end{align*}
To prove the statement, it suffices to show that $v_0 = 0$
implies $v_{\Lambda} = 0$ for all $\Lambda$.
Fix $\nu \in Q^+$ and take $y \in U_{q,\chi}(\n^+)_{\nu}$. Then
\begin{align*}
 \Delta(y) = K_{\nu}\tensor y + \sum_{i = 1}^m y_{1,m}\tensor y_{2,m}
\end{align*}
with $y_{1,m} \in U_{q,\chi}(\frb^+)_{\nu - \nu_m}$ and $y_{2,m} \in U_{q,\chi}(\n^+)_{\nu_m}$, where $\nu_m \in Q^+$ such that $\nu_m < \nu$.
If $\nu \neq 0$, we have
\begin{align*}
 0 &= \Delta(y)\sum_{\Lambda} v_{\Lambda}\tensor \aF^{\Lambda}\tensor 1\\
&= \sum_{\Lambda}K_{\nu}v_{\Lambda}\tensor y\aF^{\Lambda}\tensor 1 + 
\sum_{i = 1}^m \sum_{\Lambda} y_{1,m}v_{\Lambda}\tensor y_{2,m}\aF^{\Lambda}\tensor 1.
\end{align*}
Looking at the terms in $V\tensor M_{\chi}(\lbd)_{\lbd}$, we obtain
\begin{align*}
 \sum_{\Lambda\cdot\alpha = \nu} \chi_{\lbd}(\cl{S}_{\nu}(y,\aF^{\Lambda}))K_{\nu}v_{\Lambda} = -\sum_{i = 1}^m \sum_{\Lambda\cdot\alpha = \nu_m}\chi_{\lambda}(\cl{S}_{\nu_m}(y_{2,m},\aF^{\Lambda}))y_{1,m}v_{\Lambda}.
\end{align*}
Hence, if we see that $\chi_{\lbd}(\det \cl{S}_{\nu}) \neq 0$, we can conclude $v_{\Lambda} = 0$ when $\Lambda\cdot\alpha = \nu$
from $v_{\Lambda} = 0$ when $\Lambda\cdot\alpha < \nu$. By Lemma \ref{lemm:Shapovalov determinant}, $\chi_{\lbd}(\det \cl{S}_{\nu})$ is a non-zero scalar multiple of
\begin{align*}
  \prod_{\beta \in \roots^+ \cap \roots_0^+}\prod_{m = 1}^{\infty}&
(q_{\beta}^{2(\lbd + \rho,\beta^{\vee})}\chi_{2\beta} - q_{\beta}^{2m})^{P(\nu - m\beta)} \\
&\times
 \prod_{\beta \in \roots^+ \setminus \roots_0^+}\prod_{m = 1}^{\infty}
(q_{\beta}^{2(\lbd + \rho,\beta^{\vee})} - q_{\beta}^{2m}\chi_{-2\beta})^{P(\nu - m\beta)}.
\end{align*}
Fix $\beta \in \roots^+\cap \roots_0^+$ and take $m > 0$ so that $\nu > m\beta$. If $q^{(\wt{V} + \lbd + \rho,2\beta)}\chi_{2\beta} \cap q_{\beta}^{2\Z_{\ge 0}} = \emptyset$, we can see directly that the $\beta$-factor is non-zero.
We assume that $q^{(\wt{V} + \lbd + \rho,2\beta)}\chi_{2\beta} \cap q_{\beta}^{2\Z_{\le 0}} = \emptyset$. If $\mu + \nu$ is not in $\wt{V}$,
there is nothing to prove since $v_{\Lambda} \in V_{\mu + \nu} = \{0\}$.
Hence we assume that $\mu + \nu \in \wt{V}$. Then
there is $\nu_m \in \wt{V}$ such that $(\nu_m,2\beta^{\vee}) = (m\beta + \mu,2\beta^{\vee}) = 4m + (\mu,2\beta^{\vee})$. Since $s_{\beta}(\nu_m) \in s_{\beta}(\wt{V}) = \wt{V}$,
our assumption implies $q_{\beta}^{-4m}q_{\beta}^{(-\mu + \lbd + \rho,2\beta^{\vee})}\chi_{2\beta} \not\in q_{\beta}^{2\Z_{\le 0}}$. On the other hand, we have $q_{\beta}^{(\mu + \lbd + \rho,2\beta^{\vee})}\chi_{2\beta} \not\in q_{\beta}^{2\Z_{\le 0}}$. Hence we can conclude
$q_{\beta}^{-2m}q_{\beta}^{(\lbd + \rho,2\beta^{\vee})}\chi_{2\beta} \not\in q_{\beta}^{2\Z_{\le 0}}$. This implies $\chi_{\lbd}(\det \cl{S}_{\nu}) \neq 0$.
The case of $\beta \in \roots^+\setminus\roots_0^+$ can be shown
by the same argument.
\end{proof}
\begin{coro} \label{coro:algebraicity of twist}
Let $V$ and $V'$ be objects of $\rep_q^{\fin} G$ and $\lbd$ be an integral
weight. If $\wt V' + \lbd$ and $\wt V + \wt V' + \lbd$ are $\chi$-strongly regular, there is a unique invertible $U_q(\h)$-homomorphism
$\map{E_{V,V',\lbd}(\chi)}{V\tensor V'\tensor k_{\lbd}}{V\tensor V'\tensor k_{\lbd}}$
which makes the following diagram commutative:
\begin{align*}
 \xymatrix{
\bdqind{(V\tensor V'\tensor k_{\lbd})} \ar[r] \ar[d] & \bdqind{(V\tensor V'\tensor k_{\lbd})} \ar[d] \\
V\tensor \bdqind{(V'\tensor k_{\lbd})} \ar[r] & V\tensor V'\tensor \bdqind{k_{\lbd}}.
}
\end{align*}
Moreover $E_{V,V',\lbd}$ is algebraic with respect to $\chi$.
\end{coro}


\section{Actions of $\ssorb$-type and $\sflag$-type}
\label{sect:actions}
In this section we introduce the main subject of this paper and
investigate their general properties not only in the case of type A.

\subsection{Definition and examples}

\begin{defn} \label{defn:action of ssorb type}
A \emph{semisimple action of $\ssorb$-type} is a pair of a semisimple left $\rep^{\fin}_q G$-module category $\cl{M}$ and an identification $\phi\colon\Z_+(\cl{M})\xlongrightarrow{\cong}\Z_+(L_S)$ as $\Z_+(G)$-modules.
Semisimple actions $(\cl{M},\phi)$ and $(\cl{N},\psi)$ of $\fssorb$-type are said to be equivalent if there is an equivalence $\map{F}{\cl{M}}{\cl{N}}$ of left $\rep_q^{\fin}G$-module categories which
makes the following diagram commutative:
\begin{align*}
 \xymatrix{
\Z_+(\cl{M}) \ar[rr]^-{F_*} \ar[rd]_-{\phi} && \Z_+(\cl{N}) \ar[ld]^-{\psi} \\
& \Z_+(L_S), &
}
\end{align*}
where $\map{F_*}{\Z_+(\cl{M})}{\Z_+(\cl{N})}$
is the induced isomorphism.
\end{defn}

\begin{rema} \label{rema:algebraic duality}
The semisimplicity arises from our original motivation,
which lies in study of quantum groups from
the
operator-algebraic perspective. As stated in Remark \ref{rema:operator algebraic duality},
a \emph{connected} semisimple left $\rep_q^{\fin} K$-module
category with a pointed irreducible object
corresponds to an ergodic action of $K_q$ on a unital \cstar-algebra.
In the algebraic setting, as stated in \cite[Theorem 4.6]{MR3847209},
the semisimplicity is replaced by a condition on certain projectivity
of the pointed object. In light of this duality in the algebraic
setting, \emph{actions of $\ssorb$-type}
should be defined and studied, though we leave it as a future problem.
\end{rema}
\begin{rema}
By the duality theorem \cite[Theorem 4.6]{MR3847209}, a semisimple action
of $\ssorb$-type can be presented as a concrete category.
Let $\cl{M}$ be a semisimple action of $\ssorb$-type.
We define $\cl{O}_{\cl{M}}(\ssorb)$, which has
a natural structure of left $U_q(\g)$-module algebra,
as follows:
\begin{align*}
 \cl{O}_{\cl{M}}(\ssorb) &:= \int^{\rep_q G} \cl{M}(\tend\tensor X_0,X_0)\tensor \tend \\
&\cong \bigoplus_{\mu \in P^+} \cl{M}(L_{\mu}\tensor X_0, X_0)\tensor L_{\mu},
\end{align*}
where $X_{\lbd}$ is an irreducible object corresponding to
$L^S(\lbd)$ under the identification $\Z_+(\cl{M})\cong \Z_+(L_S)$.
Note that $\cl{O}_{\cl{M}}(\ssorb)$ has the same
spectral decomposition with $\cl{O}(\ssorb)$.

Now the category of finitely generated right $\cl{O}_{\cl{M}}(\ssorb)$-modules with left semisimple actions of $U_q(\g)$
is denoted by $G_q\text{-}\mathrm{mod}_{\cl{O}_{\cl{M}}(\ssorb)}$.
Then we have the equivalence $\cl{M}\cong G_q\text{-}\mathrm{mod}_{\cl{O}_{\cl{M}}(\ssorb)}$ of left $\rep_q^{\fin}G$-module categories,
given by
\begin{align*}
 X \longmapsto \int^{\rep_q G} \cl{M}(\tend\tensor X_0,X)\tensor \tend \cong \bigoplus_{\mu \in P^+} \cl{M}(L_{\mu}\tensor X_0, X)\tensor L_{\mu}.
\end{align*}
\end{rema}

\begin{exam} \label{exam:std quantization}
The most fundamental example
of a semimsimple action of $\ssorb$-type
is the representation category $\rep_q^{\fin} L_S$ with
the natural action $(\pi,\rho)\longmapsto \pi|_{U_q(\frl_S)}\tensor \rho$ and the usual idendification
$\Z_+(\rep_q^{\fin} L_S)\cong \Z_+(L_S)$. It is not difficult
to see that $\cl{O}_{\rep_q^{\fin} L_S}(\ssorb)$
is the quantum coordinate algebra $\cl{O}_q(\ssorb)$.
\end{exam}

We obtain a large family of semisimple actions of $\ssorb$-type
from deformed quantum enveloping algebras.

\begin{prop} \label{prop:semisimple action arising from the category O}
 For any $\chi \in X_{\roots}(k)$ satisfying $\chi_{2\alpha} \not\in q_{\alpha}^{2\Z}$ for $\alpha \in \roots\setminus\roots_S$ and
$\chi_{2\alpha} = 1$ for $\alpha \in \roots_S$,
the category $\pintO{\chi}$ is a semisimple action of $\ssorb$-type,
equipped with the identification $\Z_+(\pintO{\chi})\cong \Z_+(L_S)$
induced from the $\chi$-shifted parabolic induction functor $\dqpind$.
\end{prop}
\begin{proof}
 By Theorem \ref{thrm:characterization of semisimplicity},
$\intO{\chi}$ is semisimple. By the left $U_q(\g)$-comodule
structure on $U_{q,\chi}(\g)$, it has a canonical structure of
a left $\rep_q^{\fin} G$-module category.

To see that the map $\map{(\dqind)_*}{\Z^+(\rep_q^{\fin} L_S)}{\Z^+(\pintO{\chi})}$ is
an isomorphism of $\Z_+(G)$-modules, it suffices to
see $\dqpind{(V\tensor W)}\cong V\tensor \dqpind{W}$
for all objects. This follows from the semisimplicity and comparison of formal characters.
\end{proof}

Recall that there is a canonical embedding $X_{\ssorb}(k) \longrightarrow X_{\roots}(k)$.

\begin{defn}
For $\phi \in X_{\ssorb}(k)$, the category
$\pintO{\chi_{\phi}}$ is denoted by $\pintO{\phi}$.
\end{defn}

By Lemma \ref{lemm:identification of toric and moduli}, $\pintO{\phi}$ is semisimple if and only if $\phi \in X_{\ssorb}^{\circ}(k)$. Moreover, $\pintO{\phi}$ defines a semisimple action of $\ssorb$-type in this case.

\begin{rema}[See {\cite[Subsection 4.4]{MR4837934}} for detail]
\label{rema:def quant aspect}
Even in the formal setting, the same construction
of left $\rep_h^{\fin} G$-module categories works
after modifying the definition of deformed qunatum enveloping
algebras slightly. In this case
each $\phi \in X_{\ssorb}(k)$ defines the semisimple category.
Then the corresponding algebra, denoted by
$\cl{O}_{h,\phi}(\ssorb)$, provides a deformation quantization
of $(\ssorb, \pi_{\phi})$
equipped with the action of $U_h(\g)$.
\end{rema}

We also introduce another approach to semisimple actions of $\ssorb$-type.

\begin{defn} \label{defn:associator}
 An \emph{associator} on $\rep_q^{\fin} L_S$ is an natural automorphism
$\Phi$ on the tensor product functor $\map{\tend\tensor\tend\tensor \tend}{\rep_q^{\fin} G\times \rep_q^{\fin} G\times \rep_q^{\fin} H}{\rep_q^{\fin} L_S}$
satisfying the following conditions:
\begin{enumerate}
 \item $\Phi_{V,\mathbf{1},W} = \id$, $\Phi_{\mathbf{1},V,W} = \id$.
 \item $\Phi_{V_1\tensor V_2,V_3,W}\circ \Phi_{V_1,V_2,V_3\tensor W} = \Phi_{V_1,V_2\tensor V_3,W}\circ (\id_{V_1}\tensor \Phi_{V_2,V_3,W})$.
\end{enumerate}

Equivalently, we say that $\Phi$ is an associator when
$\rep_{q,\Phi}^{\fin} L_S := (\rep_q^{\fin} L_S,\tensor, \Phi)$ is
a left $\rep_q^{\fin} G$-module category.
\end{defn}

Note that $\rep_{q,\Phi}^{\fin} L_S$ is canonically a semisimple action
of $\ssorb$-type. We say that two associators $\Phi$ and $\Psi$
are equivalent when $\rep_{q,\Phi}^{\fin} L_S \cong \rep_{q,\Psi}^{\fin} L_S$
as semisimple actions of $\fssorb$-type. In terms of natural transformations, this is equivalent to the existence of
an natural automorphism $b$ on
$\map{\tend\tensor\tend}{\rep_q^{\fin} G\times \rep_q^{\fin} L_S}{\rep_q^{\fin} L_S}$ satisfying
\begin{align*}
 \Phi_{V,V',W}b_{V,V'\tensor W}(\id\tensor b_{V',W})
= b_{V\tensor V',W}\Psi_{V,V',W}.
\end{align*}

\begin{lemm} \label{lemm:reduction to associator}
Any semisimple action of $\ssorb$-type is equivalent to $\rep_{q,\Phi}^{\fin} L_S$ for some associator $\Phi$.
\end{lemm}
\begin{proof}
Let $\cl{M}$ be a semisimple action of $\ssorb$-type and fix
a $k$-linear equivalence $\map{F}{\rep_q^{\fin} L_S}{\cl{M}}$
compatible with the identification $\Z_+(L_S)\cong \Z_+(\cl{M})$.
Since this identification preserves the action of $\Z_+(G)$,
we have a natural automorphism
$\map{f}{F(\tend\tensor \tend)}{\tend\tensor F(\tend)}$. Then the fully faithfulness of $F$ implies that there
is an associator $\Phi$ whose image under $F$ coincides with
the following composition of morphisms:
\begin{align*}
 F(V\tensor V'\tensor W)
\xlongrightarrow{f_{V,V'\tensor W}} V\tensor F(V'\tensor W)
&\xlongrightarrow{\id_V\tensor f_{V',W}} V\tensor V'\tensor F(W)\\
&\xlongrightarrow{f_{V\tensor V',\tensor W}^{-1}} F(V\tensor V'\tensor W).
\end{align*}
Now we can see that $\rep_{q,\Phi}^{\fin}L_S$ is equivalent to $\cl{M}$
as a semisimple action of $\ssorb$-type.
\end{proof}

\subsection{Twist of actions and Arkhipov's twisting}

Since the formal character of a finite dimensional representation
of $G$ is invariant under the action of the Weyl group $W$,
there are several identifications of $\Z_+$-modules in general.

We assume that $w \in W$ satisfies $w(R_{S}^+) = R_{S'}^+$. Then we see that the inner isomorphism $\map{\mathrm{Ad}_w}{G}{G}$
restricts to an isomorphism from $L_{S'}$ to $L_S$. In particular, it induces the $\Z^+(G)$-isomorphism $\Z^+(L_S) \cong \Z^+(L_{S'})$.
Under the identification $\Z^+(L_S) \cong \Z[P^{+_S}]$, this isomorphism is described as $\lbd \longmapsto w(\lbd)$.
Also note that $w$ induces a bijection $X_{L_{S}\backslash G}(k) \cong X_{L_{S'}\backslash G}(k)$ by $\chi \longmapsto w\cdot\chi$.

\begin{defn}
Let $\cl{M}$ be a semisimple action of $\ssorb$-type.
and $w \in W$ be an element such that $w(R_{S}^+) = R_{S'}^+$.
Then we define a semisimple action $w_*\cl{M}$
of $L_{S'}\backslash G$-type as $\cl{M}$ equipped with
the identification $\Z_+(\cl{M})\cong \Z_+(L_S)\cong\Z_+(L_{S'})$
\end{defn}

For the semisimple actions arising from deformed quantum enveloping algebras, we have the following comparison theorem.

\begin{prop} \label{prop:twisted catO}
 For any $\chi \in X_{L_S\backslash G}^{\circ}(k)$ and $w \in W$ such that $w(\roots_S^+) = \roots_{S'}^+$,
we have $w_*\pintO{\chi}\cong \cl{O}^{\frp_{S'},\mathrm{int}}_{q,w\cdot \chi}$
as semisimple actions of $L_{S'}\backslash G$-type.
\end{prop}

The proof of this proposition is based on Arkhipov's twisting functor in our setting. In the following,
we fix $\chi \in \mathrm{Ch}_k\,2Q^+$ and $w \in W$. We also fix $\eps \in \simples$ such that $\ell(ws_{\eps}) > \ell(w)$. Then we have $t_{ws_{\eps}} = t_wt_{s_{\eps}}$ for the automorphisms in Subsection \ref{ssect:dqea}, hence we obtain $\map{t_{\eps}}{U_{q,(ws_{\eps})^{-1}\chi}(\g)}{U_{q,w^{-1}\chi}(\g)}$.

Fix a reduced expression $s_{\bs{i}}$ of the longest element satisfying $\{\eps = \alpha^{\bs{i}}_1,\alpha^{\bs{i}}_2,\dots,\alpha^{\bs{i}}_k\} = w^{-1}(\roots^+)\cap \roots^+$ for some $k \ge 1$. Then $s_{i_1}s_{i_2}\cdots s_{i_k}$ is a reduced expression of $w^{-1}w_0$.

\begin{lemm}
 A subset $\{\aF_{\eps}^n\}_{n \ge 0}$ is an Ore subset of $U_{q,w^{-1}\chi}(\g)$.
\end{lemm}

Recall that a subset $S$ of a ring $R$ is an Ore subset if and only if $Sr \cap sR \neq \emptyset$ and $rS \cap Rs \neq \emptyset$ hold for any $r \in R$ and $s \in S$.

\begin{proof}
At first we see $\aF_{\eps}^{\Z_{\ge 0}}u \cap U_{q,w^{-1}\chi}(\g)\aF_{\eps}^n \neq \emptyset$ for any $u \in U_{q,w^{-1}(\chi)}(\g)$ and $n \ge 0$. It suffices to show the following for all $2 \le k \le N$: for any $2 \le \Lbd^- \le k \le \Lbd^+$ and $n \ge 1$, there exists $m \ge 0$
such that
\begin{align*}
 \aF_{\eps}^m\aE^{\Lbd^+} \in \mathrm{span}\{\aE^{\Gamma^+}\}_{k \le \Gamma^+}U_{q,w^{-1}\chi}(\ah)\aF_{\eps}^n,\\
 \aF_{\eps}^m\aF^{\Lbd^-} \in \mathrm{span}\{\aF^{\Gamma^-}\}_{\Gamma^- \le k}U_{q,w^{-1}\chi}(\ah)\aF_{\eps}^n.
\end{align*}
Actually these statements follow from Levendorskii-Soibelman's relations in \cite[Proposition 5.2]{MR1116413} and \cite[Proposition 3.4]{MR4837934}, by induction on $k$.

The same argument works for the other condition.
\end{proof}

Now we have the Ore localization of $U_{q,w^{-1}\chi}(\g)$, which will be denoted by $U_{q,w^{-1}\chi}(\g)_{\aF_{\eps}}$.
Then we obtain a $U_{q,w^{-1}\chi}(\g)$-bimodule
\begin{align*}
S_{q,w^{-1}\chi}^{\eps} := U_{q,w^{-1}\chi}(\g)_{\aF_{\eps}}/U_{q,w^{-1}\chi}(\g).
\end{align*}

\begin{defn}
We define an Arkhipov's twisting functor $\map{T_{\eps}}{\mod{U_{q,w^{-1}\chi}(\g)}}{\mod{U_{q,(ws_{\eps})^{-1}\chi}(\g)}}$ by
\begin{align*}
 T_{\eps}(M) := t_{\eps}^*(S_{q,w^{-1}\chi}^{\eps}\tensor_{U_{q,w^{-1}\chi}(\g)}M).
\end{align*}
\end{defn}

Next we make $T_{\eps}$ into a left $\rep_q^{\fin} G$-module functor.
Note that $\mod{U_{q,w^{-1}\chi}(\g)_{\aF_{\eps}}}$ is a left $\rep_q^{\fin} G$-module category. Actually, since $F_{\eps}$ acts nilpotently on objects of $\rep_q^{\fin} G$, we can extend
the $U_{q,w^{-1}\chi}(\g)$-action on $V\tensor M$ to the $U_{q,w^{-1}\chi}(\g)_{\aF_{\eps}}$-action, where $V \in \rep_{q}^{\fin} G$ and $M \in \mod{U_{q,w^{-1}\chi}(\g)}$. Then it is not difficult to see that the restriction functor $\mod{U_{q,w^{-1}\chi}(\g)_{\aF_{\eps}}}\longrightarrow\mod{U_{q,w^{-1}\chi}(\g)}$
is a left $\rep_q^{\fin} G$-module functor. Since $\rep_q^{\fin} G$ is rigid, its adjoint also has a structure of left $\rep_q^{\fin} G$-module structure. Namely, it is given by
\begin{align*}
 U_{q,w^{-1}\chi}(\g)_{\aF_{\eps}}\tensor_{U_{q,w^{-1}\chi}(\g)}(V\tensor M)&\longrightarrow V\tensor U_{q,w^{-1}\chi}(\g)_{\aF_{\eps}}\tensor_{U_{q,w^{-1}\chi}(\g)}M,\\
u\tensor(v\tensor m)&\longmapsto u_{(1)}v\tensor (u_{(2)}\tensor m).
\end{align*}
Hence we also have an isomorphism
\begin{align*}
S_{q,w^{-1}\chi}^{\eps}\tensor_{U_{q,w^{-1}\chi}(\g)}(V\tensor M) \cong V\tensor S_{q,w^{-1}\chi}^{\eps}\tensor_{U_{q,w^{-1}\chi}(\g)}M,
\end{align*}
which also defines a structure of left $\rep_{q}^{\fin} G$-module functor.

Finally, we consider the following composition of isomorphisms:
\begin{align*}
 T_{\eps}(V\tensor M)
&\cong t_{\eps}^*(S_{q,w^{-1}\chi}^{\eps}\tensor_{U_{q,w^{-1}\chi}(\g)}(V\tensor M))\\
&\cong t_{\eps}^*(V\tensor S_{q,w^{-1}\chi}^{\eps}\tensor_{U_{q,w^{-1}\chi}(\g)} M)\\
&\cong \cl{T}_{\eps}^*(V)\tensor t_{\eps}^*(S_{q,w^{-1}\chi}^{\eps}\tensor_{U_{q,w^{-1}\chi}(\g)} M)\\
&\cong V\tensor T_{\eps}(M),
\end{align*}
where the third isomorphism is given by the left multiplication by $A_{\eps}$ in the proof of Lemma \ref{lemm:left coaction},
and the fourth isomorphism is given by $\cl{T}_{\eps}^{-1}\tensor \id$. Then this defines a structure of left $\rep_q^{\fin} G$-module functor
on $T_{\eps}$, since we have the following identities:
\begin{align*}
 A_{\eps}\Delta(t_{\eps}(x))A_{\eps}^{-1} = (\cl{T}_{\eps}\tensor t_{\eps})\circ \Delta(x),\quad
A_{\eps}\Delta(\cl{T}_{\eps}) = \cl{T}_{\eps}\tensor \cl{T}_{\eps}.
\end{align*}

Finally we prove Proposition \ref{prop:twisted catO}
\begin{proof}[Proof of Proposition \ref{prop:twisted catO}]
At we prove the statement for $w^{-1}\cdot\chi$ with $\chi \in \mathrm{Ch}_k\,2Q^+$ and $w$ such that $w(\roots_S^+) = \roots_{S'}^+$. Note that $w\cdot\rho - \rho \in P_{S'}^+$.
Take a reduced expression $w = s_{i_1}s_{i_2}\cdots s_{i_k}$
and define a functor $\map{T}{\mod{U_{q,\chi}(\g)}}{\mod{U_{q,w^{-1}\chi}(\g)}}$ as a composition $T = T_{i_1}T_{i_2}\cdots T_{i_k}$.
By the discussion above, this functor is a left $\rep_q^{\fin} G$-module functor.
We also see that this functor gives a functor from $\pintO{\chi}$ to $\pintO{w^{-1}\chi}$. To conclude
this, we need the following presentation of $T$, which can be seen by the analogue of \cite[Proposition 3.5]{andersen2003}:
\begin{align*}
 T(M) \cong t_w^*(U_{q,\chi}(\n_w^-)^{\vee}\tensor_{U_{q,\chi}(\n_w^-)} M),
\end{align*}
where $U_{q,\chi}(\n_w^{-1})$ is the subalgebra of $U_{q,\chi}(\g)$ generated by $\{\aF_{\bs{i},j}\}_{j = 1}^k$.
By using this presentation, we see that
\begin{align*}
 \dim T(M_{S',\chi}(\lbd))_{\mu} \cong \dim M_{S,w^{-1}\chi}(w^{-1}\cdot\lbd)_{\mu}.
\end{align*}
This implies the functor $T$ preserves the parabolic category $\cl{O}$ with integral weights. Moreover we see that
$T(M_{S',\chi}(\lbd)) \cong M_{S,w^{-1}\chi}(w^{-1}\cdot\lbd)$ since the RHS is simple. By the semisimplicity of the parabolic categories $\cl{O}$ for $\chi$
and $w^{-1}\chi$, we see $T$ restricts to the equivalence of left $\rep_q^{\fin} G$-module categories.
By using an isomorphism $U_{q,w^{-1}\cdot \chi}(\g)\cong U_{q,w^{-1}\chi}(\g)$ as left $U_q(\g)$-comodule algebras satisfying $\aK_{\lbd} \longmapsto q^{-(w^{-1}\cdot{0},\lbd)}\aK_{\lbd}$ (c.f. \ref{lemm:comparison with the usual qea}), we obtain $w_*\cl{O}_{q,w^{-1}\cdot\chi}^{\frp_S,\mathrm{int}}\cong \cl{O}_{q,\chi}^{\frp_{S'},\mathrm{int}}$.

To deal with the general case, take $\chi$ and $w$ as in the statement. Also take a positive system $\roots_1^+$ satisfying
$\chi \in \mathrm{Ch}_k 2Q_1^+$ and $S \subset \simples_1$ (see \cite{MR4837934} for example).
In particular, if we consider $w_1 \in W$ such that $w_1(\roots_1^+) = \roots_+$,
$S_1 := w_1(S)$ is a subset of $\simples$. Hence we have $(w_1)_*\pintO{\chi}\cong \cl{O}_{q,w_1\cdot\chi}^{\frp_{S_1},\mathrm{int}}$.
Next consider $w_2 := w_1w^{-1}$ and $\roots_2^+ = w_2^{-1}(\roots^+)$. Then we have $w_2(\roots_{S'}^+) = \roots_{S_1}^+$.
Hence we have $(w_2)_*\cl{O}_{q,w\cdot\chi}^{\frp_{S'},\mathrm{int}}\cong \cl{O}_{q,w_1\cdot\chi}^{\frp_{S_1},\mathrm{int}}$. Combining
these results, we see the statement.
\end{proof}

\subsection{Invariant coefficients}

The objective of this subsection is to give a basic strategy to
distinguish different semisimple actions of $\ssorb$-type.
As a consequence, we prove the following proposition:

\begin{prop} \label{prop:distinction of actions}
Let $\chi$ and $\chi'$ be elements of $X_{\ssorb}^{\circ}(k)$.
If $\chi \neq \chi'$, we have $\pintO{\chi} \not\cong \pintO{\chi'}$.
\end{prop}

To explain the construction, we focus on the associator
picture (Definition \ref{defn:associator}) of semisimple actions of
$\ssorb$-type.

Take an associator $\Phi$ on $\rep_q^{\fin} L_S$. Take a finite
dimensional representation $V,V'$ of $U_q(\g)$ and $\lbd \in P_S^+$.
For an endomorphism $A \in \End_{U_q(\g)}(V\tensor V')$, we consider the following $U_q(\frl_S)$-morphism on
$V\tensor V'\tensor L_{\lbd}$:
\begin{align*}
 \Phi_{V,V',L_{\lbd}}^{-1}(A\tensor \id)\Phi_{V,V',L_{\lbd}}.
\end{align*}
In general this morphism depends on the representative $\Phi$
of an equivalence class of associators. Actually, if $\Psi$
is another associator equivalent to $\Phi$, there is
a natural automorphism $b$ such that $\Psi_{V,V',W} = b_{V\tensor V',W}^{-1}\Phi_{V,V',W}b_{V,V'\tensor W}(\id\tensor b_{V',W})$.
Then we have
\begin{align*}
 \Psi_{V,V',L_{\lbd}}^{-1}&(A\tensor \id)\Psi_{V,V',L_{\lbd}}\\
&=(b_{V,V'\tensor L_{\lbd}}(\id\tensor b_{V',L_{\lbd}}))^{-1}\Phi_{V,V',L_{\lbd}}(A\tensor \id)\Phi_{V,V',L_{\lbd}}b_{V,V'\tensor L_{\lbd}}(\id\tensor b_{V',L_{\lbd}}).
\end{align*}

Now consider the irreducible decompositions of
$V\tensor V'\tensor L_{\lbd}$. Then naturality of $b$ implies that
$b_{V,V'\tensor L_{\lbd}}(\id\tensor b_{V',L_{\lbd}})$ preserves
isotypical components in the sense that
\begin{align*}
 b_{V,V'\tensor L_{\lbd}}(\id\tensor b_{V',L_{\lbd}})([V\tensor V'\tensor L_{\lbd}]_{\mu}&\cap V\tensor [V'\tensor L_{\lbd}]_{\mu'})\\
&=[V\tensor V'\tensor L_{\lbd}]_{\mu}\cap (V\tensor [V'\tensor L_{\lbd}]_{\mu'}),
\end{align*}
where $[W]_{\mu}$ is the $L_{\mu}$-isotypical component of $W \in \rep_q^{\fin} L_S$.

In particular, if a given intersection of isotypical component
is of multiplicity one, the operator acts as a scalar. Hence
$\Phi_{V,V',W}$ gives a scalar only depending on its equivalence class, after restricting on such a component.
We call the scalar an \emph{invariant coefficient}
of $\Phi$.

To give the typical situation in which we can define invariant coefficiets, we take dominant integral weights $\mu,\nu$
and a simple root $\eps$ such that $(L_{\mu})_{\mu - \eps}$
and $(L_{\nu})_{\nu - \eps}$ are non-zero. Then, for any $w \in W$,
all of $(L_{\mu})_{w\mu}, (L_{\mu})_{w(\mu - \eps)}, (L_{\nu})_{w\nu},(L_{\nu})_{w(\nu - \eps)}$ are $1$-dimensional.

Hence, if $\lbd \in P_S^+$ is sufficiently dominant,
the following intersection is $1$-dimensional:
\begin{align*}
 [L_{\mu}\tensor L_{\nu}\tensor L_{\lbd}]_{\lbd + w(\mu + \nu - \eps)}\cap L_{\mu}\tensor [L_{\nu}\tensor L_{\lbd}]_{\lbd + w(\mu)}.
\end{align*}
Now, by considering
the projection $\map{A = P^{\mu,\nu}_{\eps}}{L_{\mu}\tensor L_{\nu}}{L_{\mu + \nu - \eps}}$ regarded as an endomorphism on $L_{\mu}\tensor L_{\nu}$, we obtain the invariant coefficient $c_{\mu,\nu;w,\eps}(\Phi;\lbd) \in k$. We also use $c_{\mu,\nu;w,\eps}(\cl{M};\lbd)$,
where $\cl{M}$ is a semisimple action of $\ssorb$-type equivalent
to $\rep_{q,\Phi}^{\fin} L_S$.

In order to calculate the invariant coefficient
$c_{\mu,\nu;w,\eps}(\cl{M};\lbd)$
for a given semisimple action $\cl{M}$ of $\fssorb$-type,
it is convenient to use another definition of the invariant coefficient.
Let $\cl{M}$ be a semisimple $\fssorb$-type action. For $\lbd \in P_S^+$ and $w \in W$, we have
\begin{align*}
\cl{M}(X&_{\lbd + w(\mu + \nu - \eps)}, L_{\mu}\tensor L_{\nu}\tensor X_{\lbd})\\
&\cong \cl{M}(X_{\lbd + w(\nu - \eps)}, L_{\nu}\tensor X_{\lbd})
 \tensor \cl{M}(X_{\lbd + w(\mu + \nu - \eps)}, L_{\mu}\tensor X_{\lbd + w(\nu - \eps)}) \\
&\oplus \cl{M}(X_{\lbd + w(\nu)}, L_{\nu}\tensor X_{\lbd})
 \tensor \cl{M}(X_{\lbd + w(\mu + \nu - \eps)}, L_{\mu}\tensor X_{\lbd + w(\nu)}).
\end{align*}
According to this decomposition, we consider the matrix presentation
of
\begin{align*}
{\cl{M}(X_{\lbd + w(\mu + \nu - \eps)}, L_{\mu}\tensor L_{\nu}\tensor X_{\lbd})}
\overset{P_{\eps}^{\mu,\nu}\circ\tend}{\longrightarrow}
{\cl{M}(X_{\lbd + w(\mu + \nu - \eps)}, L_{\mu}\tensor L_{\nu}\tensor X_{\lbd})}.
\end{align*}
Then $c_{\mu,\nu;w,\eps}(\cl{M};\lbd)$ appears as
the $(2,2)$-entry of the matrix. Moreover, by using this picture,
we can define $c_{\mu,\nu;w,\eps}(\intO{\chi};\lbd)$
for $\chi \in X_{\roots}(k)$ whenever $\wt V + \wt V' + \lbd$
and $\wt V' + \lbd$ are $\chi$-strongly regular.


\begin{lemm} \label{lemm:invariant coefficients at extremal}
In the setting above, we have
\begin{align*}
 c_{\mu,\nu;w,\eps}(\intO{\chi};\lbd) = \frac{[(\nu,\eps^{\vee})]_{q_{\eps}}}{[(\mu + \nu,\eps^{\vee})]_{q_{\eps}}}\frac{[(\mu + \nu + w^{-1}\cdot\lbd,\eps^{\vee});\chi_{w(2\eps)}]_{q^{\eps}}}{[(\nu + w^{-1}\cdot \lbd,\eps^{\vee});\chi_{w(2\eps)}]_{q^{\eps}}}.
\end{align*}
\end{lemm}

\begin{proof}
We assume that $\eps$ is contained in $\roots_0^+$. The other case
is similar.

To calculate the matrix coefficient, we have to determine
the isomorphism $\bdqind{(V\tensor k_{\lbd})} \cong V\tensor \bdqind{k_{\lbd}}$ at least on some weight vectors. It is not difficult to see
\begin{align*}
 (1\tensor v_{\mu}\tensor 1)\longmapsto v_{\mu}\tensor (1\tensor v_{\lbd}).
\end{align*}
Similarly we also have
\begin{align*}
1\tensor  (F_{\eps}v_{\mu}\tensor 1)\longmapsto q_{\eps}^{(\mu,\eps^{\vee})}\frac{\chi_{2\eps}q_{\eps}^{2(\lbd,\eps^{\vee})}-1}{q_{\eps} - q_{\eps}^{-1}}F_{\eps}v_{\mu} \tensor (1\tensor 1) - [(\mu,\eps^{\vee})]_{q_{\eps}}
v_{\mu}\tensor (\aF_{\eps}\tensor 1).
\end{align*}
Hence we have
\begin{align}
 1\tensor (v_{\mu}\tensor F_{\eps}u_{\nu})\longmapsto &q_{\eps}^{(\nu,\eps^{\vee})}\frac{\chi_{2\eps}q_{\eps}^{2(\lbd,\eps^{\vee})}-1}{q_{\eps} - q_{\eps}^{-1}}v_{\mu}\tensor F_{\eps}v_{\nu} \tensor (1\tensor 1) + \cdots, \\
1\tensor (F_{\eps}v_{\mu}\tensor u_{\nu})\longmapsto &q_{\eps}^{(\mu,\eps^{\vee})}\frac{\chi_{2\eps}q_{\eps}^{2(\lbd + \nu,\eps^{\vee})}-1}{q_{\eps}- q_{\eps}^{-1}}F_{\eps}v_{\mu}\tensor v_{\nu}\tensor (1\tensor 1)  \label{eq:second highest weight vector}\\
&- q_{\eps}^{(\nu,\eps^{\vee})}[(\mu,\eps^{\vee})]_{q_{\eps}}v_{\mu}\tensor F_{\eps}v_{\nu}\tensor (1\tensor 1) + \cdots. \notag
\end{align}
To determine $c_{\mu,\nu;1,\eps}(\intO{\chi};\lbd)$, it suffices to
consider the image of the right hand side in (\ref{eq:second highest weight vector}) under $P_{\eps}^{\mu,\nu}\tensor \id$. Since this projection kills $F_{\eps}(v_{\mu}\tensor v_{\nu})$ and preserves $q_{\eps}^{(\mu,\eps^{\vee})}[(\nu,\eps^{\vee})]_{q_{\eps}}F_{\eps}v_{\mu}\tensor v_{\nu} - [(\mu,\eps^{\vee})]_{q_{\eps}}v_{\mu}\tensor F_{\eps}v_{\nu}$,
\begin{align*}
 P_{\eps}^{\mu,\nu}(F_{\eps}v_{\mu}\tensor v_{\nu})
&= -q_{\eps}^{(\nu,\eps^{\vee})}P_{\eps}^{\mu,\nu}(v_{\mu}\tensor F_{\eps}v_{\nu})\\
&=\frac{1}{[(\mu + \nu,\eps^{\vee})]_{q_{\eps}}}\left(q_{\eps}^{(\mu,\eps^{\vee})}[(\nu,\eps^{\vee})]_{q_{\eps}}F_{\eps}v_{\mu}\tensor v_{\nu} - [(\mu,\eps^{\vee})]_{q_{\eps}}v_{\mu}\tensor F_{\eps}v_{\nu}\right).
\end{align*}
Hence we have
\begin{align*}
 c_{\mu,\nu;1,\eps}(\intO{\chi};\lbd) = q_{\eps}^{-(\mu,\eps^{\vee})}\frac{[(\nu,\eps^{\vee})]_{q_{\eps}}}{[(\mu + \nu,\eps^{\vee})]_{q_{\eps}}}\frac{\chi_{2\eps}q_{\eps}^{2(\lbd + \mu + \nu,\eps^{\vee})} - 1}{\chi_{2\eps}q_{\eps}^{2(\lbd + \nu,\eps^{\vee})} - 1}.
\end{align*}

To deal with general $w \in W$, we consider the case $\chi \in X_{\fssorb}^{\circ}(k)$. Then we have $w_*\intO{\chi} \cong \intO{w\cdot\chi}$ by Proposition \ref{prop:twisted catO}. Hence we have
\begin{align*}
 c_{\mu,\nu;w,\eps}(\intO{\chi};\lbd)
= c_{\mu,\nu;1,\eps}((w^{-1})_*\intO{\chi};w^{-1}(\lbd))
= c_{\mu,\nu;1,\eps}(\intO{w^{-1}\cdot\chi};w^{-1}(\lbd)).
\end{align*}
On the other hand, $c_{\mu,\nu;w,\eps}(\intO{\chi};\lbd)$
is algebraic on $\chi$ by Corollary \ref{coro:algebraicity of twist}.
Hence the equality
above also holds for $\chi$ satisfying our condition.
\end{proof}

This completes the proof of Proposition \ref{prop:distinction of actions}.
\begin{proof}[Proof of Proposition \ref{prop:distinction of actions}]
 If $S$ is empty, the previous lemma shows the statement. For general case, note that the isomorphism $\dqind{(V\tensor k_{\lbd})}\cong
V\tensor \dqind{k_{\lbd}}$ induces an isomorphism
$\dqpind{(V\tensor L_{\lbd})} \cong V\tensor \dqpind{L_{\lbd}}$ if $\wt V + \lbd$ is $\chi$-strongly regular. Hence we have
\begin{align*}
 c_{\mu,\nu;w,\eps}(\pintO{\chi};\lbd)
=  c_{\mu,\nu;w,\eps}(\intO{\chi};\lbd)
\end{align*}
when $\wt L_{\mu} + \wt L_{\nu} + \lbd$ and $\wt L_{\nu} + \lbd$ are $\chi$-strongly regular. Then Lemma \ref{lemm:invariant coefficients at extremal} shows the statement.
\end{proof}

\subsection{\cstar-structure}

In this subsection we discuss on \cstar-structure
on semisimple actions of $\ssorb$-type.
The base field is $\C$ and $q$ is a real number between $0$ and $1$.
We consider $U_q(\k)$, which is $U_q(\g)$ with the \star-structure.
Then, as pointed out in Subsection \ref{subsec:compact real forms}, we can form
a \cstar-tensor category $\rep_q^{\fin} K$ of finite dimensional
unitary representations of $U_q(\k)$. 

See \cite{MR3121622} for the notion of tensor categories and their module categories in the \cstar-algebraic setting.

\begin{defn} \label{defn:action of flag manifold type}
An action of $\sflag$-type is a pair of a semisimple left $\rep_q^{\fin} K$-module \cstar-category $\cl{M}$
and an identification $\Z_+(\cl{M}) \cong \Z_+(K_S)$ as $\Z_+(K_S)$-modules.
\end{defn}
\begin{rema} \label{rema:operator algebraic duality}
As same with the algebraic setting, we have the duality theorem
for a \emph{connected} semisimple left $\rep_q^{\fin} K$-module
category with a pointed irreducible object
corresponds to an ergodic action of $K_q$ on a unital \cstar-algebra.
Since $\rep_q^{\fin} K_S$ with $\mathbf{1}$ as a pointed object
corresponds to the \emph{standard quantum flag manifold} $C_q(\sflag)$, it is natural to regard an action of $\sflag$-type
as a noncommutative analogue of $\sflag$.
\end{rema}
Let $\cl{M}$ be a left $\rep_q^{\fin} G$-module category $\cl{M}$.
A \emph{unitarization} of $\cl{M}$ is a pair of a left $\rep_q^{\fin} K$-module \cstar-category $\cl{M}^{\mathrm{uni}}$ and the equivalence
$\cl{M}\cong \cl{M}^{\mathrm{uni}}$ as a left $\rep_q^{\fin}G$-module category. We also say that $\cl{M}$ is \emph{unitarizable}
if it admits a unitarization.

The following lemma implies that unitarizations of
a left $\rep_q^{\fin} G$-module category are
unitarily equivalent to each other.

\begin{lemm} \label{lemm:uniqueness of unitarization}
Let $\cl{C}$ be a \cstar-tensor category and $\cl{M},\cl{M'}$
be semisimple left $\cl{C}$-module \cstar-categories.
If $\cl{M}$ is equivalent to $\cl{M}'$ as a left $\cl{C}$-module
category, $\cl{M}$ is equivalent to $\cl{M}'$ as a left $\cl{C}$-module
\cstar-categories.
\end{lemm}
\begin{proof}
The proof of \cite[Proposition 2.11]{MR4538281} works even when $\cl{C}$ is not a fusion category.
\end{proof}

By this lemma, there is a natural bijection between
the unitary equivalence classes of actions of $\sflag$-type
and the equivalence classes of unitarizable semisimple actions
of $\ssorb$-type. Hence it suffices to discuss on
unitarizability of semisimple actions of $\ssorb$-type.

\begin{prop} \label{prop:non-unitarizability}
For $\chi \in X_{L_S\backslash G}^{\circ}(\C)\setminus X_{K_S\backslash K}^{\quot}$, $\pintO{\chi}$ is not unitarizable.
\end{prop}
\begin{proof}
If $\pintO{\chi}$ is unitarizable, $c_{\mu,\nu;w,\eta}(\pintO{\chi},\lbd)$
is non-negative. By Lemma \ref{lemm:invariant coefficients at extremal} and the proof of Proposition \ref{prop:twisted catO},
this condition is equivalent to the non-negativity of
\begin{align*}
 \frac{\chi_{2w(\eps)}q_{\eps}^{2(w^{-1}\cdot\lbd + \mu + \nu, \eps^{\vee})} - 1}{\chi_{2w(\eps)}q_{\eps}^{2(w^{-1}\cdot\lbd + \nu, \eps^{\vee})} - 1}
\end{align*}

Now set $\mu = \nu = \rho$. Then $\lbd + 2\rho$ with $\lbd \in P^{+_S}$
satisfies the regularity condition. In this case
\begin{align*}
 \frac{\chi_{2w(\eps)}q_{\eps}^{2(w^{-1}\cdot\lbd + \mu + \nu, \eps^{\vee})} - 1}{\chi_{2w(\eps)}q_{\eps}^{2(w^{-1}\cdot\lbd + \nu, \eps^{\vee})} - 1}
=
 \frac{\chi_{2w(\eps)}q_{\eps}^{2(w^{-1}(\rho),\eps^{\vee})}q_{\eps}^{2 + 2(\lbd, w(\eps)^{\vee})} - 1}{\chi_{2w(\eps)}q_{\eps}^{2(w^{-1}(\rho),\eps^{\vee})}q_{\eps}^{2(\lbd, w(\eps)^{\vee})} - 1}.
\end{align*}
Now take $\alpha \in \roots^+\setminus \roots_S^+$. Then there are
a simple root $\eps \in \simples$ and $w \in W$ such that $\alpha = w(\eps)$. Since $\alpha$ is not in $\roots_S$, $(\varpi_{\eta},\alpha^{\vee}) \neq 0$ for some $\eta \in \simples\setminus S$. Hence the value of
$(\lbd,\alpha^{\vee})$ for $\lbd \in P^{+_S}$ takes any integer. Then our condition on $c_{\rho,\rho;w,\eps}(\chi,\lbd)$ implies
the sign of $\chi_{2w(\eps)}q_{\eps}^{2n} - 1$
 does not change when $n \in \Z$ varies. This implies $\chi_{2\alpha} \not\in(0,\infty)$ for all $\alpha \not\in \roots_S$ and $\chi \in X_{K_S\backslash K}^{\quot}$.
\end{proof}



On the other hand, we cannot use the invariant coefficients to
see unitarizability of $\pintO{\chi}$ for $\chi \in X_{\sflag}^{\mathrm{quot}}$. By Proposition \ref{prop:twisted catO}, we may assume
$\chi \in \Ch_{\R} 2Q^+$. In this case we have the following \star-strucutre on $U_{q,\chi}(\g)$ inherited from $U_{q}(\k)$:
\begin{align*}
 \aF_{\eps}^* = \aE_{\eps},\quad
 \aK_{\lbd}^* = \aK_{\lbd},\quad
 \aE_{\eps}^* = \aF_{\eps}.
\end{align*}
This \star-algebra is denoted by $U_{q,\chi}(\k)$.
Note that $\aF_{\alpha}^* \neq \aE_{\alpha}$ in general since the braid group action does not preserves the \star-structure.

To see the unitarizability of $\pintO{\chi}$, we would like to introduce
the following \cstar-category.

\begin{defn}
 A unitary $U_{q,\chi}(\k)$-module in category $\pintO{\chi}$ is a $U_{q,\chi}(\k)$-module $M$ equipped with an inner product satisfying the following conditions:
\begin{enumerate}
 \item $\ip{x^*m, m'} = \ip{m,xm'}$ for all $x \in U_{q,\chi}(\k)$
and $m,m' \in M$.
 \item The underlying $U_{q,\chi}(\g)$-module belongs to $\pintO{\chi}$.
\end{enumerate}

The category of unitary $U_{q,\chi}(\k)$-module in the category $\pintO{\chi}$ is denoted by $\unipintO{\chi}$.
\end{defn}

Since Hom spaces of $\pintO{\chi}$ is finite dimensional, $\unipintO{\chi}$ have a canonical structure of semisimple
\cstar-category. Moreover it is not difficult to see that these
are semisimple left $\rep_q^f K$-module \cstar-categories.

By definition there is the forgelful functor $\pintO{\chi}\longrightarrow
\unipintO{\chi}$, which is faithful.

Let $\ah_{q,\R}^*$ be the set of $\R^{\times}$-valued characters on $2P$.

\begin{lemm}
Let $M$ be a highest weight module of $U_{q,\chi}(\k)$. Then there
is a sesquilinear form on $M$ satisfying $\ip{x^*m, m'} = \ip{m,xm'}$ for all $x \in U_{q,\chi}(\k)$ and $m,m' \in M_{\chi}(\lbd)$.
Moreover such a sesquilinear form is unique up to scalar multiplication.
\end{lemm}
\begin{proof}
The usual argument can be applied.
\end{proof}

\begin{prop} \label{prop:unitarizability}
Assume $\chi_{2\alpha} = 1$ for $\alpha \in \roots_S^+$ and
$\chi_{2\alpha} \le 0$ for all $\alpha \in \roots^+\setminus\roots_S^+$.
Then the forgetful functor gives an equivalence $\unipintO{\chi}\cong \pintO{\chi}$ as actions of $\ssorb$-type.
\end{prop}

What we have to prove is essential-surjectivity of this functor. In other words, we have to show that each $M_{S,\chi}(\lbd)$
is unitarizable if $\chi$ satisfies the condition. In general this is difficult, but still possible for a special case.
Then we extend such a result by using the continuity of $M_{S,\chi}(\lbd)$ with respect to $\chi$.

We need some preparation.

\begin{lemm}
For $x \in U_q(\g)$ and $\eps \in \Delta$,
$\cl{T}_{\eps}(x)^* = (-1)^{(\wt{x},\eps^{\vee})}\cl{T}_{\eps}^{-1}(x^*)$.
\end{lemm}
\begin{proof}
This can be seen directly from \cite[8.14]{MR1359532}.
\end{proof}

\begin{lemm} \label{lemm:formula for adjoint}
The adjoint of $F_{\bs{i},k}K_{\alpha^{\bs{i}}_k}$ has the following expression:
\begin{align*}
(F_{\bs{i},k}K_{\alpha^{\bs{i}}_k})^*
= q^{-(\alpha^{\bs{i}}_k, \alpha^{\bs{i}}_1 + \alpha^{\bs{i}}_2 + \cdots + \alpha^{\bs{i}}_{k - 1})}E_{\bs{i},k} + \sum_{\Lambda\cdot\alpha^{\bs{i}} = \alpha^{\bs{i}}_k, \Lambda \neq \delta_k} C_{\Lambda}E_{\bs{i},N}^{\Lambda_N}E_{\bs{i}, N - 1}^{\Lambda_{N - 1}}\cdots E_{\bs{i},1}^{\Lambda_1}.
\end{align*}
\end{lemm}
\begin{proof}
We use induction on $k$. If $k=1$, the statement follows from the definition of the braid group action.

For general cases, we consider
a reduced expression  $s_{\bs{j}} = s_{i_2}s_{i_3}\cdots$
of $w_0$.
Then we have $F_{\bs{i},k}K_{\alpha^{\bs{i}}_k} = \cl{T}_{i_1}(F_{\bs{j},k-1}K_{\alpha^{\bs{j}}_{k - 1}})$. As a consequence of the induction hypothesis and the previous lemma, we have
\begin{align*}
 (F_{\bs{i},k}K_{\alpha^{\bs{i}}_k})^*
&= (-1)^{(\alpha^{\bs{j}}_{k - 1}, \alpha^{\bs{i}\vee}_1)}\cl{T}_{i_1}^{-1}((F_{\bs{j}, k -1}K_{\alpha^{\bs{j}}_{k - 1}})^*) \\
&=(-1)^{(\alpha^{\bs{j}}_{k - 1}, \alpha^{\bs{i}\vee}_1)}
\cl{T}_{i_1}^{-1}\biggl(q^{-(\alpha^{\bs{j}}_{k - 1}, \alpha^{\bs{j}}_1 + \alpha^{\bs{j}}_2 + \cdots + \alpha^{\bs{j}}_{k - 2})}E_{\bs{j},k-1}\\
&\hspace{20mm}+ \sum_{\Lambda\cdot\alpha^{\bs{j}} = \alpha^{\bs{j}}_{k - 1}, \Lambda \neq \delta_{k - 1}} C'_{\Lambda}E_{\bs{j},N}^{\Lambda_N}E_{\bs{j}, N - 1}^{\Lambda_{N - 1}}\cdots E_{\bs{j},1}^{\Lambda_1}\biggr).
\end{align*}

Now determine the coefficient of $E_{\bs{i},k}$ in
$(F_{\bs{i},k}K_{\alpha^{\bs{i}}_k})^*$. Take a finite
dimensional representation $V$ and a weight vector $v$
such that $E_{\bs{i},1}v = E_{\bs{i},2}v = \cdots = E_{\bs{i}, k-1}v = 0$.
Since $C_{\Lambda} \neq 0$ implies that $\Lambda = \delta_k$ or
$\Lambda_{< k} \neq 0$, we have
\begin{align*}
 (F_{\bs{i},k}K_{\alpha^{\bs{i}}_k})^*v = C_{\delta_k}E_{\bs{i},k}v.
\end{align*}

On the other hand, we can use the expression above to calculate the
LHS. By $E_{\bs{i},1}v = 0$, $v$ is a highest weight vector with respect to $U_q(\frl_{i_1})$. Hence $\cl{T}_{i_1}v$
is a lowest weight vector. Moreover this satisfies
\begin{align*}
 E_{\bs{j},1}\cl{T}_{i_1}v
= E_{\bs{j},2}\cl{T}_{i_1}v
=\cdots
= E_{\bs{j},{k - 2}}\cl{T}_{i_1}v = 0
\end{align*}
since $\cl{T}_{i_1}v$ is a scalar multiple of $\cl{T}_{i_1}^{-1}v$
and $\cl{T}_{i_1}E_{\bs{j,l}}\cl{T}_{i_1}^{-1}v = E_{\bs{i},l + 1}v = 0$ for $1 \le l \le k - 2$. Hence we have
\begin{align*}
 (F_{\bs{i},k}K_{\alpha^{\bs{i}}_k})^*v
&= (-1)^{(\alpha_{k - 1}^{\bs{j}},\alpha^{\bs{i}\vee}_1)}
q^{-(\alpha^{\bs{j}}_{k - 1}, \alpha^{\bs{j}}_1 + \alpha^{\bs{j}}_2 + \cdots + \alpha^{\bs{j}}_{k - 2})}\cl{T}_{i_1}^{-1}E_{\bs{j}, k-1}\cl{T}_{i_1}v\\
&= (-1)^{(\alpha_{k}^{\bs{i}},\alpha^{\bs{i}\vee}_1)}
q^{-(\alpha^{\bs{i}}_{k}, \alpha^{\bs{i}}_2 + \alpha^{\bs{i}}_3 + \cdots + \alpha^{\bs{i}}_{k - 1})}\cl{T}_{i_1}^{-2}E_{\bs{i}, k}\cl{T}_{i_1}^2v.
\end{align*}
On the other hand, we have
$\cl{T}_{i_1}^2v = (-1)^{(\wt{v}, \alpha^{\bs{i}\vee}_1)}q^{(\wt{v},\alpha^{\bs{i}}_1)}v$.
Moreover the Levend\"{o}rskii-Soibelman relation implies $E_{\bs{i},k}v$ is also a highest weight vector with respect to $U_q(\fsl_{j_1})$, hence we also have
\begin{align*}
 \cl{T}_{j_1}^{-2}E_{\bs{i},k}v = (-1)^{(\wt{v} + \alpha^{\bs{i}}_k,\alpha^{\bs{i}\vee}_1)}q^{-(\wt{v} + \alpha^{\bs{i}}_k,\alpha^{\bs{i}}_1)}E_{\bs{i},k}v.
\end{align*}
We can see the statement from these facts.
\end{proof}

Let $0_S^+$ be the character on $2Q^+$ determined uniquely by $0_S^+(2\alpha) = 1$ for $\alpha \in \roots_S^+$ and $0_S^+(2\alpha) = 0$
for $\alpha \in \roots^+\setminus\roots_S^+$. In the following, $U_{q,0_S^+}(\k)$ is denoted by $U_{q,S}(\k)$ and so on.

The following is essential to calculate the sesquiliner form on $M_{S,\chi}(\lbd)$:

\begin{lemm} \label{lemm:hermitian form for the degenerate case}
 Let $\map{P}{U_{q,S}(\g)}{U_{q,S}(\frl_S)}$ be the projection
along the triangular decomposition $U_{q,S}(\g)\cong U_{q,S}(\fru_S^-)\tensor U_{q,S}(\frl_S)\tensor U_{q,S}(\fru_S^+)$. Then $\{\aF^{\Lambda}\}_{\Lambda \le N - N_S}$
is an orthogonal family with respect to the sesquilinear map
$(x,y)\longmapsto P(x^*y)$.
Moreover we have
\begin{align*}
 P((\aF_{\bs{i}}^{\Lambda})^*\aF_{\bs{i}}^{\Lambda}) =
\prod_{k = 1}^N  (-1)^{\Lambda_k}q^{-(\Lambda_k\alpha^{\bs{i}}_k, \alpha^{\bs{i}}_1 + \cdots \alpha^{\bs{i}}_{k - 1})}\frac{q_{\alpha^{\bs{i}}_k}^{-\Lambda_k(\Lambda_k - 1)}[\Lambda_k]_{q_{\alpha^{\bs{i}}_k}}!}{(q_{\alpha^{\bs{i}}_k} - q_{\alpha^{\bs{i}}_k}^{-1})^{\Lambda_k}}.
\end{align*}
\end{lemm}
\begin{proof}
Note the following commutation relations:
\begin{align*}
 \aE_{\bs{i},k}\aF_{\bs{i},l} &= q^{x}\aF_{\bs{i},l}\aE_{\bs{i},k}
\quad (x \in \Z,\, 1\le k < l \le N - N_S),\\
 \aE_{\bs{i},k}\aF_{\bs{i},k}^n - q_{\alpha^{\bs{i}}_k}^{-2n}\aF_{\bs{i},k}^n\aE_{\bs{i},k} &= -\frac{q_{\alpha^{\bs{i}}_k}^{1 - n}[n]_{q_{\alpha^{\bs{i}}_k}}}{q_{\alpha^{\bs{i}}_k} - q_{\alpha^{\bs{i}}_k}^{-1}}\aF_{\bs{i},k}^{n - 1}.
\end{align*}
To see the first relation, we use the following relation in $U_q(\g)$ (see \cite[Eq. (4)]{MR4837934}. Note the difference in notation):
\begin{align*}
 [E_{\bs{i},k},F_{\bs{i},l}K_{\alpha^{\bs{i}}_l}]_{q^x} \in \mathrm{span}\{F_{\bs{i}}^{\Lambda^-}K_{\Lambda^-\cdot\alpha^{\bs{i}}}K_{2\mu}E_{\bs{i}}^{\Lambda^+}\mid \substack{\Lambda^+ < k < l < \Lambda^-,\\ \mu = \alpha^{\bs{i}}_k - \Lambda^+\cdot\alpha^{\bs{i}} = \alpha^{\bs{i}}_l - \Lambda^-\cdot\alpha^{\bs{i}} \in Q^+}\},
\end{align*}
where $x$ is an integer. Since $\{\alpha^{\bs{i}}_k\}_k$ is in a convex order, we see that $\mu$ in the above linear span is not contained in $Q_S^+$.
Especially $\chi_{2\mu} = 0$ holds, and we obtain the first relation. The second relation also follows from $\chi(2\alpha^{\bs{i}}_k) = 0$ for $1 \le k \le N - N_S$.

Now take $\Lambda, \Gamma$ and let $k, l$ be the minimum numbers such that $\lambda_k \neq 0$ and $\gamma_l \neq 0$. Since $P$ is \star-preserving, we may assume $k \le l$.
Then we can ignore the terms in the above sum with $\Lambda' \neq \delta_k$ in the computation of $P((\aF_{\bs{j}}^{\Lambda})^*\aF_{\bs{j}}^{\Gamma})$ since such a term is contained in
the left ideal generated by $U_{q,S}(\fru_S^+)$ by the first relation above. By the same reason we can see $k = l$ and $\lambda_k \le \gamma_k$.
Now we take the ajoint again. Then the argument above implies that
$\gamma_k \le \lambda_k$, hence we have $\lambda_k = \gamma_k$ combining the other inequality.
Moreover we can see that $P((\aF_{\bs{j}}^{\Lambda})^*\aF_{\bs{j}}^{\Gamma}) \neq 0$ implies that $P((\aF_{\bs{j}}^{\Lambda'})^*\aF_{\bs{j}}^{\Gamma'}) \neq 0$ where $\Lambda' = \Lambda_{k <}$ and $\Gamma' = \Gamma_{k < }$.
Now we can repeat this argument to obtain the orthogonality.
\end{proof}

\begin{proof}[Proof of Proposition \ref{prop:unitarizability}]
Fix $V \in \rep_q^{\fin} K_S$ and consider a sesquilinear form on $\dqpind{V}$ such that
$\ip{1\tensor v,1\tensor v'}_{\chi} = \ip{v,v'}$.

Since $\dqpind{V}$ has a basis $\{\aF^{\Lambda}\tensor v_i\}_{\Lambda,i}$,
we can identify $\dqpind{V}$ for all $\chi$ such that $\chi_{2\alpha} \le 0$ on $\roots^+\setminus\roots_S^+$. Since $U_{q,\chi}(\k)$ is
a continuous family with respect to $\chi$, we obtain a continuous
family of sesquilinear forms on a sigle vector space. Moreover, each sesquilinear form is non-degenerate.
Since the set is connected, positive-definiteness for some
$\ip{\cdot,\cdot}_{\chi}$ implies positive-definiteness for all $\chi$.
Since $\ip{\cdot,\cdot}_{0_S^+}$ is positive definite by Lemma \ref{lemm:hermitian form for the degenerate case}, this completes the proof.
\end{proof}

Now we see the following no-go theorem on \emph{noncommutative}
compact full flag manifolds.

\begin{thrm} \label{prop:characterization of unitarizability}
 For $\phi \in X_{\ssorb}(\C)$, $\pintO{\phi}$
is unitarizable if and only if $\phi \in X_{\sflag}^{\mathrm{quot}}$.
\end{thrm}
\begin{proof}
If $\pintO{\phi}$ is unitarizable, it must be semisimple since
each space of morphisms is finite dimensional. This implies
$\phi$ must belongs to $X_{\ssorb}^{\circ}(\C)$.
Then Proposition \ref{prop:non-unitarizability}
implies $\phi$ must belongs to $X_{\sflag}^{\quot}$.

On the other hand, if $\phi$ belongs to $X_{\sflag}^{\quot}$
and satisfies $\phi_{\alpha} \neq 1$ for all $\alpha \in \roots^+\setminus\roots_S^+$,
the corresponding
element $\chi_{\phi} \in X_{\roots}(\C)$ is a character on $2Q^+$
satisfying $\chi_{2\alpha} \le 0$ for all $\alpha \in\roots^+$.
Hence Proposition \ref{prop:unitarizability}
implies unitarizability. The other cases reduces to this case
by Proposition \ref{prop:twisted catO}.
\end{proof}

\section{Classification theorems for $\ssorba$ and $\sflaga$}
\label{sect:classification}
In this section, we classify
semisimple actions of $\ssorba$-type and actions of $\sflaga$-type.

\subsection{Generating morphisms in $\rep_q^{\fin}\SL_n$ and their relations}

At first we recall some concrete construction in $\rep_q^{\fin} \SL_n$.
We use the constructions in \cite{MR3263166} with modifications arising from
the difference of convension. Namely the coproduct of
$U_q(\fsl_n)$ in \cite{MR3263166} is the opposite of ours. Hence we need to
reverse the order of tensor factors.

We identify $\h_{\R}^*$ with
$\{x \in \R^n\mid x_1 + x_2 + \cdots + x_n = 0\} \cong \R^n/\R(1,1,\dots,1)$. Let $(e_i)_{i = 1}^n$ be the standard basis of $\R^n$.
Then we have $\roots = \{e_i - e_j\mid i\neq j\}$ and $Q = \Z^n \cap \h_{\R}^*$. The usual positive system is given by
$\roots^+ = \{e_i - e_j\mid i < j\}$ and $\simples = \{\eps_i\}_{i = 1}^{n - 1}$ is given by
$\eps_i := e_i - e_{i + 1}$. For later convenience, a subset of $\simples$
is denoted by $I$ in this section. Then the corresponding Levi subgroups
are denoted by $\SL_I$ and $\SU(I)$.

Let $(\varpi_i)_i$ be the fundamental weights such that $\ip{\varpi_i, \eps_j} = \delta_{ij}$.
Namely $\varpi_i$ is given by
\begin{align*}
 \varpi_i = [e_1 + e_2 + \cdots + e_i] = \frac{n - i}{n}\sum_{j = 1}^ie_j - \frac{i}{n}\sum_{j = i + 1}^n e_j.
\end{align*}

Now we consider the following representation $\Lambda_q^1$ with a
basis $(x_i)_{i = 1}^n$:
\begin{align*}
 E_ix_j =
\begin{cases}
 x_{j - 1} & (j = i + 1) \\
 0 & (j \neq i + 1),
\end{cases}\quad
 K_{\lbd}x_j = q^{\lbd_j}x_j,\quad
 F_ix_j =
\begin{cases}
 x_{j + 1} & (j = i) \\
 0 & (j \neq i).
\end{cases}
\end{align*}
Then $\wt x_i = [e_i]$, which implies that this representation is
the irreducible representation with highest weight $\varpi_1$.

To obtain the other fundamental representations, we consider
the quotient of the tensor algebra $T(\Lambda_q^1)$ with the
following relations:
\begin{align*}
 x_i^2 &= 0 \quad (1 \le i \le n),\quad \\
x_jx_i + qx_ix_j &= 0 \quad (1 \le i < j \le n).
\end{align*}
This is called the \emph{quantum exterior algebra},
which is also a representation of $U_q(\fsl_n)$.
It has direct summands $(\Lambda_q^i)_i$ which are the images
of $(\Lambda_q^1)^{\tensor i}$. It is known that each $\Lambda_q^i$
is the irreducible representation with highest weight $\varpi_i$. In these
representations, the image of $v_1\tensor v_2 \tensor \cdots \tensor v_i$
is denoted by $v_1\wedge_q v_2\wedge_q \cdots \wedge_q v_i$.
Then we define $x_S$ for $S \subset \{1,2,\dots,n\}$ as
\begin{align*}
 x_S := x_{i_1}\wedge_q x_{i_2}\wedge_q \cdots \wedge_q x_{i_k}
\end{align*}
with $S = \{i_1,i_2,\dots,i_k\},\, i_1 < i_2 < \cdots < i_k$.

There are some distinguished morphisms in $\rep_q^{\fin} \SL_n$.
Since the quantum exterior algebra is $U_{q}(\fsl_n)$-algebra,
the multiplication is a $U_q(\fsl_n)$-homomorphism. In particular it
restricts to a morphism $\map{M_{k,l}}{\Lambda_q^k\tensor \Lambda_q^l}{\Lambda_q^{k + l}}$:
\begin{align*}
 M_{k,l}(x_T\tensor  x_S) =
\begin{cases}
(-q)^{\ell(S,T)} x_{S\cup T} & (S\cap T = \emptyset), \\
 0 & (S\cap T \neq \emptyset),
\end{cases}
\end{align*}
where $\ell(S,T) = \abs{\{(i,j) \in S\times T\mid i < j\}}$.

Similarly we also have a morphism $\map{M_{k,l}'}{\Lambda_q^{k + l}}{\Lambda_q^k\tensor \Lambda_q^l}$:
\begin{align*}
 M_{k,l}'(x_S) = (-1)^{kl}\sum_{T\subset S} (-q)^{-\ell(S\setminus T,T)}x_{S\setminus T}\tensor x_T.
\end{align*}
Additionally we also have evaluations and coevaluations:
\begin{align*}
 &\map{\eps^+_i}{(\Lambda_q^i)^*\tensor \Lambda_q^i}{k},\quad \eps^+_i(f\tensor v) = f(v),\\
&\map{\eta^+_i}{k}{(\Lambda_q^i)^*\tensor \Lambda_q^i},\quad \eta^+_i(1) = \sum_{i} e^i \tensor K_{-2\rho}e_i,\\
 &\map{\eps^-_i}{\Lambda_q^i\tensor (\Lambda_q^i)^*}{k},\quad \eps^-_i(v\tensor f) = f(K_{2\rho}v),\\
&\map{\eta^-_i}{k}{\Lambda_q^i \tensor (\Lambda_q^i)^*},\quad \eta^-_i(1) = \sum_{i} e_i \tensor e^i,
\end{align*}
where $V^*$ is regarded as $U_q(\fsl_n)$-module by $(xf)(v) := f(S(x)v)$
for any $V \in \rep_q^{\fin} \SL_n$.

Note $\Lambda_q^n \cong k$ via the morphism given by
$x_1\wedge_q x_2\wedge_q \cdots \wedge_q x_n \longmapsto q^{n(n + 1)/4}$.
Hence we regard $M_{k,n - k}$ and $M_{k,n - k}'$ as
morphisms between $\Lambda_q^{k}\tensor \Lambda_q^{n - k}$ and $k$.

\begin{prop} \label{prop:generating morphisms}
The tensor category $\rep_q^{\fin} \SL(n)$ is generated by
$\{\Lambda_q^i\}_i$, $\{(\Lambda_q^i)^*\}_i$ and $\{M_{k,l}\}_{k,l}$
$\{M_{k,l}'\}_{k,l}$, $\{\eps_i^{\pm}\}_{i,\pm}$, $\{\eta_{i}^{\pm}\}_{i,\pm}$ as an idempotent complete $k$-linear tensor category.
\end{prop}

The following relations can be verified by comparing our construction
with the construction in \cite{MR3263166}, noting that
their $M_{k,l}$ is our $M_{l,k}$ and their $M'_{k,l}$ is our $M'_{l,k}$.

\begin{align}
(\eps_k^-\tensor\id)\circ(\id\tensor \eta^+_k) &= \id, \label{eq:conjugation equation 1}\\
(\id\tensor \eps_k^+)\circ(\eta_k^-\tensor\id) &=\id, \label{eq:conjugation equation 2}\\
(M_{k,n-k}\tensor \id)\circ (\id\tensor \eta^-_{n - k})
&= (-1)^{k(n - k)} (\id\tensor M_{n - k,k})\circ (\eta^+_{n - k}\tensor \id), \label{eq:rotation}\\
 M_{k,l + m}\circ (\id\tensor M_{l,m}) &= M_{k + l,m}\circ (M_{k,l}\tensor \id), \label{eq:associativity for m}\\
 (\id\tensor M'_{l,m})\circ M'_{k,l + m} &= (M'_{k,l}\tensor \id)\circ M'_{k + l,m}, \label{eq:associativity for m'}\\
 M_{k,l}\circ M'_{k,l} &= \qbin{k + l}{k}{q}\id, \label{eq:bubble popping}\\
 (M_{n - k,k}\tensor \id)\circ(\id\tensor M'_{k,l}) &= (-1)^{l(n - l)}(\id\tensor M_{n - k - l, k + l})\circ (M'_{k,n - k - l}\tensor \id). \label{eq:flipping n}
\end{align}
We also have the following relations, called the square switch relations.
\begin{align}
(\id&\tensor M_{r,k - s})\circ (M'_{l + s - r,r}\circ M_{l,s}\tensor \id)\circ (\id\tensor M'_{s, k - s}) \notag \\
&= \sum_t \qbin{k - l + r - s}{t}{q} (M_{l - r + t, s - t}\tensor \id)\circ (\id\tensor M'_{s - t, k - s + r}\circ M_{r - t, k}) \label{eq:square switch1}\\
&\hspace{90mm}\circ (M'_{l - r + t, r - t}\tensor \id). \notag\\
 (M_{k-s,r}&\tensor \id)\circ (\id\tensor M'_{r,l + s - r}\circ M_{s,l})\circ (M'_{k - s,s}\tensor\id) \notag \\
&= \sum_t \qbin{k - l + r - s}{t}{q} (\id\tensor M_{s-t,l - r + t})\circ (M'_{k - s + r,s-t}\circ M_{k,r - t}\tensor\id) \label{eq:square switch2}\\
&\hspace{90mm}\circ (\id\tensor M'_{r-t,l - r + t}). \notag
\end{align}
Next we would like to take the \cstar-structure into account.

\begin{lemm}
For $1 \le k \le n$, $\Lambda_q^k$ is a unitary representation of $U_q(\fsl_n)$ with respect to the following inner product:
\begin{align*}
 \ip{x_S,x_T} = \delta_{S,T} q^{\sum S},
\end{align*}
where $\sum S$ is the sum of all elements of $S$.
\end{lemm}
\begin{proof}
By induction on $k$. The case of $k = 1$ follows from direct calculation.
Assume the statement holds for $k$. Then we can embed $\Lambda_q^{k + 1}$
into $\Lambda_q^k\tensor \Lambda_q^1$ by $M_{k,1}'$. Fix $S = \{i_1,i_2,\dots, i_{k + 1}\}$ with $i_1 < i_2 < \cdots < i_{k + 1}$. Then we have
\begin{align*}
 M_{k,1}'(x_S) = (-1)^{k}\sum_{l = 1}^{k + 1} (-q)^{-(l - 1)} x_{S_l}\tensor x_{i_l},
\end{align*}
where $S_l = \{i_1,i_2,\cdots, i_{l - 1}, i_{l + 1},\cdots, i_{k + 1}\}$.
Now consider the inner product on $\Lambda_q^k\tensor \Lambda_q^1$.
It induces an inner product on $\Lambda_q^{k + 1}$ which is compatible
with the action of $U_q(\fsu(n))$. More concretely the square of $\nor{x_S}$ is calculated as follows:
\begin{align*}
 \nor{x_S}^2 = \nor{M_{k,1}'(x_S)}^2 = \sum_{l = 1}^{k + 1}q^{-2(l - 1)}\nor{x_{S_l}}^2\nor{x_{i_l}}^2
= q^{-k(k + 1)}q^{\sum S}.
\end{align*}
Then, by rescalling the inner product, we can see the statement.
\end{proof}

In the rest of this paper, each $\Lambda_q^k$ is regarded as
a unitary representation of $U_q(\fsu(n))$. Then we can consider
the adjoint of $M_{k,l}, M_{k,l}', \eps_i^{\pm}, \eta_i^{\pm}$.
It is not difficult to see the following relations:
\begin{align*}
 M_{k,l}^* = q^{kl}M_{k,l}',\quad
(\eps_i^{\pm})^* = \eta_i^{\pm}
\end{align*}

To deal with a semisimple action of $\ssorba$-type, we briefly recall
the fusion rule of $\rep_q^{\fin} \SL_I$. The simple objects are parametrized by
\begin{align*}
 P^{+_I} = \{\lbd = [m] \in \Z^n/\Z(1,1,\dots,1)\mid m_i \ge m_j \text{ for } \eps_i - \eps_j \in \roots_I^+\}.
\end{align*}
Then the tensor product $\Lambda_q^k \tensor L_{\lbd}^I$
has the following irreducible decomposition:
\begin{align*}
 \Lambda_q^k \tensor L_{\lbd}^I \cong \bigoplus_{\stackrel{\abs{S}=k}{[e_S] + \lbd \in P^{+_I}}} L_{[e_S] + \lbd}^I.
\end{align*}
In the rest of this paper, we fix an inclusion from $L_{\lbd + [e_S]}$
to $\Lambda_q^k \tensor L_{\lbd}^I$, which is denoted by $x_{S,\lbd}^I$.

\subsection{Classification theorems}
The goal of this section is the following.
\begin{thrm} \label{thrm:algebraic classification}
Let $\cl{M}$ be a semisimple action of $\ssorba$-type. Then there is a unique
$\phi \in X_{\ssorba}^{\circ}(k)$ such that
$\cl{M} \cong \cl{O}_{q,\phi}^{\frp_I,\mathrm{int}}$.
\end{thrm}

It is convenient to consider the associator picture.
Actually we focus on some invariant coefficients and show that
they are complete invariants. In the following, we consider
$\{x_S\}_{\abs{S} = k}$ as a basis of each irreducible representation $\Lambda_q^k$.

Let $\Phi$ be an associator and consider
the following map:
\begin{align*}
 \Lambda_q^k\tensor (\Lambda_q^l\tensor L^I_{\lbd})
\xlongrightarrow{\Phi}
 (\Lambda_q^k\tensor \Lambda_q^l)\tensor L^I_{\lbd}
\xlongrightarrow{M_{k,l}}
 \Lambda_q^{k + l}\tensor L^I_{\lbd}.
\end{align*}

Then we obtain the matrix coefficient
$m_{S,T}(\Phi;\lbd) \in k$, which satisfies
\begin{align*}
 M_{k,l}\circ \Phi\circ(\id\tensor x_{T,\lbd}^I)x_{S,\lbd + [e_T]}^I = m_{S,T}(\Phi;\lbd)x_{S\cup T,\lbd}^I.
\end{align*}

In a similar way we also obtain the following scalars from $M'_{k,l},\eps^{\pm}_k, \eta^{\pm}_k$ respectively:
\begin{align*}
 m_{S,T}'(\Phi;\lbd),\quad \eps_S^{\pm}(\Phi;\lbd),\quad
\eta_S^{\pm}(\Phi;\lbd).
\end{align*}
For $b = \{b^{\pm}_S(\lbd)\}_{\pm, S,\lbd}$, we define the perturbation of these scalars by $b$ as follows:
\begin{align*}
 m_{S,T}(\Phi;\lbd)_b &:= b^+_{S\cup T}(\lbd)^{-1}m_{S,T}(\Phi;\lbd)b^+_S([e_T] + \lbd)b^+_T(\lbd),\\
 m'_{S,T}(\Phi;\lbd)_b &:= b^+_S([e_T] + \lbd)^{-1}b^+_T(\lbd)^{-1}m'_{S,T}(\Phi;\lbd)b^+_{S\cup T}(\lbd),\\
 \eps_S^{\pm}(\Phi;\lbd)_b &:= \eps_S^{\pm}(\Phi;\lbd)b^{\mp}_S([e_S] + \lbd)b^{\pm}_S(\lbd),\\
 \eta_S^{\pm}(\Phi;\lbd)_b &:= b^{\mp}_S([e_S] + \lbd)^{-1}b^{\pm}_S(\lbd)^{-1}\eta_S^{\pm}(\Phi;\lbd).
\end{align*}

\begin{lemm} \label{lemm:cell system}
Let $\Phi$ and $\Phi'$ be associators.
The following are equivalent:
\begin{enumerate}
 \item There is an equivalence $\rep_{q,\Phi}^{\fin} \SL_I \cong \rep_{q,\Phi'}^{\fin} \SL_I$ of semisimple
$\ssorba$-type action.
 \item For some $b$, $(m_{S,T}(\Phi'), m'_{S,T}(\Phi'), \eps_S^{\pm}(\Phi'), \eta_S^{\pm}(\Phi'))_{S,T}$ is the perturbation of
$(m_{S,T}(\Phi), m'_{S,T}(\Phi), \eps_S^{\pm}(\Phi), \eta_S^{\pm}(\Phi))_{S,T}$ by $b$.
\end{enumerate}
\end{lemm}
\begin{proof}
If $\rep_{q,\Phi}^{\fin} \SL_I \cong \rep_{q,\Phi'}^{\fin} \SL_I$,
we can take an equivalence
$\map{(\id,b)}{\rep_{q,\Phi}^{\fin}\SL_I}{\rep_{q,\Phi'}^{\fin} \SL_I}$.
Consider a component $\map{b}{\Lambda_q^k\tensor L_{\lbd}^I}{\Lambda_q^k\tensor L_{\lbd}^I}$. Since this map perserves the irreducible decomposition, it naturally defines a scalar $b^+_S(\lbd)$
for any subset $S \subset \{1,2,\dots,n\}$. Similarly $b^-_S(\lbd)$
is also defined by replacing $\Lambda_q^k$ with $(\Lambda_q^k)^*$.
Then it is not difficut to see that $b = (b_S^{\pm}(\lbd))_{S,\lbd}$
satisfies the condition (ii).

To see the converse,
let $w = k_1k_2\dots k_l$ be a finite word of
$\{\pm 1, \pm 2, \dots, \pm(n - 1)\}$.
Set $\Lambda_q^w := \Lambda_q^{k_1}\tensor \Lambda_q^{k_2}\tensor \cdots \tensor \Lambda_q^{k_l}$, where $\Lambda_q^k = (\Lambda_q^{\abs{k}})^*$ when
$k < 0$. Then, by induction on $l$, we have a family
$\{\map{b_w}{\Lambda_q^w\tensor \tend}{\Lambda_q^w\tensor \tend}\}_w$
of natural transformaions satisfying the following conditions:
\begin{enumerate}
 \item If $w = k_1$, $b_w$ is the natural transformation
canonically induced by $b^{\pm}_{S}$ with $\abs{S} = \abs{k_1}$.
 \item For any finite words $w$ and $w'$, the following diagram is commutative:
\begin{align*}
 \xymatrix{
\Lambda_q^{ww'}\tensor \tend \ar[d]_-{b_{ww',\tend}} \ar[r]^-{\Phi'^{-1}} & \Lambda_q^w \tensor (\Lambda_q^{w'}\tensor \tend) \ar[d]^-{b_{w,\Lambda_q^{w'}\tensor \tend}\circ (\id\tensor b_{w',\tend})} \\
\Lambda_q^{ww'}\tensor \tend \ar[r]_-{\Phi^{-1}} & \Lambda_q^w \tensor (\Lambda_q^{w'}\tensor \tend).
}
\end{align*}
\end{enumerate}
Moreover we can check the naturarity of $\{b_w\}_w$ with respect to
$w$, which means commutativity of the following diagram
for any morphism $\map{T}{\Lambda_q^w}{\Lambda_q^{w'}}$:
\begin{align*}
 \xymatrix{
\Lambda_q^w\tensor \tend \ar[r]^-{T\tensor \id} \ar[d]_-{b_{w,\tend}} & \Lambda_q^{w'} \ar[d]^-{b_{w',\tend}}\tensor\tend \\
\Lambda_q^w\tensor \tend \ar[r]^-{T\tensor \id} & \Lambda_q^{w'}\tensor\tend.
}
\end{align*}
Actually, by Proposition \ref{prop:generating morphisms} and the property (ii) above, we may assume
$T$ is either of $M_{k,l}, M'_{k,l}, \eta^+_{k}, \eta^-_{k}$.
Here we consider the case of $T = M'_{k,l}$. Since
$m'_{S,T}(\Phi,\lbd)_b = m'_{S,T}(\Phi',\lbd)$ for all $S,T,\lbd$,
the following diagram commutes:
\begin{align*}
 \xymatrix{
\Lambda_q^{k + l}\tensor k_{\lbd} \ar[r]^-{M'_{k,l}\tensor \id} \ar[d]_-{b_{{k + l},k_{\lbd}}} & (\Lambda_q^k\tensor \Lambda_q^l)\tensor k_{\lbd} \ar[r]^-{\Phi'^{-1}} & \Lambda_q^k\tensor (\Lambda_q^l\tensor k_{\lambda}) \ar[d]^-{b_{k,\Lambda_q^l\tensor k_{\lbd}}(\id\tensor b_{l,k_{\lbd}})} \\
\Lambda_q^{k + l}\tensor k_{\lbd} \ar[r]_-{M'_{k,l}\tensor \id} & (\Lambda_q^k\tensor \Lambda_q^l)\tensor k_{\lbd} \ar[r]_{\Phi^{-1}} & \Lambda_q^k\tensor (\Lambda_q^l\tensor k_{\lambda}).
}
\end{align*}
On the other hand, the property (ii) implies that the right square of
the following diagram commutes:
\begin{align*}
 \xymatrix{
\Lambda_q^{k + l}\tensor k_{\lbd} \ar[r]^-{M'_{k,l}\tensor \id} \ar[d]_-{b_{{k + l},k_{\lbd}}} & (\Lambda_q^k\tensor \Lambda_q^l)\tensor k_{\lbd} \ar[r]^-{\Phi'^{-1}} \ar[d]^-{b_{kl,k_{\lbd}}} & \Lambda_q^k\tensor (\Lambda_q^l\tensor k_{\lambda}) \ar[d]^-{b_{k,\Lambda_q^l\tensor k_{\lbd}}(\id\tensor b_{l,k_{\lbd}})}\\
\Lambda_q^{k + l}\tensor k_{\lbd} \ar[r]_-{M'_{k,l}\tensor \id} & (\Lambda_q^k\tensor \Lambda_q^l)\tensor k_{\lbd} \ar[r]_{\Phi^{-1}} & \Lambda_q^k\tensor (\Lambda_q^l\tensor k_{\lambda}).
}
\end{align*}
Hence the left square also commutes.

The same argument works in the case of $T = M_{k,l},\eps_{S}^{\pm},\eta_S^{\pm}$.

Then we have an equivalence $\rep_{q,\Phi}^{\fin} \SL_I \cong \rep_{q,\Phi'}^{\fin} \SL_I$ which preserves the action of $\Lambda_q^w$ for all finite words $w$. By taking the idempotent completion, we see the condition (i).
\end{proof}

Since we have to consider equivalence classes with respect to
the perturbation, it is natural to look at a datum which does not
depend on the choice of the representative, i.e., invariant
coefficients. In the following we consider the invariant coefficient
$\gamma_{\Phi}(S,T;\lbd)$ arising from the projections onto
$\Lambda_q^{k + l} \subset \Lambda_q^k\tensor \Lambda_q^l$.
By definition we have
\begin{align*}
 \gamma_{\Phi}(S,T;\lbd) &= m_{S,T}(\Phi;\lbd)m'_{S,T}(\Phi;\lbd).
\end{align*}
Note that this scalar is defined if and only if all of $\lbd, \lbd + [e_T], \lbd + [e_{S\cup T}]$ are contained in $P^{+_I}$.
Then the family $\gamma_{\Phi} := \{\gamma_{\Phi}(S,T;\lbd)\}_{S,T,\lbd}$
only depends on the equivalence class of $\rep_{q,\Phi}^{\fin} \SL_I$. We also use
$\gamma_{\cl{M}}$ when $\cl{M}\cong \rep_{q,\Phi}^{\fin} \SL_I$.
Surprisingly, this datum contains enough information to distiguish
different semisimple actions of $\ssorba$-type.

\begin{lemm} \label{lemm:identification by scalar system}
Let $\cl{M},\cl{M}'$ be semisimple $\ssorba$-type actions.
If $\gamma_{\cl{M}} = \gamma_{\cl{M}'}$ holds,
$\cl{M}$ is equivalent to $\cl{M}'$.
\end{lemm}

In the rest of the present paper, we substitute $\{j\}$ by $j$ for ease to read. For example, $S\cup j = S\cup\{j\}$. Similarly
we substitute $ij$ for $\{i,j\}$ and so on.

\begin{proof}
Take associators $\Phi$ and $\Phi'$ so that $\cl{M}\cong \rep_{q,\Phi}^{\fin} H$ and $\cl{M'}\cong \rep_{q,\Phi}^{\fin} H$.
 It suffices to show Lemma \ref{lemm:cell system} (ii)
for $\Phi$ and $\Phi'$.

Set $f(S,T;\lbd) := m_{S,T}(\Phi';\lbd)/m_{S,T}(\Phi;\lbd)$. 
Moreover, for a mutually disjoint family $\{S_i\}_{i = 1}^l$,
we define $f(S_1,S_2,\cdots, S_l;\lbd)$ recurrsively as follows:
\begin{align*}
 f(S_1,S_2,\cdots, S_l;\lbd) := f(S_1,S_2,\cdots, S_{l - 1}\cup S_l;\lbd) f(S_{l - 1},S_l;\lbd).
\end{align*}
Then (\ref{eq:associativity for m}) and
(\ref{eq:associativity for m'})
imply that $f$ satisfies a kind of associativity,
which is of the following form for example:
\begin{align*}
 f(S_1\cup S_2, S_3,S_4,\lbd)&f(S_1,S_2,[e_{S_3\cup S_4}] + \lbd)\\
& = f(S_1,S_2\cup S_3,S_4,\lbd)f(S_2,S_3,[e_{S_4}] + \lbd).
\end{align*}

Now we construct $b = \{b_S^{\pm}(\lbd)\}_{\pm, S, \lbd}$ so that
\begin{align} \label{eq:compatibility}
 f(S_1,S_2,\cdots,S_k,\lbd) = b_{S_1\cup S_2\cup \cdots \cup S_k}(\lbd)^{-1}\prod_{i = 1}^k b_{S_i}(\lbd + [e_{S_{i+1}\cup \cdots \cup S_k}])
\end{align}
holds for any possible $S_i$ and $\lbd$.

At first we set $b^+_1(\lbd) = 1$ for all $\lbd$. Actually we can use arbitrary value for $b_1^+$.

Next we construct $b_2^+$. For $\lbd \in P^{+_I}$ such that $m_1 - m_2 = 1$, we set $b_2^+(\lbd) = 1$. For other $\lbd$, we determine $b_2^+(\lbd)$
inductively on $\abs{m_1 - m_2 - 1}$ so that the following holds:
\begin{align*}
 \frac{f(1,2;\lbd)}{b_1^+(\lbd + [e_2])b_2^+(\lbd)} = \frac{f(2,1;\lbd)}{b_2^+(\lbd + [e_1])b_1^+(\lbd)}.
\end{align*}
Then we can define $b_{12}^+(\lbd)$ so that (\ref{eq:compatibility})
holds for $k = 2$ and $(S_1,S_2) = (1,2),(2,1)$.

Next we construct $b_3^+$. For $\lbd \in P^{+_I}$ such that $m_1 - m_2 = m_2 - m_3 = 1$, we set $b_3^+(\lbd) = 1$. Then we determine $b_3^+(\lbd)$ for $\lbd$ such that $m_1 - m_2 = 1$ inductively on $\abs{m_2 - m_3 - 1}$ so that the following holds:
\begin{align} \label{eq:12-3 comparison}
 \frac{f(12,3;\lbd)}{b_{12}^+(\lbd + [e_3])b_3^+(\lbd)} = \frac{f(3,12;\lbd)}{b_3^+(\lbd + [e_{12}])b_{12}^+(\lbd)}.
\end{align}
Now we determine $b_3^+(\lbd)$ inductively on $\abs{m_1 - m_2 - 1}$
so that the following holds:
\begin{align*}
 \frac{f(1,3;\lbd)}{b_1^+(\lbd + [e_3])b_3^+(\lbd)} = \frac{f(3,1;\lbd)}{b_3^+(\lbd + [e_1])b_1^+(\lbd)}.
\end{align*}
Again we can define $b_{13}^+$ so that (\ref{eq:compatibility}) holds.
To define $b_{23}^+$, we have to check the following:
\begin{align*}
 \frac{f(2,3;\lbd)}{b_2^+(\lbd + [e_3])b_3^+(\lbd)} = \frac{f(3,2;\lbd)}{b_3^+(\lbd + [e_2])b_2^+(\lbd)}.
\end{align*}
It suffices to show the equality after multiplying a non-zero scalar,
like the following:
\begin{align} \label{eq:23 comparison}
 \frac{f(1,2,3;\lbd)}{b_1(\lbd + [e_{23}])b_2^+(\lbd + [e_3])b_3^+(\lbd)} = \frac{f(1,3,2;\lbd)}{b_1(\lbd + [e_{23}])b_3^+(\lbd + [e_2])b_2^+(\lbd)}.
\end{align}
Both sides are transformed as follows:
\begin{align*}
 \frac{f(1,2,3;\lbd)}{b_1(\lbd + [e_{23}])b_2^+(\lbd + [e_3])b_3^+(\lbd)}
&= \frac{f(2,1,3;\lbd)}{b_2(\lbd + [e_{13}])b_1^+(\lbd + [e_3])b_3^+(\lbd)} \\
&= \frac{f(2,3,1;\lbd)}{b_2(\lbd + [e_{13}])b_3^+(\lbd + [e_1])b_1^+(\lbd)}, \\
\frac{f(1,3,2;\lbd)}{b_1(\lbd + [e_{23}])b_3^+(\lbd + [e_2])b_2^+(\lbd)}
&= \frac{f(3,1,2;\lbd)}{b_3(\lbd + [e_{12}])b_1^+(\lbd + [e_2])b_2^+(\lbd)} \\
&= \frac{f(3,2,1;\lbd)}{b_3(\lbd + [e_{12}])b_2^+(\lbd + [e_1])b_1^+(\lbd)}.
\end{align*}
Then the last terms are again non-zero scalar multiples of both sides
of (\ref{eq:23 comparison}) in which $\lbd$ is replaced by $\lbd + [e_1]$.
Hence it suffices to show the equality when $m_1 - m_2 = 1$. By looking at the transformation above again, this case can be checked by using
(\ref{eq:12-3 comparison}).

Iterating this argument, we can obtain $b_S^+(\lbd)$ for all possible $S$ and $\lbd$ satisfying (\ref{eq:compatibility}).

Using this $b^+$,
we can see that $m_{S,T}(\Phi';\lbd) = m_{S,T}(\Phi';\lbd)_b$ for all $S,T$ and $\lbd$. Then $\gamma_{\Phi} = \gamma_{\Phi'}$ implies
$m'_{S,T}(\Phi';\lbd) = m'_{S,T}(\Phi;\lbd)_b$.
We take $b_S^-$ so that $\eps^+_S(\Phi';\lbd) = \eps^+_S(\Phi;\eps)_b$. Then we see $\eta^{\pm}_S(\Phi';\lbd) = \eta^{\pm}_S(\Phi;\lbd)_b$
as follows.
By (\ref{eq:flipping n}), we have $m_{S,S^c}(\Phi;\lbd - [e_{S^c}])m'_{S^c,S}(\Phi;\lbd) = 1$.
Hence
\begin{align}
 \eps_S^+(\Phi;\lbd)\eta_S^+(\Phi;\lbd) = \eps_S^+(\Phi;\lbd)m_{S,S^c}(\Phi;\lbd - [e_{S^c}])m'_{S^c,S}(\Phi;\lbd)\eta_S^+(\Phi;\lbd). \label{eq:intermediate}
\end{align}
On the other hand, by (\ref{eq:rotation}), we also have
\begin{align*}
 m_{S,S^c}(\Phi;\lbd - [e_{S^c}])\eta_S^+(\Phi;\lbd)
= (-1)^{\abs{S}(n - \abs{S})}m_{S^c,S}(\Phi;\lbd)\eta_S^-(\Phi;\lbd - [e_{S^c}]).
\end{align*}
Hence the RHS of (\ref{eq:intermediate}) is equal to
\begin{align*}
(-1)^{\abs{S}(n - \abs{S})} \eps_S^+(\Phi;\lbd)m_{S^c,S}(\Phi;\lbd)m'_{S^c,S}(\Phi;\lbd)&\eta_S^-(\Phi;\lbd - [e_{S_c}])\\
&= (-1)^{\abs{S}(n - \abs{S})}\gamma_{\Phi}(S^c,S;\lbd).
\end{align*}
This proves $\eta^-_S(\Phi';\lbd - [e_{S^c}]) = \eta^-_S(\Phi;\lbd - [e_{S^c}])_b$. The case of $+$ is similar.

By (\ref{eq:conjugation equation 1}) and (\ref{eq:conjugation equation 2}), we see
$\eps^-_S(\Phi';\lbd) = \eps^-_S(\Phi;\eps)_b$ and
$\eta^+_S(\Phi';\lbd) = \eta^+_S(\Phi;\lbd)_b$.

Hence we see Lemma \ref{lemm:cell system} (ii).
\end{proof}

Next we consider the following generalization of $\gamma_{\cl{M}}$.
For $S,T \subset [1,n]$ and $\lbd \in P^{+_I}$,
we say that $(S,T,\lbd)$ is $I$-admissible when $[e_T] + \lbd$
and $[e_{S\cup T}] + \lbd$ are in $P^{+_I}$.

\begin{defn} \label{defn:scalar system}
A \emph{scalar system of $\ssorba$-type} is a scalar family
\begin{align*}
\gamma = (\gamma(S,T;\lbd))_{(S,T,\lbd)\colon\text{admissible}}
\end{align*}
satisfying the following conditions:
\begin{enumerate}
 \item $\gamma(S,T,\lbd)\gamma(T,S,\lbd - [e_S]) = 1$.
 \item $\gamma(S\cup i, j;\lbd) \gamma(j,S;\lbd - [e_S]) + \gamma(S\cup j, i;\lbd) \gamma(i,S;\lbd - [e_S]) = [2]_q$.
 \item $\gamma(S,U;\lbd)\gamma(U,S\cup T;\lbd - [e_{S\cup T}])
= \gamma(S\cup U, T;\lbd - [e_T])\gamma(T,S;\lbd - [e_{S\cup T}])$.
 \item $\gamma(S,T;\lbd + [e_U])\gamma(S\cup T, U;\lbd) = \gamma(S, T\cup U;\lbd)\gamma(T,U;\lbd)$.
 \item $\sum_{i \in S,\,\lbd + [e_{S\setminus i}] \in P^{+_I}} \gamma(i,S\setminus i;\lbd) = [\abs{S}]_q$.
 \item $\sum_{i \in S,\,\lbd + [e_i] \in P^{+_I}} \gamma(S\setminus i,i;\lbd) = [\abs{S}]_q$.
\end{enumerate}
\end{defn}
\begin{rema}
In the rest of this section, we prove various facts on this scalar system.
In every statement, all triples $(S,T,\lbd)$ necessary for the statment itself are assumed to be admissible. On the other hand,
all triples necessary for justifying its proof are not assumed to be admissible. Hence, to make it mathematically rigorous, we need to confirm
that such triples are admissible. In fact we have done it, but we omit the detail.
\end{rema}
The following lemma is repeatedly used to check that
$\gamma_{\cl{M}}$ for a semisimple action $\cl{M}$ of $\ssorba$-type is actually a scalar system of $\ssorba$-type.

\begin{lemm} \label{lemm:sum of weights}
Let $S,S',T,T'$ be subsets of $\{1,2,\dots,n\}$ such that
$S \subsetneq S', T\subsetneq T'$ and $\abs{S} + \abs{S'} = \abs{T} + \abs{T'}$.
If $[e_S] + [e_{S'}] = [e_T] + [e_{T'}]$ holds,
$S = T$ and $S' = T'$.
\end{lemm}

\begin{prop}
For a semisimple action $\cl{M}$ of $\ssorba$-type,
$\gamma_{\cl{M}}$ is
a scalar system of $\ssorba$-type.
\end{prop}
\begin{proof}
We may assume $\cl{M} = \rep_{q,\Phi}^{\fin} \SL_I$
for some associator $\Phi$.

The relation (iv) follows from (\ref{eq:associativity for m})
and (\ref{eq:associativity for m'}).

The relation (v) and (vi) follow from (\ref{eq:bubble popping}).

The other relations follow from the square switch relation (\ref{eq:square switch1}).
To obtain the relation (i), we consider the following special
case of the relation:
\begin{align*}
 (\id&\tensor M_{l,k})\circ (M'_{k,l}\circ M_{k,l}\tensor \id)\circ (\id\tensor M'_{l, k}) \notag \\
&= \sum_t \qbin{l}{t}{q} (M_{k - l + t, l - t}\tensor \id)\circ (\id\tensor M'_{l - t, k + l}\circ M_{l - t, k + l})\circ (M'_{k - l + t, l - t}\tensor \id). \notag
\end{align*}
On $\Lambda_q^k\tensor (\Lambda_q^{k + l}\tensor L^I_{\lbd - [e_S]})$,
we have
\begin{align*}
&(\id\tensor (M_{l,k}\tensor \id)\circ\Phi)\circ \Phi^{-1}\circ(M'_{k,l}\circ M_{k,l}\tensor \id \tensor \id)\circ \Phi\circ (\id\tensor \Phi^{-1}\circ (M'_{l, k}\tensor \id)) \notag \\
&= \sum_t \qbin{l}{t}{q} (M_{k - l + t, l - t}\tensor \id\tensor\id)\circ \Phi \circ (\id\tensor \Phi^{-1}\circ (M'_{l - t, k + l}\circ M_{l - t, k + l}\tensor \id)\circ \Phi) \\
&\hspace{90mm}\circ \Phi^{-1}\circ (M'_{k - l + t, l - t}\tensor \id\tensor \id). \notag
\end{align*}
Then take disjoint subsets $S,T \subset \{1,2,\dots,n\}$ such that $\abs{S} = k$
and $\abs{T} = l$ and consider the composition of $(\id\tensor x_{S\cup T,\lbd - [e_S]}^I)x^I_{S,\lbd + [e_T]}$ and the LHS map.
Since $(M_{k,l}\tensor \id)(\id\tensor x^I_{T',\lbd})x^I_{S,\lbd + [e_{T'}]} = 0$ if $S\cap T' \neq \emptyset$,
we can see
\begin{align*}
 (M_{k,l}\tensor \id\tensor \id)\circ&\Phi\circ (\id\tensor \Phi^{-1}\circ (M'_{l, k}\tensor \id))(\id\tensor x_{S\cup T,\lbd - [e_S]}^I)x^I_{S,\lbd + [e_T]}\\
&= m_{S,T}(\Phi;\lbd)m'_{T,S}(\Phi;\lbd - [e_S])(\id\tensor x^I_{S,\lbd - [e_S]})x^I_{S\cup T,\lbd}.
\end{align*}
Hence we can see that the composition is
\begin{align*}
 \gamma_{\Phi}(S,T;\lbd)\gamma_{\Phi}(T,S;\lbd - [e_S])(\id\tensor x_{S\cup T,\lbd - [e_S]}^I)x^I_{S,\lbd + [e_T]}.
\end{align*}
On the other hand, we have
\begin{align*}
 (\id\tensor (M_{l - t, k + l}\tensor \id)\circ \Phi)&\circ \Phi^{-1}\circ (M'_{k - l + t, l - t}\tensor \id\tensor \id)(\id\tensor x_{S\cup T,\lbd - [e_S]}^I)x^I_{S,\lbd + [e_T]}\\
&\in \mathrm{span}\left\{(\id\tensor x^I_{B,\lbd - [e_S]})x^I_{A,\lbd - [e_S] + [e_B]}\middle| \substack{A\subset S, S\cup T\subset B\\ [e_A] + [e_B] = [e_S] + [e_{S\cup T}] \\ \abs{A} = k - l + t \\ \abs{B} = k + 2l - t} \right\}.
\end{align*}
If the image is non-zero, the condition on $A, B$ implies that $A =S, B = S\cup T$
and $t = 0$ by Lemma \ref{lemm:sum of weights}. Hence this composition
is
$(\id\tensor x_{S\cup T,\lbd - [e_S]}^I)x^I_{S,\lbd + [e_T]}$, which implies the relation (i).

To obtain the relation (ii), we consider the following special case
of the square switch relation:
\begin{align*}
 (\id&\tensor M_{1,k-1})\circ (M'_{k + 2,1}\circ M_{k + 2,1}\tensor \id)\circ (\id\tensor M'_{1,k-1}) \notag \\
&= (M_{k+1, 1}\tensor \id)\circ (\id\tensor M'_{1, k}\circ M_{1, k})\circ (M'_{k+1, 1}\tensor \id)  - [2]_q\id\notag
\end{align*}
Then take a subset $S \subset \{1,2,\dots,n\}$ such that $\abs{S} = k$
and also take $i,j \in S^c$. Then, by looking at the image
of $(\id\tensor x^I_{S,\lbd})x^I_{S\cup ij,\lbd + [e_S]}$, a similar argument shows the relation (ii).

To obtain the relation (iii), assume $t := \abs{T} \ge \abs{U} =: u$
at first. Also set $s = \abs{S}$. We consider the following
special case of the square switch relation (\ref{eq:square switch1}):
\begin{align*}
(\id\tensor M_{t,s})\circ (M'_{s + u,t}&\circ M_{s + u,t}\tensor \id)\circ (\id\tensor M'_{t,s}) \notag \\
&= (M_{s, u}\tensor \id)\circ (\id\tensor M'_{u, s + t}\circ M_{u, s + t})\circ (M'_{s, u}\tensor \id). \notag
\end{align*}
Then, by looking at
$(\id\tensor x^I_{S\cup T,\lbd - [e_{S\cup T}]})x^I_{S\cup U,\lbd}$,
a similar argument shows the relation (iii). Use (\ref{eq:square switch2}) for the case $\abs{T} \le \abs{U}$.
\end{proof}

In the rest of this section,
we classify a scalar systems of $\ssorba$-type.

Surprisingly we only have to
look at $\gamma(i,j;\lbd)$ by the following lemma.

\begin{lemm}
 Let $\gamma,\gamma'$ be scalar systems of $\ssorba$-type. If $\gamma(i,j;\lbd) = \gamma'(i,j;\lbd)$ for all admissible $(i,j,\lbd)$,
 we have $\gamma = \gamma'$
\end{lemm}
\begin{proof}
We show $\gamma(S,T;\lbd) = \gamma'(S,T;\lbd)$ for every admissible $(S,T,\lbd)$ by induction on $\abs{S} + \abs{T}$.

If $\abs{S} + \abs{T} = 2$, there is nothing to prove.

Next consider the induction step, hence we assume that $\gamma(S,T;\lbd) = \gamma'(S,T;\lbd)$ when $\abs{S} + \abs{T} < N$.
Take a subset $A$ with $\abs{A} = N$ arbitrary. Let $k$ be the maximum
element of $A$. If $i,A\setminus i,\lbd)$ is admissible for $i \in A$, we have
\begin{align*}
 \gamma(i,A\setminus i;\lbd)\gamma(k,A\setminus ik;\lbd)
&\stackrel{\mathrm{(iv)}}{=} \gamma(ik,A\setminus ik;\lbd)\gamma(i,k;\lbd + [e_{A\setminus ik}])\\
 \gamma(k,A\setminus k;\lbd)\gamma(i,A\setminus ik;\lbd)
&\stackrel{\mathrm{(iv)}}{=} \gamma(ik,A\setminus ik;\lbd)\gamma(k,i;\lbd + [e_{A\setminus ik}]).
\end{align*}
Combining these equalities with the induction hypothesis, we see that
$\gamma(i,A\setminus i;\lbd)/\gamma(k,A\setminus k;\lbd)$
only depends on $\{\gamma(i,j;\lbd)\}_{i,j,\lbd}$. This implies that
$\gamma(i,A\setminus i;\lbd) = \gamma'(i,A\setminus i;\lbd)$
when the same equation holds for $i = k$.
Now we have
\begin{align*}
 \sum_{\substack{i \in A,\\ \lbd + [e_{A\setminus i}] \in P^{+_I}}} \gamma(i,A\setminus i;\lbd) \stackrel{\mathrm{(v)}}{=} [\abs{A}]_q.
\end{align*}
By the discussion above, this equality can be interpreted as a linear equation on $\gamma(k,A\setminus k;\lbd)$ with coefficients depending only on $\{\gamma(i,j;\lbd)\}_{i,j,\lbd}$.
Therefore
we see $\gamma(k,A\setminus k;\lbd) = \gamma'(k,A\setminus k;\lbd)$.
Since $A$ is arbitrary, we see the statement for $\abs{S} = 1$
and $\abs{T} = N - 1$.

Now consider $(S,T)$ with $\abs{S} + \abs{T} = N$ arbitrary.
For $i = \min S$, $(S\setminus i, T\cup i,\lbd)$ is admissible.
Moreover we have
\begin{align*}
 \gamma(S,T;\lbd)\gamma(S\setminus i,i;\lbd + [e_T])
\stackrel{\mathrm{(iv)}}{=} \gamma(S\setminus i, T\cup i;\lbd)\gamma(i,T;\lbd)
\end{align*}
Hence, by the induction hypothesis, $\gamma(S,T;\lbd) = \gamma'(S,T;\lbd)$ if and only if $\gamma(S\setminus i,T\cup i;\lbd) = \gamma'(S\setminus i, T\cup i;\lbd)$. Now we use the induction on $\abs{S}$,
since the case $\abs{S} = 1$ is already proven.
\end{proof}

For $(i,j,\lbd)$ such that $\lbd,\lbd + [e_i],\lbd + [e_{ij}] \in P^{+_I}$, define $\chi_{ij}(\lbd) \in \P^1(k)$ as follows:
\begin{align*}
 \gamma(i,j;\lbd) = \frac{[m_i - m_j - (i - j) - 1;\chi_{ij}(\lbd)]}{[m_i - m_j - (i - j);\chi_{ij}(\lbd)]}.
\end{align*}

What we need to see is:

\begin{prop} \label{prop:classification of scalar system}
 Let $\gamma$ be a scalar system of $\ssorba$-type.
Then $\chi_{ij}(\lbd)$ does not depend on $\lbd$,
and $\chi = \{\chi_{ij}\}_{ij}$ is an element of $X_{\ssorba}^{\circ}(k)$
\end{prop}

By (i) and (v) in Definition \ref{defn:scalar system}, it is not difficult see that $\chi_{ij}(\lbd)$ does not depend on $m_i$ and $m_j$.

\begin{lemm} \label{lemm:fundamental decomposition of 2-1 type}
The following equalities hold if the six triples are admissible:
\begin{align*}
 \gamma(ij,k;\lbd)\gamma(k,i;\lbd - [e_i]) &= \gamma(j,k;\lbd),\\
 \gamma(ki,j;\lbd)\gamma(j,i;\lbd - [e_i]) &= \gamma(k,j;\lbd).
\end{align*}
\end{lemm}
\begin{proof}
At first we have
\begin{align*}
(\gamma(ij,k;\lbd)\gamma(k,i;\lbd - [e_i]))^2
&\stackrel{\mathrm{(iii)}}{=} \gamma(ij,k;\lbd)\gamma(k,i;\lbd - [e_i])\gamma(i,j;\lbd + [e_k])\gamma(j,ki;\lbd- [e_i])\\
&\stackrel{\mathrm{(iv)}}{=} \gamma(i,jk;\lbd)\gamma(j,k;\lbd)^2\gamma(jk,i;\lbd- [e_i])\\
&\stackrel{\mathrm{(i)}}{=} \gamma(j,k;\lbd)^2.
\end{align*}
By exchaging $j$ and $k$, we also have
\begin{align*}
 (\gamma(ki,j;\lbd)\gamma(j,i;\lbd - [e_i]))^2 = \gamma(k,j;\lbd)^2
\end{align*}
On the other hand, we have
\begin{gather*}
 \gamma(ij,k;\lbd)\gamma(k,i;\lbd - [e_i]) + \gamma(ki,j;\lbd)\gamma(j,i;\lbd - [e_i]) \stackrel{\mathrm{(ii)}}{=} [2]_q,\\
 \gamma(j,k;\lbd) + \gamma(k,j;\lbd) \stackrel{\mathrm{(v)}}{=} [2]_q.
\end{gather*}
Since $[2]_q \neq 0$, this implies the statement.
\end{proof}




\begin{lemm} \label{lemm:invariance under homogeneous translation}
Let $S$ and $U$ be disjoint subsets of $[1,n]$ satisfying $S \cup U = [1,k]$ for some $k$. Then, for all $l \ge k$,
we have $\gamma(U,S;\lbd) = \gamma(U,S;\lbd + [e_{[1,l]}])$.
\end{lemm}
\begin{proof}
Set $T = [k + 1,l]$. Then we have
\begin{align*}
 \gamma(S,T;\lbd + [e_{S\cup U}])&\gamma(T,S\cup U;\lbd) \\
&\stackrel{\mathrm{(iii)}}{=} \gamma(S\cup T,U;\lbd + [e_S])\gamma(U,S;\lbd),\\
 \gamma(S,T;\lbd + [e_{S\cup U}])&\gamma(U,S\cup T;\lbd + [e_{S\cup U}])\\
&\stackrel{\mathrm{(iv)}}{=} \gamma(U\cup S,T;\lbd + [e_{S\cup U}])\gamma(U,S;\lbd + [e_{S\cup U \cup T}]).
\end{align*}
Since we have $\gamma(U\cup S,T;\lbd + [e_{S\cup U}]) = \gamma(T,S\cup U;\lbd)^{-1}$ and $\gamma(U,S\cup T;\lbd + [e_{S\cup U}]) = \gamma(S\cup T,U;\lbd + [e_S])^{-1}$, we obtain the statement.
\end{proof}

\begin{proof}[Proof of Proposition \ref{prop:classification of scalar system}]
At first we see that $\chi_{ij}(\lbd)$ does not depend on $\lbd$ for all $i,j$ by induction on $\max\{i,j\}$. If $\max\{i,j\} = 2$, the statement follows from Lemma \ref{lemm:invariance under homogeneous translation}.

Now assume the statement for the case $\max\{i,j\} < k$ and consider the case $\max\{i,j\} = k$.

\par
\noindent\textbf{Claim.} For $k' \ge k$, we have $\chi_{ij}(\lbd) = \chi_{ij}(\lbd + [e_{[1,k']}])$.
\par

Assume $j = k$. Then we have
\begin{align*}
\gamma([k-1]\setminus i, ik;\lbd)\gamma(i,k;\lbd) \stackrel{\mathrm{(iv)}}{=} \gamma([k-1]\setminus i,i;\lbd + [e_k])\gamma([k-1],k;\lbd).
\end{align*}
Then Lemma \ref{lemm:invariance under homogeneous translation} implies the claim.

Next assume $i = k$. If $j \neq k - 1$, we have $\chi_{kj}(\lbd) = \chi_{jk}(\lbd)^{-1}$. Hence the previous case
can be applied. Even if $j = k-1$, the same argument can be used if $k -1 \not\in I$. If $k - 1 \in I$, we
have $\chi_{k,k-1}(\lbd) = 1$ since $\gamma(k,k-1;\lbd) = [2]_q$ when $m_{k - 1} = m_k$.

\par
Using this claim, we can assume $m_k \gg m_{k+1}$ in the following discussion. Especially we have $\lbd - [e_k] \in P^{+_I}$.

\par
\noindent\textbf{Claim.} For $h \neq i$, we have $\chi_{ki}(\lbd) = \chi_{ki}(\lbd + [e_h])$ and $\chi_{ik}(\lbd) = \chi_{ik}(\lbd + [e_h])$.
\par

By induction on $h$. By induction hypothesis, we can assume $m_{h - 1} \gg m_h$. Especially we have $\lbd + [e_h],\,\lbd + 2[e_h] \in P^{+_I}$.
We also assume $i \le k - 2$ or $i = k - 1 \not \in I$ since $i = k - 1 \in I$ implies $\chi_{ki}(\lbd) = \chi_{ik}(\lbd) = 1$ for all possible $\lbd$.
Then Lemma \ref{lemm:fundamental decomposition of 2-1 type} implies
\begin{align*}
 \gamma(hk,i;\lbd) &= \gamma(h,i;\lbd)\gamma(k,i;\lbd),\\
 \gamma(ik,h;\lbd) &= \gamma(i,h;\lbd)\gamma(k,h;\lbd).
\end{align*}
Similar argument shows
\begin{align*}
 \gamma(h,i;\lbd + [e_{hk}])\gamma(i,hk;\lbd) &= \gamma(i,k;\lbd + [e_h]),\\
 \gamma(h,k;\lbd + [e_{hi}])\gamma(k,hi;\lbd) &= \gamma(k,i;\lbd + [e_h]).
\end{align*}
Then the first equation shows 
\begin{align*}
 \gamma(ih,k;\lbd) \stackrel{\mathrm{(iv)}}{=} \frac{\gamma(i,hk;\lbd)\gamma(h,k;\lbd)}{\gamma(i,h;\lbd + [e_k])} \stackrel{\mathrm{(i)}}{=} \gamma(i,k;\lbd + [e_h])\gamma(h,k;\lbd).
\end{align*}

On the other hand, for $\lbd + [e_h]$, we have
\begin{align*}
 \gamma(hk,i;\lbd + [e_h]) &= \gamma(h,i;\lbd + [e_h])\gamma(k,i;\lbd + [e_h]),\\
 \gamma(ik,h;\lbd + [e_h]) &= \gamma(i,h;\lbd + [e_h])\gamma(k,h;\lbd + [e_h]),\\
 \gamma(ih,k;\lbd + [e_h]) &= \gamma(i,k;\lbd + [e_h])\gamma(h,k;\lbd + [e_h]).
\end{align*}
by Lemma \ref{lemm:fundamental decomposition of 2-1 type} since $(\lbd + [e_h]) - [e_h] = \lbd \in P^{+_I}$. Moreover we have
\begin{gather*}
 \gamma(ih,k;\lbd) + \gamma(ki,h;\lbd) + \gamma(hk,i;\lbd) \stackrel{\mathrm{(vi)}}{=} [3]_q,\\
 \gamma(ih,k;\lbd + [e_h]) + \gamma(ki,h;\lbd + [e_h]) + \gamma(hk,i;\lbd + [e_h]) \stackrel{\mathrm{(vi)}}{=} [3]_q.
\end{gather*}
Since $\gamma(h,i;\lbd) = \gamma(h,i;\lbd + [e_h])$ and so on, we see $\gamma(hk,i;\lbd) = \gamma(hk,i;\lbd + [e_h])$,
which implies $\gamma(k,i;\lbd) = \gamma(k,i;\lbd + [e_h])$.

The remaining part is to show $\chi = \{\chi_{ij}\}_{ij} \in X^{\circ}_{\ssorb}(k)$. This follows from the following general fact:
If $\phi_{ij} = [a_{ij}\colon b_{ij}] \in \P^1(k)$ satisfies $\phi_{ij}^{-1} = \phi_{ji}$ and
\begin{align*}
\sum_{i = 1}^3 \frac{[m_{i + 1} - m_i - 1;\phi_{i+1,i}]}{[m_{i + 1} - m_i;\phi_{i+1,i}]}\frac{[m_{i - 1} - m_i - 1;\phi_{i-1,i}]}{[m_{i - 1} - m_i;\phi_{i-1,i}]}
= [3]_q,
\end{align*}
then $a_{12}a_{23}a_{31} = b_{12}b_{23}b_{31}$.
\end{proof}

Finally we prove Theorem \ref{thrm:algebraic classification}.
\begin{proof}[Proof of Theorem \ref{thrm:algebraic classification}]
At first we show that $\gamma_{\chi} := \gamma_{\cl{O}_{q,\chi}^{\frp_I,\mathrm{int}}}$ corresponds to $\chi$.

Take $w \in W\cong \mathfrak{S}_n$ so that $w(1) = j, w(2) = i$.
If $\lbd \in P^{+_I}$ is sufficiently dominant, we have
\begin{align*}
 \gamma_{ij}(\lbd)
&= [2]_q c_{\varpi_1,\varpi_1;w,\eps_1}(\lbd)\\
&= \frac{[(w^{-1}\cdot\lbd,\eps_1^{\vee}) + 2;\chi_{w(2\eps_1)}]_q}{[(w^{-1}\cdot\lbd,\eps_1^{\vee})+1;\chi_{w(2\eps_1)}]_q}\\
&= \frac{[m_j - m_i - (j - i) + 1;\chi_{ji}]_q}{[m_j - m_i - (j - i);\chi_{ji}]_q}\\
&= \frac{[m_i - m_j - (i - j) - 1;\chi_{ij}]_q}{[m_i - m_j - (i - j);\chi_{ij}]_q}.
\end{align*}
Now we see the statement by Lemma \ref{lemm:identification by scalar system} and Proposition \ref{prop:classification of scalar system}.
\end{proof}

As a corollary of Theorem \ref{thrm:algebraic classification}
and Theorem \ref{prop:characterization of unitarizability},
we also obtain a classification of actions of $\sflaga$-type.

\begin{coro} \label{coro:operator algebraic classification}
Let $\cl{M}$ be actions of $\sflaga$-type.
Then there is a unique $\phi \in X_{\sflaga}^{\quot}$
such that $\cl{M} \cong \cl{O}_{q,\phi}^{\frp_I,\mathrm{int}}$.
\end{coro}



\section{The non-quantum case}
\label{sect:non-quantum}
It is natural to expect results for the genuine groups $K$ and $G$
analogous to the results for the quantum groups $K_q$ and $G_q$.
Since $K$ and $G$ can be thought as quantizations of
the Poisson groups $K^{\zero}$ and $G^{\zero}$,
whose Poisson-Lie structures are trivial,
the parameter space for the classification in the algebraic setting
shoule relate with the space of Poisson $G^{\zero}$-structures. We have the following description for this space,
which is similar to Proposition \ref{prop:classification of equiv Poisson var}.
\begin{align*}
 X_{\ssorb,0}(k) := \left\{\phi = (\phi_{\alpha})_{\alpha \in \roots\setminus\roots_S} \in k^{\roots\setminus\roots_S}\middle|\substack{\phi_{-\alpha} = -\phi_{\alpha},\\ \phi_{\alpha}\phi_{\beta} = \phi_{\alpha + \beta}(\phi_{\alpha} + \phi_{\beta}),\\ \phi_{\alpha + \beta} = \phi_{\beta} \text{ for }\alpha \in \roots_S}\right\}.
\end{align*}
Then the parameter space might be the following subset
of $X_{\ssorb,0}(k)$:
\begin{align*}
 X_{\ssorb,0}^{\circ}(k) := \{\phi \in X_{\ssorb}(k)\mid \phi_{\alpha} \in k\setminus\{1/nd_{\alpha}\}_{n \in \Z\setminus\{0\}}\}.
\end{align*}
For the classification in the \cstar-algebraic setting,
the parameter space should be the space of Poisson
$K^{\zero}$-structures on $\sflag$ admitting a $0$-dimensional
symplectic leaf. Since $K^{\zero}$-equivariance is $K$-invariance,
the space consists of only the trivial Poisson structure.

Unlike the case of quantum groups, we do not have a unified approach to construct
semisimple actions of $\ssorb$-type
corresponding to all Poisson structures.
On the other hand, the construction
using the category $\cl{O}$ and the induction of actions still work.
For $\chi \in \h^*$ such that $\chi(\alpha) = 0$ for $\alpha \in \roots_S$, we define $\cl{O}_{\chi}^{\frp_S,\mathrm{int}}$ as the full
subcategory of the parabolic category $\cl{O}^{\frp_S}$ consisting of all modules
whose weights are contained in $\chi + P$. It carries
a canonical structure of left $\rep^{\fin} G$-module category.

The following can be seen using some fundamental
results on the category $\cl{O}$.

\begin{prop}
For $\chi \in \h^*$ such that $\chi|_{\roots_S} = 0$, the category $\cl{O}_{\chi}^{\frp_S,\mathrm{int}}$
is semisimple if and only if $\chi(\alpha) \not\in d_{\alpha}\Z$
for all $\alpha \in \roots\setminus\roots_S$. In this case, $\cl{O}_{\chi}^{\frp_S,\mathrm{int}}$
has a canonical structure of semisimple actions of $\ssorb$-type
given by the parabolic Verma modules with highest weights in $\chi + P$.
\end{prop}

It would be natural to regard $\cl{O}_{\chi}^{\frp_S,\mathrm{int}}$
as a semisimple action corresponding to $\chi^{-1} := \{\chi(\alpha)^{-1}\}_{\alpha \in \roots\setminus\roots_S} \in X_{\ssorb,0}^{\circ}$. For general $\phi \in X_{\ssorb,0}^{\circ}(k)$, note that $\roots_{\phi}:= \roots_S \cup \{\alpha \in \roots\setminus\roots_S \mid \phi_{\alpha} \neq 0\}$
forms a closed subsystem of $\roots$. Then it defines a subalgebra $\g' \subset \g$ containing $\frl_S$ as a Levi subalgebra. Moreover we have $\chi \in \h^*$ such that $\chi(\alpha) = \phi_{\alpha}^{-1}$ for $\alpha \in \roots_{\phi}\setminus\roots_S$ and $\chi|_{\roots_S} = 0$, which is not unique in general.
Then consider the shifted integral part $\cl{O}_{\g',\chi}^{\frp_S\cap \g',\mathrm{int}}$ of the parabolic category $\cl{O}_{\g'}^{\frp_S \cap \g'}$
for the subalgebra $\g'$. This also carries a natural structure of semisimple actions of $\ssorb$-type. Moreover
its equivalence class does not depend on the choice of $\chi$. This
action is denoted by $\cl{O}_{\phi}^{\frp_S,\mathrm{int}}$.
We can see that these actions are mutually inequivalent and gives
a family parametrized by $X_{\ssorb,0}^{\circ}(k)$.
On unitarizability, we have the following criteria:

\begin{prop}
 For $\phi \in X_{\ssorb,0}^{\circ}(\C)$, the category $\cl{O}_{\phi}^{\frp_S,\mathrm{int}}$ is unitarizable if and only if $\phi = 0$.
\end{prop}

We also have a classification result in the case of $G = \SL_n$.

\begin{thrm}
Any semisimple action of $\ssorba$-type is
equivalent to $\cl{O}_{\phi}^{\frp_S,\mathrm{int}}$ for unique $\phi \in X_{\ssorb,0}^{\circ}(k)$. Any action of $\sflaga$-type is equivalent to $\rep^{\fin} \SU(I)$.
\end{thrm}

\vspace{10pt}
\noindent
{\bf Acknowlegements.}
The author is grateful to Yasuyuki Kawahigashi for his invaluable supports. He also thanks Yuki Arano for fruitful discussion.
He also thanks Noriyuki Abe for sharing his note on Arkhipov's twisting functor.
He also thanks Kenny De Commer and staffs in Vriji Universiteit Brussel
for their support and hospitality
during the author's stay in Brussels.



\input{clsqfma-bbl.tex}


\end{document}